\documentclass[a4paper,11pt]{article}
\usepackage{hyperref}
\usepackage{mathtools,mathrsfs,amssymb,amsthm,array,booktabs,longtable}
\usepackage[margin=26mm]{geometry}
\usepackage{Article}
\usepackage{Symbols}
\usepackage[backend=biber,style=numeric,sorting=nyt,giveninits=true,url=false,eprint=false,isbn=false,maxbibnames=99]{biblatex}
\AtEveryBibitem{%
  \ifentrytype{article}{\clearfield{series}\iffieldundef{eid}{}{\clearfield{pages}}}{}%
  \iffieldundef{howpublished}{}{\clearfield{eprint}}%
}
\allowdisplaybreaks
\numberwithin{equation}{section}
\numberwithin{theorem}{section}
\hypersetup{hidelinks,pdftitle={Weighted configurations and the nonexistence of arithmetic ball quotient compactifications},pdfauthor={Ryunosuke Nakano and Hisato Matsukawa}}
\title{Weighted configurations and the nonexistence of arithmetic ball quotient compactifications}
\author{Ryunosuke Nakano \and Hisato Matsukawa}
\date{October 16, 2026}
\begin{document}
\maketitle
\begin{abstract}
We prove that, for every \(n\ge9\) and all positive rational weights smaller than one with total weight two, the GIT quotient of \(n\) ordered points on \(\mathbb P^1\) is not isomorphic to the Baily--Borel compactification of an arithmetic complex ball quotient.
The weights may lie on stability walls, and the lattice may have torsion; the quotient map may be branched along divisors.
No compatibility with a prescribed period map is assumed.
We study homomorphisms from the tangent bundle pulled back to complete curves into line bundles of negative degree and compare the resulting constraints with normal bundles of the images of balls fixed by finite reflection groups in the arithmetic lattice.
We describe these homomorphisms through coincident marked points and prove vanishing statements for their kernels.
The comparison, together with integrality of Chern classes and the invariant theory of finite reflection groups, gives the nonexistence theorem without a Fano assumption on the GIT quotient.
\end{abstract}

\noindent\textit{Use of AI.} The authors used OpenAI Codex as a discussion partner for mathematical arguments, to assist with revising proofs and their presentation, and to proofread the manuscript.
The authors take responsibility for the mathematics, computations, and references.

\section{Introduction}\label{si:introduction}
Let \(n\ge9\), and choose rational weights
\begin{equation}\label{split-weights}
 0<\mu_i<1,\qquad \sum_{i=1}^{n}\mu_i=2.
\end{equation}
For \(I\subset\{1,\ldots,n\}\), write \(\mu(I)=\sum_{i\in I}\mu_i\).
After clearing denominators, let
\[
 P_n=(\mathbb P^1)^n,\qquad X_\mu=P_n/\!/_{\mathcal O(m\mu_1,\ldots,m\mu_n)}\mathrm{SL}(2,\mathbb C).
\]
The labels are retained, and the quotient is independent of the integer \(m>0\) used to clear denominators.
A configuration is stable precisely when \(\mu(I)<1\) for every set \(I\) of labels whose points coincide.
We use \(\mathbb B^d\) for the complex ball of dimension \(d\), and the superscript \(\mathrm{BB}\) for the Baily--Borel compactification.
\begin{theorem}\label{si:main}\label{split-weighted-main}\label{split-nine-main}
For every \(n\ge9\) and every weight vector satisfying~\eqref{split-weights}, there is no arithmetic lattice \(\Gamma\subset\mathrm{PU}(1,n-3)\) with an algebraic isomorphism
\[
 X_\mu\simeq(\mathbb B^{n-3}/\Gamma)^{\mathrm{BB}}.
\]
\end{theorem}
We impose no compatibility with label permutations or with a specified period map.
Failure of an integrality condition for a Deligne--Mostow period map~\cite{DM} alone would not prove the theorem.
The weights may lie on stability walls.
The lattice may have torsion and need not be a normal or congruence subgroup of a prescribed group; the quotient map may be branched along divisors.
No Fano hypothesis is imposed on \(X_\mu\).

The lower bound $n\ge9$ is sharp for a theorem that applies to every weight vector satisfying~\eqref{split-weights}: for eight equal weights, \(X_\mu\) is the Baily--Borel compactification of an arithmetic ball quotient \cite{DM,HulekMaedaEight}.
This does not classify all weights with fewer than nine points.
The theorem also includes chambers whose quotient is projective space.
For example, choose a rational number $0<\varepsilon<1/(n-2)$ and put
\[
 \mu=\left(1-\varepsilon,\left(\frac{1+\varepsilon}{n-1}\right)^{n-1}\right).
\]
The heavy point cannot meet any light point, every proper subset of the light points may coincide, and all $n-1$ light points cannot coincide.
Place the heavy point at infinity.
The affine coordinates of the light points, modulo translations and nonzero scalar multiplication, give
\[
 X_\mu\simeq\mathbb P\bigl(\mathbb C^{n-1}/\mathbb C\mathbf1\bigr)\simeq\mathbb P^{n-3},
 \qquad \mathbf1=(1,\ldots,1).
\]
The unstable affine configurations are those in which all light points coincide, and the residual group is the affine group of the line.
Thus, for $d\ge6$, the theorem implies that $\mathbb P^d$ is not isomorphic to the Baily--Borel compactification of an arithmetic ball quotient.

The case of twelve equal weights answers Question~2.3 of Hulek--Kond\"o--Maeda \cite{HKM}.
The companion paper \cite{NakanoSevenConfigurations2026} proves the nonexistence theorem for seven equal weights by a comparison involving all fixed components.
It cites the propositions on tangent bundles and arithmetic groups proved here.
The proof of Theorem~\ref{si:main} does not use that result.

The universal family of projective lines with labeled sections determines homomorphisms from the tangent bundle to line bundles of negative degree and controls their kernels.
The finite reflection groups in the original arithmetic lattice fix balls whose normal directions are described by polynomial invariants.
The line bundles corresponding to these normal directions descend together with the coordinates at the cusps, according to transformation laws that include the finite characters and the multipliers of the central translations.
We compare these structures on complete curves obtained from quotients of fixed balls by their full normalizers.
Three reductions shorten the proof.
First, we use the degrees of the relative tangent line bundles along the labeled sections to obtain decompositions of effective curve classes and initial bounds for degrees of curves, without an external theorem on generators of the Mori cone.
Next, suppose that an irreducible factor of a finite stabilizer is the full generic stabilizer of its fixed ball, that this ball has dimension at least 2, and that every full interior stabilizer along that ball is the product of this factor and a group acting trivially on the normal space.
Then this factor has rank \(s\le2\), its basic invariants have equal degree, and the Baily--Borel compactification of the quotient of its fixed ball by the full normalizer is the GIT quotient on \(n-s\) labels obtained by merging \(s+1\) labels into one.
Strong induction from \(n=9,10\) excludes every such factor for \(n\ge11\).
Finally, after the stabilizers are of type \(A\), write \(N=n-2\) and
\begin{equation}\label{si:constant}
 b_N=\prod_{p\mid N}p^{\lceil2v_p(N)/3\rceil},\qquad C_N=\frac{2N^2}{b_N}.
\end{equation}
Throughout, \(L\) is the rational class of the automorphic line bundle of the original lattice, or the specified pullback of this class.
Let \(W=X_\nu\) be the smooth GIT quotient obtained by perturbing \(\mu\) to weights \(\nu\) in a nearby chamber, with the small projective morphism \(\pi:W\to X_\mu\) of Lemma~\ref{sg:resolution}.
Using integrality of Chern classes and complete curves obtained by moving marked points, we show that \(K_W+C_NL\) is nef even when \(W\) is isomorphic to projective space.
For \(N\ge11\),
\[
 C_N<\frac{N^2-4N+1}{2},
\]
contradicting the lower bound of Proposition~\ref{sc:type-a} for curves obtained from disks fixed by reflection groups of type \(A\).

The proof proceeds through the following steps.
The first row gives the common geometric input; the remaining rows indicate where the arguments depend on the dimension.
\begin{center}
\small
\begin{tabular}{@{}p{.15\linewidth}p{.79\linewidth}@{}}
\toprule
Range & Main steps\\
\midrule
All $n$ & Propositions~\ref{sg:negative}--\ref{sg:determinant} control homomorphisms from the tangent bundle to line bundles of negative degree and their kernels; Proposition~\ref{sc:curve-determinant} compares fixed components, retaining the full normalizer, inertia characters, and cusp widths.\\
$n=9$ & Theorems~\ref{sn:compact-terminal} and~\ref{sn:nonuniform-terminal} treat the initial case; the $D_5$ argument uses successive kernels, and the $G_{33}$ argument uses the theta calculation in Lemma~\ref{sa:g33-theta}.\\
$n=10$ & Theorem~\ref{st:ten-main} gives the second induction case by proving nefness of \(K_W\) plus smaller multiples of \(L\) and treating the exceptional stabilizers.\\
$n=11,12$ & Propositions~\ref{sh:eleven} and~\ref{sh:twelve} combine the induction with exceptional invariant calculations and the examination of every case in the finite ranges of Appendix~\ref{sf:finite}.\\
$n\ge13$ & Proposition~\ref{sh:large} reduces the remaining stabilizers to type $A$ and applies the strict inequality for \(C_N\) stated after~\eqref{si:constant}.\\
\bottomrule
\end{tabular}
\end{center}
After the induction and the stabilizer reductions, the argument of Proposition~\ref{sh:large} applies uniformly for $n\ge13$.
The numerical inequality itself already holds for $N\ge9$; for eleven and twelve points we exclude the stabilizers and prove the required nefness bounds separately before proceeding to the induction for thirteen or more points.
The appendices retain the local algebra, the arguments for fixed curves in projective space, and the finite ranges used in the proof.
The notation \(H=-K_W\) does not assert ampleness; in a calculation on projective space its hyperplane class is identified separately.
\section{Curves in GIT quotients, arithmetic groups, and cusps}
\label{sg:section}

Put \(d=n-3\), \(N=d+1=n-2\), and \(\mu(I)=\sum_{i\in I}\mu_i\).
Throughout this section \(n\ge6\).
For a weight vector \(\nu\), put \(\nu(I)=\sum_{i\in I}\nu_i\), and call \(I\) short or long if \(\nu(I)<1\) or \(\nu(I)>1\), respectively.
A configuration is \(\nu\)-stable if \(\nu(I)<1\) for every set \(I\) of labels whose points coincide.
Stable configurations have at least three distinct points, so the action of \(\mathrm{PGL}(2,\mathbb C)\) is free.
A pair \(\{i,j\}\) is allowable if \(\nu_i+\nu_j<1\).
At a wall \(\mu(I)=1\), the vertex is genuine if \(|I|,|I^c|\ge3\).

\subsection{The family with labeled sections and quotients of the tangent bundle onto line bundles of negative degree}

\begin{lemma}\label{sg:resolution}\label{ew-weighted-resolution}
There is a smooth GIT quotient \(W=X_\nu\), obtained by perturbing \(\mu\) to weights \(\nu\) in a nearby chamber, and a small projective morphism \(\pi:W\to X_\mu\) such that \(K_W\) is \(\pi\)-nef.
The original stable quotient \(X_\mu^{\mathrm s}\) is an open subset of \(W\).
At a genuine wall vertex the local model is
\[
 \operatorname{Cone}\bigl(\mathbb P^{|I|-2}\times
 \mathbb P^{|I^c|-2},\mathcal O(1,1)\bigr).
\]
The resolution retains the smaller factor \(\mathbb P^u\), where \(u=\min(|I|,|I^c|)-2\), and
\[
 \operatorname{codim}_W\operatorname{Exc}(\pi)
 \ge \lceil(n-2)/2\rceil.
\]
\end{lemma}
\begin{proof}
After fixing one marked point at each of \(0,\infty\), the transverse torus weights are \((1^{|I|-1},-1^{|I^c|-1})\).
Products of opposite coordinates generate its invariant ring, with the two by two minors as relations.
Keeping the smaller factor gives normal bundle \(\mathcal O(-1)^{\oplus(v+1)}\), \(v=\max(|I|,|I^c|)-2\), and degree against \(-K_W\) \(u-v\le0\).
All these choices are made simultaneously by
\[
 \nu=(1+\epsilon)\mu-\epsilon(2/n,\ldots,2/n)+\epsilon^2\eta,
 \qquad \sum_i\eta_i=0,
\]
with \(\eta\) generic and \(\epsilon>0\) sufficiently small.
Every \(I\) with \(\mu(I)=1\) and \(|I|<n/2\) becomes long; when \(|I|=n/2\), \(\eta\) decides which of \(I\) and \(I^c\) becomes long.
The exceptional codimension is \(d-u\).
In particular, it is at least four for \(n\ge9\).
\end{proof}

From the principal \(\mathrm{PGL}(2,\mathbb C)\)-bundle of stable configurations, we obtain the associated bundle \(\mathcal P\to W\) with fiber \(\mathbb P^1\) and labeled sections \(\sigma_i\); in each fiber, the values of the sections are the marked points of the configuration.
For a set \(I\) of labels, the sections labeled by \(I\) coincide over a subset of \(W\) if \(\sigma_i=\sigma_j\) over it for all \(i,j\in I\); we then call them coincident sections.
For a nonconstant map \(f:C\to W\) from a smooth projective connected curve, write the pullback of \(\mathcal P\) as \(\mathbb P_C(V)\), using the convention of quotients of dimension one.
A section supplies a divisor of relative degree one, so such a vector bundle \(V\) of rank 2 exists.
Write \(V\twoheadrightarrow L_i\) for the quotients defined by the sections \(\sigma_i\) and put \(M_i=L_i^2(\det V)^{-1}\), \(m_i=\deg M_i\), and \(a=-K_W\cdot f_*C\).
Thus \(M_i\) is the pullback by \(f\) of the line bundle \(\sigma_i^*T_{\mathcal P/W}\) on \(W\), the relative tangent line bundle along \(\sigma_i\); we use the same letter for this line bundle on \(W\), its restrictions and its pullbacks, and \(M_I\) denotes the common restriction of the \(M_i\), \(i\in I\), to a locus where the sections labeled by \(I\) coincide.
The infinitesimal action and the brackets of distinct sections give
\begin{equation}\label{sg:tangent}
\begin{gathered}
 0\longrightarrow\operatorname{End}_0(V)\longrightarrow\bigoplus_iM_i\longrightarrow f^*T_W\longrightarrow0,\\
 a=\sum_i m_i,\qquad m_i+m_j=2\deg D_{ij}\ge0\quad(\sigma_i\ne\sigma_j).
\end{gathered}
\end{equation}
Here \(D_{ij}\) is the divisor of the bracket.
If the sections labeled by a set \(I\) of labels coincide, the smooth locus \(Z_I^{\mathrm s}\) where they coincide satisfies
\begin{equation}\label{ew-git-coincidence}
 N_{Z_I^{\mathrm s}/W}
 =M_I\otimes_{\mathbb C}(\mathbb C^I/\mathbb C\mathbf1).
\end{equation}

\begin{proposition}\label{sg:negative}\label{ew-git-negative}
Suppose \(Q\) is a line bundle of negative degree and \(\operatorname{Hom}(f^*T_W,Q)\ne0\).
Every homomorphism \(f^*T_W\to Q\) factors through the stalkwise quotient
\[
 D=M^{\oplus(\ell-1)},\qquad \psi:f^*T_W\twoheadrightarrow D,\qquad
 \psi\,df=0.
\]
The \(\ell\ge2\) labels associated with this quotient form a set \(I\) of weight \(\nu(I)<1\) whose sections coincide over \(C\), and \(M=M_I\) has degree \(-e<0\).
For \(j\notin I\), write \(D_j\) for the divisor of the bracket of \(\sigma_j\) with this common section, and put \(\delta=\sum_{j\notin I}\deg D_j\).
Then
\begin{equation}\label{ew-git-degree}
 M_j=M^{-1}(2D_j),\quad
 a=(n-2\ell)e+2\delta,\quad
 \operatorname{Hom}(f^*T_W,Q)
 =H^0(C,QM^{-1})^{\oplus(\ell-1)}.
\end{equation}
In particular, \(\deg Q=-b<0\) implies \(e\ge b\).
If this Hom space has dimension one, then \(\ell=2\) and \(a\ge(n-4)b\).
\end{proposition}
\begin{proof}
By \eqref{sg:tangent}, two line bundles \(M_i\) of negative degree cannot belong to distinct sections.
Thus all line bundles \(M_i\) of negative degree arise from one section, and its brackets give the first two formulas.
For \(j\notin I\), all line bundles \(M_j\) have positive degree and map trivially to \(Q\).
Evaluation of \(\operatorname{End}_0(V)\) onto the common tangent line is surjective at every stalk, and its image in the \(I\) components is exactly the diagonal.
Quotienting those components by the diagonal proves the assertion and the Hom formula.
If there is only one such component it is removed entirely, so \(\ell\ge2\).
Coincident sections have identical derivatives, proving \(\psi\,df=0\).
\end{proof}

\begin{proposition}\label{sg:kernel}\label{ew-git-kernel}
Put \(F=f^*T_{Z_I^{\mathrm s}}\), let \(T_f\subset F\) be the saturated image of \(df\), and put \(\mathcal K=F/T_f\).
Then
\[
 f^*T_W=F\oplus D,\qquad
 N_f^{\mathrm{sat}}=\mathcal K\oplus D,\qquad
 \bigoplus_{j\notin I}M_j\twoheadrightarrow F\twoheadrightarrow\mathcal K .
\]
For every line bundle \(Q\) and every integer \(q\ge1\),
\[
 \deg Q<qe\ \Longrightarrow\
 \operatorname{Hom}(\textstyle\bigwedge^qF,Q)=
 \operatorname{Hom}(\textstyle\bigwedge^q\mathcal K,Q)=0.
\]
For \(C=\mathbb P^1\), all line summands of \(F\) have degree at least \(e\).
If a line bundle of degree less than \(e\) is a direct summand of \(f^*T_W\), or of \(N_f^{\mathrm{sat}}\), then it is isomorphic to \(M\).
If \(f\) is immersed, \(N_f^{\mathrm{sat}}\) is the normal bundle of \(f\).
\end{proposition}
\begin{proof}
Split the \(I\) components of \eqref{sg:tangent} into their diagonal and the subbundle with sum zero.
The infinitesimal action has zero projection to the subbundle with sum zero, so \(f^*T_W=F\oplus D\) and \(N_f^{\mathrm{sat}}=\mathcal K\oplus D\).
Stalkwise evaluation onto the diagonal removes that diagonal and gives the surjection \(\bigoplus_{j\notin I}M_j\twoheadrightarrow F\).
The exterior powers of \(\bigoplus_{j\notin I}M_j\) have summands of degrees at least \(qe\), proving the vanishing.
On \(\mathbb P^1\), this also says that \(F(-e)\) is globally generated.
For a line summand of degree less than \(e\), its retraction vanishes on \(F\), or on \(\mathcal K\), and factors through \(D\).
The resulting idempotent of \(D\), after tensoring by \(M^{-1}\), is a constant matrix; its image of rank 1 is trivial.
\end{proof}

\begin{proposition}
\label{sg:determinant}\label{ew-git-ranktwo}
Let \(\mathcal Q\) be a vector bundle of rank \(q\) on \(C\) with \(\deg\det\mathcal Q<0\), admitting a surjection from \(f^*T_W\) or \(N_f^{\mathrm{sat}}\).
Then the set \(I\) of labels of coincident sections and the bundles \(F\), \(\mathcal K\), and \(D\) in Proposition~\ref{sg:kernel} exist.
If the restriction to \(F\), respectively \(\mathcal K\), has generic rank \(k\), then
\begin{equation}
\label{ew-git-determinant}
\deg\det\mathcal Q\ge(2k-q)e\ge-qe.
\end{equation}
If \(\deg\det\mathcal Q<-(q-2)e\), the surjection factors through \(D\) and \(\ell\ge q+1\).
If \(\deg\det\mathcal Q=-qe\), then \(\mathcal Q=M^{\oplus q}\).
If \(f\) takes values in the stable locus of the original equal weight quotient with \(n=7,8\) and \(q=2\), we have \(\ell=3\), \(\mathcal Q=M^{\oplus2}\), and \(e=a/\varepsilon_n\), where \(\varepsilon_7=1\), \(\varepsilon_8=2\).
On that stable locus for \(n=6\), no such quotient of rank \(2\) exists.
\end{proposition}

\begin{proof}
The inequality \(\deg\det\mathcal Q<0\) forces at least two coincident sections with negative degrees \(m_i\): stalkwise evaluation removes a single such section.
Nonzero minors show that the determinant of the saturated image of \(F\), or \(\mathcal K\), has degree at least \(ke\); the quotient generated by \(D\) contributes at least \(-(q-k)e\).
This proves \eqref{ew-git-determinant} and the strict case.
Equality makes \(M^{-1}\mathcal Q\) globally generated with determinant of degree zero, hence trivial.
For \(q=2\), negativity gives \(k=0\) and \(\ell\ge3\).
Stability for the original equal weights forces \(\ell=3\) and \(D_j=0\) for \(n=7,8\), so \(D\twoheadrightarrow\mathcal Q\) is an isomorphism and \eqref{ew-git-degree} gives \(e\); \(n=6\) forbids \(\ell\ge3\).
\end{proof}

\subsection{Effective curve classes without a Mori theorem for moduli spaces}

For the equal weight quotients \(X_n\), with \(\mu_i=2/n\), put \(s_n=\lfloor(n-1)/2\rfloor\) and \(\varepsilon_n=n-2s_n\).

\begin{proposition}
\label{ew-git-geometry}
The smooth stable quotient \(X_n^{\mathrm s}\) is simply connected, and the descent \(H\) of \(\mathcal O_{P_n}(2,\ldots,2)\) is ample and equals \(-K\).
For odd \(n\), we have \(X_n=X_n^{\mathrm s}\) and \(\operatorname{Pic}(X_n)=\{b\in\mathbb Z^n:\sum_i b_i\in2\mathbb Z\}, \, H=(2,\ldots,2)\).
For even \(n=2h\), the descent \(J\) of \(\mathcal O(1,\ldots,1)\) is very ample, \(\operatorname{Pic}(X_n)=\mathbb ZJ\), and \(H=2J=-K_{X_n}\).
The smooth locus equals the stable locus; each of the \(\frac12\binom nh\) singularities is the cone over \(\mathbb P^{h-2}\times\mathbb P^{h-2}\) with polarization \(\mathcal O(1,1)\).
Its link is simply connected, and its two small resolutions are crepant.
For curves \(f:C\to X_n^{\mathrm s}\), Proposition~\ref{sg:negative} gives \(\ell\le s_n\), \(a\ge\varepsilon_ne\), and vanishing whenever \(\deg Q<-a/\varepsilon_n\).
\end{proposition}

\begin{proof}
General position and the homotopy sequence of the principal bundle give simple connectedness.
The character \((-1)^{\sum b_i}\) gives odd descent.
Even descent requires \(\sum_{i\in I}b_i=\sum_{i\notin I}b_i\) at every closed orbit, forcing constant multidegree; perfect matching brackets give very ampleness \cite[Theorem~2.3]{HMSV}.
The tangent sequence gives \(H=-K\).
Lemma~\ref{sg:resolution}, the primitive circle class \((1,1)\), and equal factor dimensions give the cones, links, and crepancy.
Stability and \eqref{ew-git-degree} give the curve bounds.
\end{proof}

\begin{proposition}
\label{ew-odd-nef}
Let \(n=2h+1\), and let \(E\) be the line bundle of multidegree \(b\) on \(X_n\).
Set \(q=\frac12\sum_i b_i\) and \(d_I=q-\sum_{i\in I}b_i\) for \(|I|=h\).
Then \(Z_I=\mathbb P^{h-1},\, M_I=\mathcal O(-1),\, H|_{Z_I}=\mathcal O(1),\, E|_{Z_I}=\mathcal O(d_I)\).
The conditions \(d_I\ge0\) for every \(I\), global generation of \(E\), and nefness of \(E\) are equivalent.
If \(\ell_I\) is a line in \(Z_I\), then \(\overline{\operatorname{NE}}(X_n)=\sum_I\mathbb R_{\ge0}\ell_I\).
Every ample integral \(E\) has \(E-H\) globally generated.
\end{proposition}

\begin{proof}
The affine quotient at infinity gives \(\mathbb P((\mathbb C^{h+1}/\mathbb C\mathbf1)^\vee)\), represented by \(\mathcal O\oplus\mathcal O(1)\), proving the restrictions.
Partitioning the other labels gives \(b_j=d_I+d_{I'}\ge0\).
List \(b_i\) copies of label \(i\) consecutively in a circle of length \(2q\), and pair opposite positions.
The inequalities prevent a label from being paired with itself, so bracket products generate \(E\); \(q=0\) is immediate.
Pass to real classes by clearing denominators and limits.
Since \(H\ell_I=1\), cone duality gives the Mori cone; ampleness gives \(d_I\ge1\).
\end{proof}

\begin{proposition}
\label{ew-moving-line}
For equal weights, if \(|I|=\ell\) and \(n>2(\ell+1)\), there is a complete embedded \(\mathbb P^1\subset Z_I^{\mathrm s}\) on which \(M_I\) is trivial.
Thus no finite surjection onto \(Z_I^{\mathrm s}\) can make \(M_I^{-1}\) ample.
\end{proposition}

\begin{proof}
Moving one outside label with the others distinct and fixed gives an embedded stable \(\mathbb P^1\), since \(\ell+1<n/2\).
Three fixed labels trivialize the common tangent line, whose pullback to every complete curve above this \(\mathbb P^1\) still has degree zero.
\end{proof}

\begin{proposition}\label{sg:mori}
The cone \(\overline{\mathrm{NE}}(W)\) is generated by coordinate curves \(\gamma_i\), obtained by moving the marked point labeled \(i\) when every pair of labels containing \(i\) is allowable, and by lines \(\beta_J\) in the closed loci of positive dimension
\[
 Z_J=\mathbb P^{n-|J|-2},\qquad
 \nu(J)<1,\qquad \nu(J)+\nu_j>1\quad(j\notin J).
\]
For these curves,
\[
 -K_W\cdot\gamma_i=2,\qquad
 -K_W\cdot\beta_J=n-2|J|,\qquad
 N_{Z_J/W}=\mathcal O(-1)^{\oplus(|J|-1)}.
\]
More precisely, if all \(m_i\ge0\), then
\[
 [f_*C]=\tfrac12\sum_i m_i[\gamma_i].
\]
If \(I\) is the set of labels \(i\) with \(m_i<0\), possibly of size one, there is a complete rational curve \(\beta_I\) on which the degrees \(m_i\) are \(-1\) on \(I\) and \(1\) elsewhere, and
\begin{equation}\label{sg:curve-classes}
 [f_*C]=e[\beta_I]+\sum_{j\notin I}\deg D_j[\gamma_j],\qquad
 [\beta_I]=[\beta_J]+\sum_{j\in J\setminus I}[\gamma_j]
\end{equation}
for every maximal short set \(J\supset I\).
Every coordinate curve with positive coefficient in these formulas exists.
\end{proposition}
\begin{proof}
The \(M_i\) generate \(\operatorname{Pic}(W)_{\mathbb Q}\).
Pull a rational section of a line bundle \(B\) to the principal bundle and close its invariant divisor in \((\mathbb P^1)^n\).
Its bundle is \(\mathcal O(b_1,\ldots,b_n)\); its linearization agrees with the standard \(\mathrm{SL}(2,\mathbb C)\)-linearization because \(\mathrm{SL}(2,\mathbb C)\) has no characters.
Squaring and descending gives \(B^2=\bigotimes_iM_i^{b_i}\).
Thus the degrees \(m_i\) determine numerical curve classes.

If \(m_i>0\) and all degrees are nonnegative, each distinct section has a bracket of positive degree with \(i\), so the corresponding pair actually collides and must be allowable.
The labels of two coincident sections also form an allowable pair by stability.
Moving the marked point labeled \(i\), keeping the other marked points general, therefore gives \(\gamma_i\), with degrees \(2\delta_{ij}\).
This proves the first formula.

In the negative case all outside pairs are allowable, by their positive bracket degrees or stability.
On \(\mathbb P^1\), take the quotient of \(\mathcal O\oplus\mathcal O(1)\) onto \(\mathcal O\) for \(I\), and the quotients onto \(\mathcal O(1)\) given by \((s_j,1)\) for the other labels, with \(s_j\) general linear forms.
The common section is disjoint from the others; every other collision is an allowable pair.
This constructs \(\beta_I\).
If \(\deg D_j>0\), the original collision of \(j\) with \(I\) is stable, and all pairs containing \(j\) are allowable; hence \(\gamma_j\) exists.
Comparing the degrees \(m_i\) proves the first equality in \eqref{sg:curve-classes}.

For \(J\) maximal short, put its common section at infinity.
Every proper subset of the outside labels is short, since \(\nu(J)+\nu_j>1\).
The affine coordinates, modulo translations and scalars, give \(Z_J=\mathbb P^{n-|J|-2}\) and \(M_J=\mathcal O(-1)\), whence the normal bundle and degree against \(-K_W\).
In \eqref{sg:curve-classes} the outside labels of \(J\) are at least three: their total weight exceeds one, but every pair outside \(I\) is allowable.
Thus this \(Z_J\) has positive dimension and \(\beta_J\) exists.
All pairs containing \(j\in J\setminus I\) are allowable, either inside \(J\) or outside \(I\).
Comparison of the degrees \(m_i\) proves the second equality.
Every curve class is now in the cone on the indicated finite set, which proves the assertion after closure.
In particular no line in a locus of dimension zero where sections coincide has been included.
\end{proof}

\begin{corollary}\label{sg:initial-nef}
If \(W\ne\mathbb P^d\), then \(K_W+(n-4)\pi^*A_X\) is nef for every ample line bundle \(A_X\) on \(X_\mu\).
\end{corollary}
\begin{proof}
Only negative extremal rays require checking.
They are noncontracted because \(K_W\) is \(\pi\)-nef, so \(\pi^*A_X\) has positive integral degree on their indicated generators.
A maximal short set \(J\) with \(|J|=1\) would give \(Z_J=W=\mathbb P^d\); otherwise their degrees against \(-K_W\) are at most \(n-4\).
\end{proof}

\begin{lemma}\label{sg:curve-bound}
Let \(A=\pi^*A_X\) for an ample line bundle \(A_X\) on \(X_\mu\), and suppose the general point of \(f(C)\) is in \(X_\mu^{\mathrm s}\).
Put \(b=\deg f^*A\).
Then \(a\le(n-2)b\); if all degrees \(m_i\) are nonnegative, \(a\le2b\).
A nonzero map \(f^*T_W\to Q\) with \(\deg Q=q\le0\) implies \(a\le(n-4)b\).
More generally, for a line bundle \(Q\) of arbitrary degree \(q\), if a nonzero map \(f^*T_W\to Q\) satisfies
\[
 a-(n-4)q>2b,
\]
then at least two coincident sections have negative degrees \(m_i\), \(q<e\), and \(e+\delta\le b\).
The splitting of Proposition~\ref{sg:kernel} applies, and all maps to line bundles of degree less than \(e\) vanish on \(F\).
\end{lemma}
\begin{proof}
The perturbation in Lemma~\ref{sg:resolution} does not make an original pair of weight at least one allowable.
All \(\gamma_i\) required in the proof of Proposition~\ref{sg:mori} therefore lie in \(X_\mu^{\mathrm s}\).
Also \(\mu(I)<1\), and the constructed \(\beta_I\) has only collisions of allowable pairs of labels outside \(I\) and never meets its common section.
Thus \(A\) has positive integral degree on each curve appearing in the first decomposition of \eqref{sg:curve-classes}.
The decompositions give \(a\le2b\) or \(e+\delta\le b\), respectively, and hence \(a\le(n-2)b\).
For \(q\le0\), nonnegative degrees \(m_i\) already give \(a\le2b\le(n-4)b\).
Otherwise, for \(j\notin I\), the line bundles \(M_j\) have positive degree and map trivially to \(Q\); stalkwise evaluation removes a single section with negative degree \(m_i\), so \(\ell\ge2\) and \(a=(n-2\ell)e+2\delta\le(n-4)b\).
For the last assertion, nonnegative degrees \(m_i\) give no map when \(q<0\), and \(a-(n-4)q\le2b\) when \(q\ge0\).
If a negative degree exists but \(q\ge e\), then \(a\le(n-2)e+2\delta\le(n-4)q+2b\), again a contradiction.
Thus \(q<e\), the columns for the labels \(j\notin I\) vanish, and diagonal evaluation gives \(\ell\ge2\).
Proposition~\ref{sg:kernel} concludes.
\end{proof}

\subsection{The original arithmetic group and its automorphic line bundle}

Suppose \(X_\mu\simeq(\mathbb B^d/\Gamma)^{\mathrm{BB}}\) with \(\Gamma\) arithmetic.
We keep this group throughout.
The ball parametrizes negative lines, and \(\mathcal L\) denotes the automorphic line bundle obtained from the tautological line bundle, with rational class \(L\).
The orbifold canonical class is \(NL\).
A normalized reflection acts identically on its mirror.

\begin{proposition}\label{sg:arithmetic}\label{ew-arithmetic-full}
The original \(\Gamma\), all finite point stabilizers, and all full cusp stabilizers are generated by complex reflections.
There is a Hermitian space \((V,h)\) of dimension \(N\) over a CM extension \(K/F\), of signature \((1,d)\) at the distinguished embedding and definite at all other real embeddings of \(F\).
The normalized reflections generate a group \(\widehat\Gamma\subset\mathrm U(h,K)\) preserving one full \(\mathcal O_K\)-lattice and mapping onto \(\Gamma\), with kernel a finite group \(C_0\subset\mu(K)\) of scalars.
For \(c=|C_0|\), \(\mathcal L^c\) is an ample line bundle on \(X_\mu\).
In the nonuniform case \(F=\mathbb Q\) and \(c\mid w_K=|\mu(K)|\in\{2,4,6\}\).
\end{proposition}
\begin{proof}
The genuine wall link is simply connected: it is the circle bundle of primitive Chern class \((1,1)\) over a product of two projective spaces of positive dimension.
Its second Betti number is one, but a finite linear quotient link is a rational homology sphere.
Thus a genuine wall vertex cannot be an interior ball quotient point.
The rational smooth projective variety \(W\) is simply connected, by weak factorization \cite{WeakFactorization}.
Removing the inverse images of the cusps, of codimension at least two, preserves its fundamental group.
Brown--Higgins \cite[Propositions 1.8 and 2.8(iii)]{BrownHiggins} now makes \(\Gamma\) generated by stabilizers.
Interior smoothness and Chevalley--Shephard--Todd \cite[Theorem 1.3]{GeckMalle} show that they are generated by reflections.
The local cusp argument in Lemma~\ref{sg:cusp}, which does not use arithmeticity, gives the same assertion for the full cusp stabilizers.

Choose a normal torsion-free subgroup of \(\Gamma\) of index \(b_0\).
Reflection orders divide \(b_0\), so determinants of normalized lifts lie in \(\mu_{b_0}\), and the normalized lifts that are scalars lie in \(\mu_{Nb_0}\).
The group generated by the normalized lifts is consequently a discrete finite central extension of the original arithmetic lattice.
Apply \cite[Corollary 4.4]{Stover} to this group generated by reflections; its subgroup of elements of determinant one need not itself be generated by reflections.
This gives the stated Hermitian arithmetic datum.

For completeness, the entire projective group is rational over \(F\), not just a finite-index subgroup.
Put \(G=\mathrm{PGU}(h)\), \(\mathfrak g=\operatorname{Lie}(G)\), and choose an arithmetic \(\Gamma_{\mathrm{arith}}\subset G(F)\) commensurable with \(\Gamma\).
For \(\gamma\in\Gamma\), suitable finite-index subgroups \(\Gamma_{\mathrm{arith},1},\Gamma_{\mathrm{arith},2}\subset\Gamma_{\mathrm{arith}}\) satisfy \(\gamma\Gamma_{\mathrm{arith},1}\gamma^{-1}=\Gamma_{\mathrm{arith},2}\).
Borel density \cite{BorelDensity} and absolute irreducibility of the adjoint representation give \(\operatorname{span}_F\operatorname{Ad}(\Gamma_{\mathrm{arith},i})=\operatorname{End}_F(\mathfrak g)\).
Skolem--Noether therefore writes \(\operatorname{Ad}(\gamma)=rB\), with \(B\in\mathrm{GL}_F(\mathfrak g)\) and \(r\) real at the distinguished place.
Preservation of the bracket gives \(B[X,Y]=r[BX,BY]\); a nonzero component gives \(r\in F\).
Thus \(\gamma\in G(F)\).
Hilbert 90 for the center \(\operatorname{Res}_{K/F}\mathbb G_m\) of \(\mathrm{GU}(h)\) supplies a \(K\)-rational similitude representative.
For a reflection its mirror eigenvalue has multiplicity \(d\); because \(2d>N\), that eigenvalue lies in \(K\).
Dividing by it gives a normalized unitary reflection over \(K\).
The normalized lifts preserve a full lattice: take a finite-index normal integral subgroup and sum the finitely many translates of its invariant lattice.
This gives one lattice for the whole group generated by the normalized lifts.

Interior stabilizers act on the tautological fiber through \(C_0\), so \(\mathcal L^c\) descends to the interior of the coarse quotient.
At a cusp let \(\ell\) generate its rational isotropic line.
Every normalized cusp reflection fixes \(\ell\).
With \(h\) linear in its second variable,
\[
 s([v])=v/h(\ell,v)
\]
is a nonzero tautological frame whose \(c\)-th power is invariant under the full cusp stabilizer.
On a small cusp neighborhood \(U\), glue \(\mathcal O_U\) to the interior bundle by \(1\mapsto s^c\) on the punctured neighborhood.
This first constructs a holomorphic line bundle \(M\) on \(X_\mu\).
A sufficiently divisible ample Baily--Borel bundle \cite{BailyBorel}, killing its finite determinant character, agrees with \(M^r\) on the interior for some \(r>0\).
The isomorphism and its inverse extend across the finitely many cusps by normality, since \(d\ge2\).
They remain inverse; GAGA and the ampleness of that Baily--Borel bundle show that \(M\) is algebraic and ample.
Finally, a nonzero isotropic \(K\)-vector remains isotropic at every embedding.
In the nonuniform case this excludes all other real places of \(F\), so \(F=\mathbb Q\).
\end{proof}

\subsection{Fixed balls of finite reflection groups}

We write \(G(m,p,r)\), \(p\mid m\), for the imprimitive complex reflection group of rank \(r\), \(G_j\) for the Shephard--Todd group numbered \(j\), and \(A_r,B_r,D_r,E_r,F_4\) for the Weyl groups of the root systems of these types \cite{GeckMalle}.

\begin{lemma}\label{sg:fixed}\label{ew-fixed-quotient}
Let \(T\) be a finite reflection group in \(\Gamma\), let \(V_0\subset V\) be its common fixed space, and put \(E=V_0^\perp\), \(s=\dim E\), \(m=d-s\ge1\).
Assume \(\mathbb B(V_0)\ne\varnothing\) and \(T\) is its full generic stabilizer.
The effective image \(\Pi\) of its full normalizer is arithmetic and gives a finite map
\[
 j:(\mathbb B(V_0)/\Pi)^{\mathrm{BB}}\longrightarrow X_\mu
\]
birational onto its image.
The normalizer has finite image on \(E\), after retaining \(C_0\).
If every full interior stabilizer along this ball is \(T\times J_x\), with \(J_x\) acting trivially on \(E\), the restriction of \(j\) to \(\mathbb B(V_0)/\Pi\) is immersed and its normal bundle has components indexed by the degrees \(q\) of the basic invariants of \(T\).
A component of multiplicity \(u_q\) has finite projective monodromy and first Chern class \(-qu_qj^*L\).
The same assertion applies over the open subset \(X_\mu^{\mathrm s}\) of \(W\).
\end{lemma}
\begin{proof}
The spaces \(V_0,E\) are defined over \(K\).
Its matrices on \(E\) preserve the same \(\mathcal O_K\)-lattice and are bounded at every Archimedean embedding.
Minkowski discreteness therefore makes that image finite.
The algebraic stabilizer of their orthogonal decomposition has a compact factor on \(E\); its arithmetic subgroup projects to an arithmetic lattice on the ball by Borel--Harish-Chandra \cite{BorelHarishChandra}.
The effective full normalizer is now a discrete overgroup of this lattice and hence has finite index over it.
After a finite-index passage for this construction, a rational flag of hyperplanes and the local extension in \cite[Section 5]{Looijenga} give the extension of \(j\) across the cusps; the flag lies in \(E^\perp\), and all its normals are positive.
When \(F\ne\mathbb Q\) there are no cusps and properness is immediate.
The extended map is proper because its source is projective; its interior fibers are discrete by discreteness and the local finite point groups, hence compact and finite.
There are finitely many cusp preimages, so the extended map is finite.
A generic identification carries the full generic stabilizer \(T\) to itself and therefore belongs to the full normalizer, proving birationality.
When every full interior stabilizer along this ball is \(T\times J_x\), with \(J_x\) acting trivially on \(E\), the local map of coarse quotients is \(y\mapsto(0,y)\) in invariant coordinates for \(T\).
The linear indecomposable invariants of degree \(q\) transform by finite matrices times the inverse \(q\)-th automorphic factor.
Decomposable invariants have zero normal differential.
This gives the stated normal bundles, including their finite characters.
\end{proof}

\begin{lemma}\label{sg:marked}
If \(Z\) is the GIT quotient of \(m\ge5\) labeled points on \(\mathbb P^1\) for some weights and the relative tangent line bundle \(M_*\) along the section labeled \(*\) extends to an antiample \(\mathbb Q\)-Cartier divisor class on \(Z\), then
\[
 Z=\mathbb P^{m-3},\qquad M_*=\mathcal O(-1).
\]
\end{lemma}
\begin{proof}
First assume that the weights lie on no stability wall.
If \(\mu_*+\mu_j<1\), keep the marked points with labels other than \(j\) at distinct fixed points and move the marked point labeled \(j\) over \(\mathbb P^1\).
At each point where the marked points labeled by a pair \(\{j,k\}\) that is not allowable collide, blow up that point and contract the original fiber.
The marked points with the complementary labels then coincide, with weight \(2-\mu_j-\mu_k<1\), so the completed family is stable.
The centers avoid \(*\); each contraction increases the self-intersection of the section labeled \(*\) by one.
Thus \(M_*\) has nonnegative degree on a complete nonconstant curve, a contradiction.
All pairs with \(*\) are therefore long.
Putting \(*\) at infinity and dividing affine coordinates by translations and scalars gives \(\mathbb P^{m-3}\), with \(M_*=\mathcal O(-1)\).
At a genuine wall vertex the local class group is \(\operatorname{Pic}(\mathbb P^{|I|-2}\times\mathbb P^{|I^c|-2})/\mathbb Z[\mathcal O(1,1)]\simeq\mathbb Z\).
The class of \(M_*\) is its non-torsion generator, with sign determined by the side containing \(*\).
Hence \(M_*\) does not extend to a \(\mathbb Q\)-Cartier divisor class across that vertex.
At a wall with a factor of dimension zero the quotient for the chamber retaining that factor is locally isomorphic to the quotient for the weights on the wall.
A perturbation of the weights as in Lemma~\ref{sg:resolution} makes these choices consistently, and the argument for weights on no stability wall, which moves the marked point labeled \(j\), applies.
\end{proof}

\begin{proposition}\label{sg:persistent}\label{ew-persistent-factor}
In Lemma~\ref{sg:fixed}, suppose that every full interior stabilizer along \(\mathbb B(V_0)\) is \(T\times J_x\), with \(J_x\) acting trivially on \(E\), that \(T\) is irreducible on \(E\), and that \(s\le d-2\).
The image of \(j\) is the closed locus \(Z_I=\mathbb P^{d-s}\) where the sections labeled by a set \(I\) of labels with \(|I|=s+1\) coincide.
More precisely, \((\mathbb B(V_0)/\Pi)^{\mathrm{BB}}\) is isomorphic to the GIT quotient of the original merged weights
\[
 \mu^I=(\mu(I),(\mu_j)_{j\notin I}),\qquad
 0<\mu(I)<1,\qquad \sum\mu^I_i=2,
\]
on \(n-s\) labels, also when these weights lie on stability walls.
The automorphic line bundle of \(\mathbb B(V_0)/\Pi\) is the restriction of the original \(\mathcal L\).
All basic invariants of \(T\) have the same degree \(\delta\), and \(M_I\) extends to the antiample \(\mathbb Q\)-Cartier divisor class \(-\delta L|_{Z_I}\).
In particular \(s\le2\).
For a cocompact arithmetic lattice, the case \(s=d-1\ge3\), with the other hypotheses unchanged, is also impossible.
For equal weights, such a group \(T\) is also impossible when \(n>2(s+2)\).
\end{proposition}
\begin{proof}
In \((\mathbb B(V_0)/\Pi)^{\mathrm{BB}}\) choose a general complete intersection curve avoiding the finitely many cusps and wall vertices, and not contained in any proper inverse image of a locus where sections coincide.
After a finite cover its normal bundle is a sum of line bundles with degrees \(-q_\alpha t\), \(t>0\).
Proposition~\ref{sg:negative} factors its stalkwise normal quotient through \(M_I^{\oplus(|I|-1)}\), so \(|I|-1\ge s\).
The curve selection implies that the entire image lies in \(Z_I\); dimension gives the reverse inequality.
The selected curve lies in \(X_\mu^{\mathrm s}\), so \(\mu(I)<1\).
Thus \(|I|-1=s\), and exactness of \(\mathrm{SL}(2,\mathbb C)\)-invariants identifies the closed image \(Z_I\) of the finite map \(j\) of Lemma~\ref{sg:fixed} with the normal GIT quotient of the original merged weights \(\mu^I\).
Finite birationality then identifies \((\mathbb B(V_0)/\Pi)^{\mathrm{BB}}\) with this quotient.
Negative lines in \(V_0\) are the same lines in \(V\), so their automorphic line bundles agree.
The normal bundle is \(M_I^{\oplus s}\).
Its semistability on the curve forces all the \(q_\alpha\) to coincide.
The reflection classification \cite[Table 1]{GeckMalle} has no irreducible example with basic invariants of equal degree in rank at least 3.

On the stable interior, the determinant transition differs from the automorphic factor of degree \(-s\delta\) by a finite character.
For an integer \(r>0\) killing that character and all scalar denominators, there is an isomorphism
\[
 M_I^{sr}\simeq(\mathcal L|_{Z_I})^{-s\delta r}.
\]
Since \(\dim Z_I\ge2\), this extends \(M_I\) to a \(\mathbb Q\)-Cartier divisor class across the finitely many points of \(Z_I\) outside the stable interior; its inverse is ample.
Lemma~\ref{sg:marked} gives the projective space assertion.
In the cocompact case of rank \(d-1\), use the whole curve \(\mathbb B(V_0)/\Pi\), which has no boundary, for the same argument with the quotient onto a line bundle of negative degree and with semistability.
The resulting contradiction for basic invariants of equal degree occurs before any use of Lemma~\ref{sg:marked}; that lemma requires at least five merged labels.
For equal weights, Proposition~\ref{ew-moving-line} excludes the antiample extension when \(n>2(s+2)\).
\end{proof}

\begin{lemma}\label{sg:meyer}\label{ew-meyer}
In the nonuniform case, a finite reflection group of rank \(s\le d-2\) embeds, with the same reflection representation and reflection orders, in the quotient of a full cusp stabilizer of the original \(\Gamma\) by its intersection with the Heisenberg group.
\end{lemma}
\begin{proof}
Its fixed Hermitian space over the imaginary quadratic field has dimension at least three.
The underlying indefinite rational quadratic form has at least six variables, so Meyer \cite[Chapter IV, Section 3.2]{Meyer} supplies a fixed rational isotropic vector.
The Heisenberg lattice is torsion-free, hence the finite group injects into the quotient by the Heisenberg lattice, with unchanged nonzero roots and eigenvalues.
This argument does not apply in rank \(d-1\).
\end{proof}

\subsection{Central translations in full cusp stabilizers and cusps whose germs are cones over products of projective spaces}

The local arguments of this subsection apply to any nonuniform ball lattice.
Use Siegel coordinates \((z,w)\in\mathbb C^r\times\mathbb C\), where \(r=d-1\), and let \(Z\) be the center of the Heisenberg group, acting by \(w\mapsto w+c\).
For a full cusp stabilizer \(P\), let \(N_P\) be its intersection with the Heisenberg group, \(F_P=P/N_P\), and
\[
 P\cap Z=N_P\cap Z=c_*\mathbb Z,\qquad
 q_*=\exp(2\pi iw/c_*).
\]
The lattice \(\Lambda_P\) is the image of \(N_P\) in the group of translations of \(\mathbb C^r\).
The affine action on \(A_P=\mathbb C^r/\Lambda_P\) is faithful and has no nonidentity pure translations; no common fixed point is assumed.
For \(\lambda\in\Lambda_P\), write a lift in \(N_P\) as \((z,w)\mapsto(z+\lambda,w+a_\lambda(z))\).
The translation multipliers \(\exp(2\pi ia_\lambda(z)/c_*)\) define a negative line bundle \(\mathcal V_*\), with dual ample bundle \(M_*\) and its induced \(F_P\)-linearization.
The invariant metric is \(|q_*|^2\exp(4\pi\|z\|^2/c_*)\), of negative curvature.
The cusp contracts the zero section of \(\mathcal V_*/F_P\), with ring
\begin{equation}\label{sg:cusp-ring}
 R_P=\bigoplus_{k\ge0}H^0(A_P,M_*^k)^{F_P}.
\end{equation}
It is enough to use a high Veronese subring to construct the contraction; generation in degree one is not assumed.
The contraction of the zero section of \(\mathcal V_*/F_P\) gives the germ at the cusp of the minimal compactification \cite[Sections 2.1--2.3]{MokCompactification}; in the arithmetic case, this germ is the germ of the Baily--Borel compactification \cite[Section 4.3]{Looijenga}.
For a neat subgroup \(P_0\subset P\) of finite index, let \(f_0:A_0\to A_P\) be the induced covering of abelian varieties.
Write \(P_0\cap Z=c_0\mathbb Z\), put \(q_0=\exp(2\pi iw/c_0)\), and let \(M_0\) be the dual of the negative line bundle for \(P_0\).
If \(c_0=e_0c_*\), then \(q_*=q_0^{e_0}\) and \(f_0^*M_*=M_0^{e_0}\).
Both the constant multipliers and the remaining finite action are retained in these identities.

\begin{lemma}\label{sg:cusp}\label{ew-arithmetic-link}
If the punctured cusp germ is smooth and simply connected, \(P\) is generated by complex reflections and \(F_P\) is a reflection group.
At a smooth cusp, the linear action of \(F_P\) on \(\mathbb C^r\) is irreducible.
A cusp with germ \(\operatorname{Cone}(\mathbb P^{r_1}\times\mathbb P^{r_2},\mathcal O(1,1))\), \(r_1,r_2\ge1\), has exactly two irreducible factors in the linear action of \(F_P\) on \(\mathbb C^r\), of ranks \(r_1,r_2\).
An irreducible factor, with the lift of its action to \(M_*\), is strongly crystallographic and its ring of invariant sections is polynomial.
\end{lemma}
\begin{proof}
Apply Brown--Higgins to a contractible deep horoball and then Chevalley--Shephard--Todd at its smooth points.
The first Betti number of the link excludes nonzero vectors of \(\mathbb C^r\) fixed by \(F_P\).
Write the orthogonal irreducible factors as \(F_j\) of ranks \(r_j\).
Their invariant cohomology is \(\mathbb C[\omega_j]_j/(\omega_j^{r_j+1})_j\), \(\deg\omega_j=2\), by \cite[Theorem 1.2]{Hu}.
The Gysin sequence for the negative bundle gives \(b_{2i}=h_i-h_{i-1}\), where \(\sum h_it^i=\prod_j(1+t+\cdots+t^{r_j})\), in the first half of the link dimensions.
Comparison with a link that is a sphere or with the link of the Segre cone gives, respectively, one factor or two factors of ranks \(r_1,r_2\).
Rains \cite[Proposition 2.6]{Rains} applies to the lift of the action to \(M_*\), making the action strongly crystallographic; \cite[Theorem 2.7]{Rains} gives the polynomial ring.
Evaluation and averaging on free orbits in \(A_P\) give invariant sections in every sufficiently large degree.
Thus the degrees of the generators of the ring of invariant sections have greatest common divisor one in the grading by powers of \(M_*\).
If a group \(\mu_g\) of scalars acting only on the fibers of \(\mathcal V_*\) is first removed, the correct new coordinate is \(q_{\mathrm{new}}=q_{\mathrm{old}}^g\), with period \(c_{\mathrm{new}}=c_{\mathrm{old}}/g\), and the degrees of the generators divided by \(g\).
\end{proof}

\begin{lemma}\label{sg:cusp-lift}\label{ew-cusp-lift}
Let \((F_P)_x\) be the full stabilizer of \(x\in A_P\), and let \(\chi_x\) be its character on \((\mathcal V_*)_x\).
Its kernel \(T\) lifts to the full finite stabilizer of a nonzero point over \(x\) in a sufficiently deep horoball.
If \(T\) fixes no nonzero vector of \(\mathbb C^r\), the ball fixed by \(T\) is a disk ending at this cusp.
For the disk modulo its full normalizer, the local coordinate of the coarse quotient at this end is
\[
 Q=q_*^{\kappa_x},\qquad \kappa_x=|\operatorname{im}\chi_x|.
\]
\end{lemma}
\begin{proof}
Choose a lift of a nonzero fiber point.
Unique correction by the freely acting Heisenberg lattice lifts exactly \(\ker\chi_x\).
Its fixed Hermitian plane gives the disk when the fixed space of \(T\) in \(\mathbb C^r\) is zero.
The full central translations retain the period \(c_*\), and the residual stabilizer acts by \(\chi_x\), giving the coordinate \(Q\).
\end{proof}

\begin{lemma}\label{sg:torus}
Let \(\mathbb C^*\) act on \(\mathbb C^{p+q}\), \(p,q\ge2\), with coordinates \(X_i\) of positive weights \(u_i\) and coordinates \(Y_j\) of negative weights \(-v_j\).
Assume that, after any one weight is deleted, the remaining weights have greatest common divisor one.
If the affine quotient is smooth off the vertex and has embedding dimension \(pq\) at the vertex, all weights equal one.
\end{lemma}
\begin{proof}
At an orbit on which only \(X_i\) and \(Y_j\) are nonzero, the transverse stabilizer is \(\mu_{\gcd(u_i,v_j)}\).
The assumption on greatest common divisors excludes ineffective elements and pseudoreflections, so smoothness forces \(\gcd(u_i,v_j)=1\).
The \(pq\) monomials \(X_i^{v_j}Y_j^{u_i}\) are therefore indecomposable.
If \(u_i>1\), take \(j\ne k\) and \(1\le e<u_i\) with \(v_j+ev_k\equiv0\pmod{u_i}\).
The invariant \(X_i^{(v_j+ev_k)/u_i}Y_jY_k^e\) has an indecomposable factor divisible by none of the monomials \(X_i^{v_j}Y_j^{u_i}\), contradicting the embedding dimension.
Exchange signs to conclude.
\end{proof}

\begin{theorem}\label{sg:segre}
For a cusp with germ \(\operatorname{Cone}(\mathbb P^{r_1}\times\mathbb P^{r_2},\mathcal O(1,1))\) as in Lemma~\ref{sg:cusp}, there are product decompositions \(A_P=A_1\times A_2\) and \(F_P=F_1\times F_2\), with \(\dim A_j=r_j\).
Let \(P_j\) be the subgroup generated by all reflections whose roots lie in the \(j\)-th factor.
Then \(P_j\cap Z=P\cap Z\).
With degrees measured by \(c_*\), the ring of invariant sections of the \(j\)-th factor is polynomial on \(r_j+1\) generators of degree one.
The contraction is that of \(\mathcal O_{\mathbb P^{r_1}\times\mathbb P^{r_2}}(-1,-1)\).
For \(r_j\ge3\), \(F_j\) is \(S_{r_j+1}\) or \(G(m_j,1,r_j)\), with \(m_j\in\{2,3,4,6\}\).
The group \(F_P\) itself occurs as a finite stabilizer arbitrarily deep in the cusp.
\end{theorem}
\begin{proof}
A reflection's displacement and central multiplier belong to its root factor, and orthogonal reflections commute.
For \(j=1,2\), put \(N_j=P_j\cap N_P\), and let \(\Lambda_j\) be its image in the group of translations of \(\mathbb C^r\).
Since \(P\) is generated by reflections, \(P=P_1P_2\), \(N_P=N_1N_2\), and \(\Lambda_P=\Lambda_1\oplus\Lambda_2\).
This gives the product decomposition \(A_P=A_1\times A_2\).
Let \(c_j>0\) be the least positive period of \(P_j\) along the center, so that \(P_j\cap Z=c_j\mathbb Z\) in the same normalization of the central coordinate; write \(c_j=d_jc_*\), so \(\gcd(d_1,d_2)=1\).
For the \(j\)-th factor, let \(w_j\) be the coordinate along the center in its Siegel coordinates, put \(q_j=\exp(2\pi iw_j/c_j)\), and let \(M_j\) be the dual of the negative line bundle on \(A_j\) defined by the translation multipliers of \(P_j\).
Put \(Z_*=P\cap Z=c_*\mathbb Z\).
Rains gives \(R_1=\mathbb C[x_0,\ldots,x_{r_1}]\) and \(R_2=\mathbb C[y_0,\ldots,y_{r_2}]\), with degrees \(\alpha_i,\beta_j\) of greatest common divisor one.
Including constants and characters in the multipliers gives
\[
 q_*=q_1^{d_1}q_2^{d_2},\qquad
 M_*=M_1^{d_1}\boxtimes M_2^{d_2},\qquad
 R_P=\bigoplus_{k\ge0}(R_1)_{d_1k}\otimes(R_2)_{d_2k}.
\]
Thus \(R_P\) is the torus invariant ring of weights \(d_2\alpha_i,-d_1\beta_j\).

The residual \(\mu_{d_j}=P_jZ_*/P_j\) acts freely off the vertex of the \(j\)-th factor cone.
Otherwise combine a point of the \(j\)-th factor off the vertex with nontrivial stabilizer in \(\mu_{d_j}\) and a free orbit in the other factor.
Its finite stabilizer lies in \(P_jZ_*\), but smoothness of the punctured germ makes it generated by reflections in \(P_j\), a contradiction.
On the coordinate axis \(x_i\ne0\), the residual group \(\mu_{d_1}\) has stabilizer of order \(\gcd(d_1,\alpha_i)\), since its action is \(x_i\mapsto\zeta^{\alpha_i}x_i\).
Freeness therefore gives \(\gcd(d_1,\alpha_i)=1\) for every \(i\), and similarly \(\gcd(d_2,\beta_j)=1\) for every \(j\).
After deleting the positive weight \(d_2\alpha_i\), the remaining weights have greatest common divisor
\[
 \gcd\left(d_2\gcd_{k\ne i}\alpha_k,d_1\gcd_j\beta_j\right)
 =\gcd\left(d_2\gcd_{k\ne i}\alpha_k,d_1\right)=1.
\]
Deleting a negative weight gives the same conclusion with the factors exchanged.
This argument includes \(r_1=1\) or \(r_2=1\), because each factor still has a generator after one weight is deleted.
The analytic germ comparison gives smoothness away from the vertex and embedding dimension \((r_1+1)(r_2+1)\).
Lemma~\ref{sg:torus} forces \(d_1=d_2=\alpha_i=\beta_j=1\).
Consequently \(A_j/F_j=\mathbb P^{r_j}\), the characters of all point stabilizers on the fibers of the negative bundles are trivial, and the negative bundles descend to \(\mathcal O(-1)\) and \(\mathcal O(-1,-1)\).
The faithful affine actions have no pure translations.
The classification of smooth quotients \cite[Theorem 1.1(3)]{ALA2} gives a common fixed point, since the generators have fixed points; \cite[Theorem 1.1]{ALA1} gives the indicated irreducible groups.
The common fixed points of the two factors form a point \(x\) of \(A_1\times A_2\), so \((F_P)_x=F_P=F_1\times F_2\).
This group acts trivially on the fiber of \(\mathcal V_*\).
Lemma~\ref{sg:cusp-lift}, followed by central imaginary translation, realizes all of \(F_P\) arbitrarily deep.
\end{proof}

\section{Curves obtained from disks fixed by finite reflection groups}\label{sc:comparison}
The results of this section apply for \(n\ge6\), subject to the hypotheses of each statement.
Put \(d=n-3\), \(N=d+1\), \(H=-K_W\), \(a=H\cdot f_*C\), and \(t=L\cdot f_*C>0\).
Here \(W=X_\nu\) is the smooth GIT quotient of Lemma~\ref{sg:resolution}, obtained by perturbing \(\mu\) to weights \(\nu\) in a nearby chamber.
Let \(\Delta=\mathbb B(V_0)\) be a disk of Lemma~\ref{sg:fixed} whose full generic stabilizer, the subgroup of the original lattice fixing \(\Delta\) pointwise, is the finite reflection group \(T\), and let \(C\) be the compactification of \(\Delta/\Pi\), where \(\Pi\) is the effective image of the full normalizer of \(T\) in the original lattice; the map \(f:C\to W\) comes from this quotient.
We use \(\Delta\) for local actions on the disk and \(C\) for degrees, genera, line bundles, and intersection numbers; when \(T\) is given by its type, we call \(C\) a \(T\) curve.
The points of \(\Delta\) whose full stabilizer is larger than \(T\) form finitely many \(\Pi\)-orbits.
Let \(g=g(C)\) be the genus of \(C\), let \(S\) be the reduced divisor of the cusps of \(C\), which we call its ends, and put \(s=\deg S\).
Let \(\sigma=\sum(1-1/e)\) be the contribution of finite inertia, where \(e\) is the effective inertia order at a point; the sum excludes the points with a larger irreducible stabilizer, which are counted separately, and each application specifies the points that enter \(\sigma\).
The Riemann--Hurwitz formula for \(C\) is \(2g-2+s+\sum_x(1-e_x^{-1})=2t\), summed over the points of \(C\) with nontrivial finite inertia; the terms of the points with a larger irreducible stabilizer are written separately, and the remaining terms form \(\sigma\).
When we pass to another curve, we use the same notation for it.
We do not assume \(a>0\), immersion, or separate descent of an automorphic factor and a finite character.
From Section~\ref{sc:comparison} on, all degrees of invariant sections at cusps are measured by the positive generator \(c_*\) of \(P\cap Z\) for the full cusp stabilizer \(P\).

\subsection{The determinant of the differential and one row of its inverse}
\begin{proposition}
\label{sc:curve-determinant}\label{ew-curve-determinant}
Let \(X\) be the Baily--Borel compactification of an arithmetic ball quotient of dimension \(d\), and let \(K\) be the full generic stabilizer of a ball of dimension \(m=d-k\ge1\) fixed by a finite subgroup.
Assume that \(K\) is generated by reflections and has rank \(k\).
Let \(f:Y\to X\) be the finite morphism from the compactified quotient of this ball by the full normalizer of \(K\) in the original lattice.
Assume that \(K_X\) is \(\mathbb Q\)-Cartier near \(f(Y)\), and that the general point of every interior divisor of \(Y\) maps to the smooth locus of \(X\).
Put \(H=-K_X\), and write \(D_G\) for the sum of the degrees of the basic invariants of a reflection group \(G\).
There is an effective \(\mathbb Q\)-divisor on \(Y\) with class
\[
f^*H+(D_K+m+1)f^*L.
\]
Its interior support consists of the divisors with a larger full stabilizer.
For \(m>1\), let \(\Sigma_K\) be the sum, as reduced cycles, of the images in \(X\) of all fixed balls of dimension \(m\) whose full generic stabilizer acts on the normal space by a complex reflection representation isomorphic to that of \(K\), and define \(\Sigma_J\) in the same way for the groups \(J\) of rank \(k+1\).
Let \(n_{K,J}\) count the lines \(\ell\) in the normal representation of \(J\) whose full pointwise stabilizer \(J_\ell\) acts on the directions transverse to \(\ell\) by a representation isomorphic to that of \(K\).
Then
\[
(H+(D_K+m+1)L)\Sigma_K
=\sum_{\operatorname{rank}J=k+1}
(D_J-D_K-1)\frac{|K|n_{K,J}}{|J|}\Sigma_J.
\]
There is no cusp divisor term when \(m>1\).

For \(m=1\), write \(Y=C\) and assume that \(X\) is smooth near \(f(C)\), with polynomial rings of invariant sections at its ends.
Put \(a=\deg f^*H\), \(t=\deg f^*L\), and \(e_y=|N_{J_y}(K_y)|/|K_y|\), where \(J_y\) is the full interior stabilizer and \(K_y\) the pointwise stabilizer of the corresponding local axis.
At an end \(c\), let \(b_{c,1},\ldots,b_{c,d}\) be the degrees of the generators of the ring of invariant sections, measured by the positive generator \(c_*\) of \(P\cap Z\) for the full cusp stabilizer \(P\), and put \(A_c=\sum_i b_{c,i}\).
Let \(\kappa_c\) be the order of the character of the full stabilizer on the fiber of \(\mathcal V_*\), so that the local coordinate \(Q\) of the coarse curve at \(c\) and the coordinate \(q=q_*\) for the full cusp stabilizer \(P\) at \(f(c)\) satisfy \(Q=q^{\kappa_c}\).
Then
\begin{equation}
\label{ew-curve-determinant-equation}
a+(D_K+2)t
=\sum_{y\in C\setminus S,\ J_y\ne K_y}
\frac{D_{J_y}-D_K-1}{e_y}
+\sum_{c\in S}\frac{A_c}{\kappa_c}.
\end{equation}
No immersion or separate descent of the line bundle for a basic invariant is required.
If \(H\) is nef, every component with \(m>1\) contains an interior stratum of rank \(k+1\).
\end{proposition}

\begin{proof}
Lemma~\ref{sg:fixed} gives \(Y\) and the finite representations on invariant normal directions.
Indecomposable invariants of degree \(e\) transform by \(\mathcal L^{-e}\) with their finite representation; terms in the square of the invariant ideal disappear from the normal differential.
Repeated degrees are kept in one vector bundle.
The determinant of the tangent bundle of the fixed ball has automorphic factor \(\mathcal L^{-(m+1)}\) and a finite character.
Thus the determinant of the domain of the differential of the quotient map is \(\mathcal L^{-(D_K+m+1)}\) with a finite character.
Take a sufficiently divisible integer \(v>0\) to remove this character, make \(vK_X\) Cartier, and make \(\mathcal L^v\) descend to a Baily--Borel line bundle.
The determinant then gives a nonzero section of the line bundle
\[
f^*\mathcal O_X(-vK_X)\otimes
(f^*\mathcal L^v)^{\otimes(D_K+m+1)}.
\]
These finite representations are removed on a normal neat cover; their automorphic factors are retained.

At an interior divisor with full stabilizer \(J\), discard the common fixed directions and write \(u\) for the remaining axis coordinate, \(\eta_i\) for the basic invariants of \(K\), and \(F_j\) for those of \(J\).
Homogeneity and generic invertibility give
\[
\det\frac{\partial(F_1,\ldots,F_{k+1})}
{\partial(u,\eta_1,\ldots,\eta_k)}(u,0)
=c_0u^{D_J-D_K-1},\qquad c_0\ne0.
\]
The group \(N_J(K)/K\) acting on the axis has order \(e=|N_J(K)|/|K|\), so the divisor coefficient after division by \(v\) is \((D_J-D_K-1)/e\).
Changes of invariant coordinates multiply the determinant by a unit, and the tangent column is included.
The section is regular in codimension one away from cusps and extends across the remaining interior subset by normality of \(Y\).
When \(m>1\), the cusp set also has codimension at least two, so the same extension applies.
In a toroidal cover, a boundary divisor maps to a point and has zero pushforward as a divisor on \(Y\).

For \(m=1\), extend the section in the logarithmic frame containing \(q\,\partial/\partial q\) at each end.
On the punctured neighborhood of the cusp, the frame of Proposition~\ref{sg:arithmetic}, a power of which is invariant under the full cusp stabilizer, introduces no power of \(q_*\).
At an end \(c\), let \(x_c\) be the point of \(A_P\) determined by the curve \(f\).
Write the invariant theta functions \(\theta_{c,j}\) in local invariant coordinates \(\boldsymbol\eta\), centered at \(x_c\), as \(\bar\theta_j(\boldsymbol\eta)\).
Thus \(\theta_{c,j}(x_c)=\bar\theta_j(0)\); we omit the index \(c\) when one end is under consideration.
At an end the contraction and its constant differential are
\begin{equation}\label{sc:theta-map}
U_j=q^{b_{c,j}}\bar\theta_j(\boldsymbol\eta),\qquad J_c=[(b_{c,j}\bar\theta_j(0))_j,D_{\boldsymbol\eta}\bar\theta(0)].
\end{equation}
The differential is \(\operatorname{diag}(q^{b_{c,j}})J_c\), and \(J_c\) is invertible by smoothness of the punctured quotient with full stabilizer \(K\).
If \(Q=(q')^r\) on the neat cover, the determinant of the differential has order \(rA_c/\kappa_c\), giving \(A_c/\kappa_c\) after descent.
This proves effectivity and the curve formula.
For \(m>1\), summing over the orbits of local axes gives \(\sum1/e=|K|n_{K,J}/|J|\); pushing the divisors forward proves the formula for \((H+(D_K+m+1)L)\Sigma_K\).
Finally, \((f^*H+(D_K+m+1)f^*L)(f^*L)^{m-1}>0\) if \(H\) is nef, so the effective divisor is nonzero and contains the asserted interior stratum.
\end{proof}

\begin{lemma}\label{sc:monomial-normalizer}
For \(r\ge5\), the linear normalizer of \(G(m,p,r)\), \(m\ge2\), \(p\mid m\), lies in \(\mathbb C^*G(m,1,r)\) and preserves the line spanned by each elementary symmetric invariant and the line spanned by \(P_0=(x_1\cdots x_r)^{m/p}\).
\end{lemma}
\begin{proof}
For a nonzero vector \(x\in\mathbb C^r\), let \(k\) be the number of its zero coordinates, so \(0\le k\le r-1\).
The bounds \(m^kk!(r-k)!/p\), \(k>0\), and \(r!\), \(k=0\), for the stabilizers of such points are logarithmically convex.
The strict inequality \(m^{r-1}>pr\) identifies the coordinate axes by the maximal stabilizers of nonzero vectors.
A normalizer permutes these axes; the reflections of \(G(m,p,r)\) exchanging two coordinates up to factors in \(\mu_m\) then force all diagonal ratios into \(\mu_m\).
\end{proof}

\begin{lemma}\label{sc:age}
If \(u\mapsto\zeta_eu\) acts on a line by \(\zeta_e^a\), \(0\le a<e\), the line bundle of its invariant sections on the coarse curve has degree equal to its orbifold degree minus \(a/e\).
\end{lemma}
\begin{proof}
The invariant frame is \(u^a\) times a function of \(u^e\).
\end{proof}

For \(r=d-1\ge3\), \(m\ge2\), and \(p\mid m\), put \(T=G(m,p,r)\) and
\[
 f_j=e_j(x_1^m,\ldots,x_r^m)\ (1\le j<r),\quad
 P_0=(x_1\cdots x_r)^{m/p},\quad
 D=\frac{mr(r-1)}2+\frac{rm}{p}.
\]
Set \(k=rm\) if \(p=1\), and \(k=(r-1)m\) otherwise.
\begin{proposition}\label{sc:monomial-comparison}\label{ew-monomial}
Suppose every larger interior group on a disk fixed by \(T\) is \(G(m,p,r+1)\), with the prescribed coordinate axis, or \(T\times\mu_e\), with \(\mu_e\) trivial on the normal space.
Suppose every end \(c\) is a smooth cusp at which the character \(\chi_{x_c}\) of the full stabilizer \((F_P)_{x_c}\) in Lemma~\ref{sg:cusp-lift} has kernel \(T\) and \(|\operatorname{im}\chi_{x_c}|=\kappa_c\).
Suppose the full normalizer preserves the indecomposable direction of degree \(k\) and, in its effective image, the stabilizers of points of the curve fix that direction.
Let \(D_R\) be the reduced divisor of the points with stabilizer \(G(m,p,r+1)\), and put \(R=\deg D_R\).
Here \(\sigma\) is taken over the points with stabilizer \(T\times\mu_e\).
Put
\[
 B_c=\frac{\sum_i b_{c,i}}{\kappa_c},\quad
 \nu_c=\left\lfloor\frac{\max_i b_{c,i}}{\kappa_c}\right\rfloor,\quad
 \beta=r+\frac1p-\frac1m .
\]
There is a line bundle \(E_k\) of degree \(-kt\), and
\begin{align}
 2g-2+s+(1-m^{-1})R+\sigma&=2t,\label{sc:monomial-log}\\
 a&=\sum_cB_c+\beta R+\sigma-(D+2)t,\label{ew-comparison-canonical}\\
 0\ne f^*T_W&\longrightarrow E_k(D_R+\sum_c\nu_c[c]),\qquad
 h=-kt+R+\sum_c\nu_c.\label{sc:monomial-inverse}
\end{align}
\end{proposition}
\begin{proof}
Let \(j_\gamma\) be the factor of automorphy of the original automorphic line bundle \(\mathcal L\) for the chosen lift of \(\gamma\), and let \(\chi_q(\gamma)\) be the factor from the finite action on the line spanned by an invariant of polynomial degree \(q\).
On the line bundle for this invariant, the transition is \(j_\gamma^{-q}\chi_q(\gamma)\).
Scalar lifts cancel between the two factors, so their product is effective.
A neat cover of degree \(v\), of any genus, computes degree \(-qvt\).
This proves the assertion for \(E_k\); the Riemann--Hurwitz formula for \(C\) gives \eqref{sc:monomial-log}.
Put \(\delta=m/p\).
The local map at a point with stabilizer \(G(m,p,r+1)\) is
\begin{equation}\label{sc:monomial-matrix}
 g_1=u^m+f_1,\quad
 g_j=f_j+u^mf_{j-1}\ (2\le j<r),\quad
 g_r=P_0^p+u^mf_{r-1},\quad Z=u^\delta P_0 .
\end{equation}
Its Jacobian determinant along the axis is \(mu^{rm+\delta-1}\), so Proposition~\ref{sc:curve-determinant} gives \(\beta\).
The residual action sends \(P_0\) to \(\zeta_m^{-\delta}P_0\), so passing to the coarse curve changes the degree of the line bundle corresponding to \(P_0\) by \(-(1-1/p)R\).
For \(p>1\), write \(P_0=u^{m-\delta}\psi(u^m)\).
After elimination of the tangent column, the block of the elementary symmetric invariants has diagonal \(v=u^m\) and superdiagonal \(1\).
For \(p=1\), the same follows from \(dg_1=dv+df_1\).
One row of the inverse differential has a simple pole: \(\dot P_0=v^{-1}\dot Z\), or \(\dot f_{r-1}=v^{-1}\dot g_r\).
At a point with stabilizer \(T\times\mu_e\), the inverse differential has no pole in the normal directions.
At a cusp invert the differential in \eqref{sc:theta-map}.
The rows of the inverse for the directions of \(A_P\) have pole at most \(q_*^{-\max b_{c,i}}\).
For the row of the inverse differential being used, the finite inertia group at the cusp acts trivially on its target invariant line, so the entries of this row are functions of \(Q=q_*^{\kappa_c}\), and the pole bound \(\max_ib_{c,i}\) in \(q_*\) becomes \(\nu_c\) in \(Q\).
Equation~\eqref{ew-curve-determinant-equation} now gives \eqref{ew-comparison-canonical}.
\end{proof}
Whenever a unique degree \(m\) line descends, \(mt\) is integral.
When degrees are repeated, the same argument proves integrality only for \(t\) times the greatest common divisor of the degrees of the line bundles actually shown to descend.

\subsection{Elimination of coefficients and injections into sums of line bundles}
\begin{lemma}\label{sc:coefficient-elimination}\label{ew-coefficient}
Let \(Q,E_1,\ldots,E_q\) be line bundles on \(C\), and consider homomorphisms \(f^*T_W\to Q\) with zero coefficient on the tangent direction of \(C\).
Suppose their coefficients \(b_i\in H^0(C,QE_i^{-1})\) with respect to the line bundles \(E_i\) determine them injectively.
At a reduced divisor \(D\), suppose the conditions for their coefficients to extend across \(D\) come from the upper bidiagonal differential with first row \((a(u),1,0,\ldots)\), \(a(u)\not\equiv0\), and remaining diagonal \(u\) and superdiagonal \(1\).
At each end \(c\), suppose the differential is \(\operatorname{diag}(u^{\nu_i(c)})J(u)\).
Assume the exponents are nondecreasing and every leading block at a strict increase is invertible at zero.
If
\[
 \deg QE_i^{-1}(-D-\sum_c\nu_i(c)[c])<0\quad(1\le i<q),
\]
and \(\deg QE_q^{-1}(-D-\sum_c\nu_q(c)[c])=0\), the space of these homomorphisms has dimension at most one.
\end{lemma}
\begin{proof}
The vanishing of the coefficient on the tangent direction kills the coefficient of the first row at \(D\).
Once \(b_1,\ldots,b_{i-1}\) vanish, \(b_i\) is divisible by \(u\).
The invertible leading blocks at ends give the analogous successive divisibility.
Negative degree kills the first \(q-1\) coefficients, and degree zero bounds the last by one.
\end{proof}

\begin{lemma}\label{sc:normal-splitting}
Suppose \(\bigoplus_{i=1}^r\mathcal O_{\mathbb P^1}(-ij)\to V\), \(j\ge1\), is invertible away from \(R\) points and, at each of these points, has in its individual invariant frames the upper bidiagonal matrix with diagonal \(v\) and superdiagonal \(1\), where \(v\) is a local coordinate vanishing simply at the point.
Then \(V=\bigoplus_{i=1}^r\mathcal O(R-ij)\).
For \(p>1\), the block of the elementary symmetric invariants has length \(r-1\), and the separate direction of \(P_0\) has degree \((R-rj)/p\).
\end{lemma}
\begin{proof}
Near one of these points, let \((c_1,\ldots,c_r)\) be the local coefficients of a map \(V\to\mathcal O(q)\); the coefficients \(b_i\in H^0(\mathcal O(q+ij))\) of its pullback satisfy \(b_1=vc_1\) and \(b_i=c_{i-1}+vc_i\) for \(2\le i\le r\), and a family \((b_i)\) extends to a map from \(V\) exactly when the \(c_i\) so determined are holomorphic.
Induction kills the coefficients with \(q+ij<R\); thereafter evaluation at \(R\) points is surjective.
Thus the space of maps \(V\to\mathcal O(q)\) has dimension \(\sum_i\max(q+ij+1-R,0)\), which determines the splitting.
In the case of the line bundle corresponding to \(P_0\), add the contribution \(R\) of its elementary modification at the \(R\) points to the change \(-(1-1/p)R\) on passing to the coarse curve and to its degree \(-rj/p\) before these changes.
\end{proof}

\begin{proposition}\label{sc:bundle-tests}
Suppose \(C=\mathbb P^1\) satisfies the hypotheses of Proposition~\ref{sc:monomial-comparison} with every invariant direction preserved.
Put \(j=mt\), \(\rho=\operatorname{lcm}_c\kappa_c\), and \(B=\sum_c\max_i b_{c,i}/\kappa_c\).
When at most two widths are nontrivial, use the cyclic cover \(\varrho:C'\to C\) of degree \(\rho\), branched at two values containing those ends and avoiding the points with a larger stabilizer.
Use another end as the second value if possible; let \(s_0\) count auxiliary branch values, and put \(B_T=B+s_0(1-1/\rho)\).
For \(\rho=1\) use the identity cover and \(s_0=0\).
Let \((d_i)=(m,2m,\ldots,rm)\) for \(p=1\), or \((m,\ldots,(r-1)m,rm/p)\) otherwise.
Then \(C'=\mathbb P^1\), and the degrees of the line bundles in the following injections are integers:
\begin{align}
 (f\varrho)^*T_W&\hookrightarrow
 \mathcal O(\rho B_T+2)\oplus\bigoplus_i\mathcal O(\rho(B_T-d_it))
 &&(R=0),\label{sc:tangent-test}\\
 \varrho^*N_f^{\rm sat}&\hookrightarrow
 \begin{cases}
 \bigoplus_{i=1}^r\mathcal O(\rho(B+R-ij)),&p=1,\\
 \bigoplus_{i=1}^{r-1}\mathcal O(\rho(B+R-ij))
 \oplus\mathcal O(\rho(B+(R-rj)/p)),&p>1 .
 \end{cases}\label{sc:normal-test}
\end{align}
For \(R=0\), the comparison using the normal bundle allows arbitrarily many nontrivial widths and any finite cover \(\varrho:C'\to C\) on which they become trivial; its genus need not be zero.
For such a cover, \(\rho\) denotes \(\deg\varrho\) in~\eqref{sc:normal-test} and in the division by \(\rho\) that defines the \(v_i\), and each summand in~\eqref{sc:normal-test} is the line bundle on \(C'\) obtained in the proof from the corresponding basic invariant and the poles at the cusps; the expression inside \(\mathcal O(\cdot)\) gives only its degree.
The critical divisor of \(f\) has degree
\[
 c_f=\sum_c\left(\min_{i:\theta_{c,i}(x_c)\ne0}b_{c,i}/\kappa_c-1\right).
\]
If the quotient \(D\) of Proposition~\ref{sg:negative} is constructed from \(\ell\) coincident sections of common degree \(-e<0\), sort the degrees of the target line bundles, divided by \(\rho\), as \(v_1\ge\cdots\ge v_q\).
A necessary condition is
\begin{equation}\label{sc:sorted-test}
 \#\{i:v_i\ge e\}\ge k_0,\qquad A\le\sum_{i=1}^{k_0}v_i ,
\end{equation}
where \((k_0,A)=(n-\ell-2,a+(\ell-1)e)\) for the comparison using the tangent bundle, and \((k_0,A)=(n-\ell-3,a-2-c_f+(\ell-1)e)\) for the comparison using the normal bundle.
Failure at the least allowed \(e\) excludes every larger \(e\) with all other data unchanged.
\end{proposition}
\begin{proof}
By the Riemann--Hurwitz formula, the cover branched at two points is rational.
At an end of ramification index \(e_c\), write \(q=x^{e_c/\kappa_c}\).
Saturating the inertia gives the degrees \(-\rho ij\) of the line bundles for the elementary symmetric invariants and the degree \(-\rho rj/p-\rho(1-1/p)R\) of the line bundle corresponding to \(P_0\).
Extend through the points with stabilizer \(G(m,p,r+1)\) by \eqref{sc:monomial-matrix}; Lemma~\ref{sc:normal-splitting} gives the summands of the normal bundle.
For \(R=0\), the local charts given by the full normalizer glue on the cover \(C'\) to an extension of \(\bigoplus_i\mathcal O_{C'}(-\rho d_it)\) by \(T_{C'}\), and this extension splits by \(H^1(\mathcal O(2+\rho d_it))=0\).
At the cusps, the inverse of the differential has poles of total degree \(\rho B\); an auxiliary branch value adds only a pole of order \(\rho-1\) in the tangent direction.
This proves both injections.
When we use a cover of arbitrary genus for the comparison using the normal bundle, no interior modification is needed and the line bundles for the basic invariants already form a direct sum.
A cover on which all widths become trivial is the normalization of \(y^\rho=\prod_c(z-z(c))^{\rho/\kappa_c}\) for an affine coordinate \(z\) on \(C=\mathbb P^1\); extra branching adds no pole in the directions of \(A_P\).
Saturating the image of the differential commutes with finite pullback, and differentiating the image of the curve under the map in \eqref{sc:theta-map} gives \(c_f\).
The bundle \(F\) of Proposition~\ref{sg:kernel}, or its quotient \(\mathcal K\) by the saturated image \(T_f\) of \(df\), has rank \(k_0\), degree \(A\), and is generated by line bundles of degrees at least \(e\).
Maps to target line bundles of degree below \(e\) vanish.
A nonzero maximal minor proves the degree inequality in \eqref{sc:sorted-test}.
As \(e\) increases, the two necessary conditions in \eqref{sc:sorted-test} become more restrictive.
\end{proof}

\subsection{Two geometric inequalities}
\begin{lemma}\label{sc:complete-curves}
At a fixed point \(x\in A_P\) the cusp width is \(\kappa=\gcd\{b_i:\theta_i(x)\ne0\}\).
If \(b_i\in\{1,3\}\) and \(\kappa=1\), the curve \(f\) is immersed at the end.
\end{lemma}
\begin{proof}
Take a nonzero vector \(v\in(\mathcal V_*)_x\).
The scalars on the fibers of \(\mathcal V_*\) that fix the image of \(v\) in the quotient by \(F_P\) of the complement of the zero section form the group \(\operatorname{im}\chi_x\) of Lemma~\ref{sg:cusp-lift}; in terms of the invariant sections, this group is \(\{\lambda\in\mathbb C^*:\lambda^{b_i}=1\text{ whenever }\theta_i(x)\ne0\}=\mu_\kappa\).
A width of one forces a nonzero section of degree one, giving a nonzero derivative in \(U_i(Q)\).
Properness extends the curve to \(W\), and \(d\pi\,df_W=df_X\) preserves this immersion.
\end{proof}

\begin{lemma}\label{sc:moving-coordinate}
Suppose that the smooth coarse ball quotient contains a family of dimension \(k\) of distinct complete embedded rational curves that avoids the strata of codimension at least two defined by stabilizers and meets the branch divisor transversely.
Then a general member of the family has at least \(k+3\) branch values.
In particular, a coordinate curve \(\gamma_i\) on \(W\) satisfies
\begin{equation}\label{sc:coordinate-bound}
 L\cdot\gamma_i\ge\frac{d-2}{2(d+1)}=\frac{n-5}{2(n-2)}.
\end{equation}
\end{lemma}
\begin{proof}
On a normal neat cover the pullback families are smooth proper after restricting parameters; choose connected components of fixed degree after finite base change.
Each cyclic inertia acts freely on sheets, so every component retains all branch values.
With \(b\) branch values, finite permutation monodromy tuples leave at most \(b-3\) parameters.
The cases \(b\le2\) would make these components rational, impossible in a ball quotient.
Let \(X'\) be the smooth ball quotient associated with the normal neat cover.
For one such component \(D\), fixed as an abstract curve, and every nonconstant map \(v:D\to X'\), negative bisectional curvature and the integrated Bochner formula give \(H^0(D,v^*T_{X'})=0\) \cite{Kobayashi}.
Hence \(k\le b-3\).
For an existing coordinate curve \(\gamma_i\), all pairs containing \(i\) are allowable.
The open set where the other marked points are distinct is therefore the universal smooth projective line bundle over \(M_{0,n-1}\).
Its fibers are the complete embedded coordinate curves and its base has dimension \(d-1\).
The chamber perturbation never allows a pair of original weight at least one, so the family lies in \(X_\mu^{\mathrm s}\).
The finite cusp set has proper image in the base.
A stabilizer stratum of codimension at least two has base image of dimension at most \(d-2\), so a general fiber avoids it.
A branch divisor with nondominant base image is also avoided; otherwise its projection is generically smooth in characteristic zero, making a general fiber transverse.
The family thus satisfies the hypotheses of the lemma with \(k=d-1\), so a general \(\gamma_i\) has \(b\ge d+2\) branch values.
Since \(K_W\cdot\gamma_i=-H\cdot\gamma_i=-2\), the formula for the canonical class gives \((d+1)L\cdot\gamma_i=-2+\sum_x(1-e_x^{-1})\ge-2+b/2\), where \(e_x\) is the local ramification index of the induced neat covering of \(\gamma_i\) at the branch value \(x\); this gives~\eqref{sc:coordinate-bound}.
\end{proof}

\begin{lemma}\label{sc:negative-normal}
Let \(Z=\mathbb P^k\), \(k\ge2\), be closed in the Baily--Borel compactification of an arithmetic ball quotient, with \(Z^\circ=Z\setminus S\) in its smooth interior for a finite set \(S\) containing the boundary points.
Suppose \(N_{Z^\circ}=\mathcal O(-1)^{\oplus c}\) and the class \(L\) of the original lattice restricts to \(\lambda\mathcal O(1)\).
Retain the finite maps \(j\) of Lemma~\ref{sg:fixed}, defined by the full normalizers, and the pullbacks \(j^*L\).
Let the full stabilizer of a general point of the ball over \(Z^\circ\) have rank \(s\le c\).
If \(s=c\), the Baily--Borel compactification of the quotient of its fixed ball by the full normalizer is \(Z\).
If \(s<c\), then \(\lambda\ge1\); equality makes the inverse image of \(Z^\circ\) in the fixed ball of that stabilizer locally totally geodesic at its smooth points.
\end{lemma}
\begin{proof}
For \(s=c\), the quotient of the fixed ball by the full normalizer maps finitely to \(Z\) with a unique generic preimage over \(Z^\circ\), so its compactification is finite birational to normal \(Z\), hence isomorphic to it.
For \(q=c-s>0\), pass to a neat cover \(Y_0\) of the quotient \(\mathbb B(V_0)/\Pi\) and resolve the graph of the Gauss map of the inverse image of \(Z^\circ\).
Denote this resolution by \(\widetilde Z\), let \(\rho:\widetilde Z\to Y_0\) be the resulting map, and put \(T=T_{Y_0}\).
The saturation of the image of \(d\rho\) gives a subbundle \(\mathcal S\subset\rho^*T\) of rank \(k\).
The quotient \(\mathcal Q\) has rank \(q\), maps generically injectively to the pullback of \(\mathcal O(-1)^{\oplus c}\), and \(\mathcal Q\otimes\rho^*\mathcal L\) is semipositive.
Let \(z\in\widetilde Z\) be a point at which \(\rho\) is an immersion, let \(v\in T_z^{1,0}\widetilde Z\), and let \(\xi\in\mathcal Q_z\), lifted by the quotient metric to \(\mathcal S_z^\perp\subset T_{\rho(z)}Y_0\).
Then \(d\rho(v)\perp\xi\), and the curvature of the ball metric evaluated on \(d\rho(v)\) and \(\xi\) is \(-\|d\rho(v)\|^2\|\xi\|^2\); the curvature of the quotient bundle \(\mathcal Q\) adds the nonnegative square of the second fundamental form \cite[Section~1.6]{Kobayashi}.
On a complete curve \(\widetilde C\subset\widetilde Z\) over a line in \(Z\) avoiding \(S\), put \(b=\deg(\mathcal O_Z(1)|_{\widetilde C})>0\), where \(\mathcal O_Z(1)|_{\widetilde C}\) is the pullback of \(\mathcal O_Z(1)\) to \(\widetilde C\).
On \(\widetilde C\), the formula for the curvature of \(\mathcal Q\) at the points where \(\rho\) is an immersion gives
\[
0\le\deg\det((\mathcal Q\otimes\rho^*\mathcal L)|_{\widetilde C})\le q(\lambda-1)b.
\]
Thus \(\lambda\ge1\).
At equality the integral of the nonnegative curvature term is zero, so the second fundamental form vanishes along every such curve.
Such curves over general lines cover all tangent directions at the points where \(\rho\) is an immersion, proving that the inverse image of \(Z^\circ\) is locally totally geodesic at its smooth points.
This argument uses the quotient of the fixed ball by the full normalizer; it does not lift a splitting of the tangent bundle of the coarse quotient through ramification.
\end{proof}
The conclusion at equality is supplementary and is not needed for Theorem~\ref{si:main}.

\begin{proposition}\label{sc:type-a}
Let the curve \(f:C\to W\) have full generic stabilizer \(A_r\), \(r\ge2\), larger groups only \(A_{r+1}\) and \(A_r\times\mu_e\), and ends at which the groups \(F_P\) have rings of invariant sections generated in degree one.
Put \(q=r+1\), let \(R\) count \(A_{r+1}\) points, take \(\sigma\) over the points with stabilizer \(A_r\times\mu_e\), and put \(h=-qt+R+s\).
If \(h<0\), the space of homomorphisms to the line bundle associated with the invariant of greatest degree has dimension one.
In all cases
\begin{equation}\label{sc:type-a-bound}
 a\ge\frac{r^2-3}{2}t+\sigma .
\end{equation}
\end{proposition}
\begin{proof}
The line bundles \(E_j\), \(2\le j\le q\), associated with the polynomial invariants on the normal space have degrees \(-jt\), with the products of automorphic factors and finite characters as transition functions.
At \(A_{r+1}\), expand
\[
 (z+qu)\left((z-u)^q+\sum_{j=2}^q f_j(z-u)^{q-j}\right).
\]
The differential has diagonal entry \(-q(q+1)u\) in the tangent direction, diagonal entries \((q+1)u\) in the normal directions, and superdiagonal one.
Let \(D_R\) be the reduced divisor of \(A_{r+1}\) points.
The line bundle associated with the invariant of greatest degree is \(Q=E_q(D_R+S)\).
After division by the zeros forced by the differential, the coefficient of a homomorphism \(f^*T_W\to Q\) on the tangent direction of \(C\) has degree \(-(q-2)t-\sigma<0\).
Lemma~\ref{sc:coefficient-elimination} kills the coefficients for all \(j<q\), whose degrees after this division are \(-(q-j)t\); the last is constant.
Let \(N=N_f^{\mathrm{sat}}\), and put \(E=E_{q-1}\oplus E_q\).
The inverse normal differential followed by projection onto these two directions defines a morphism \(\psi:N\to E(2D_R+S)\).
Define \(\mathcal Q=\operatorname{im}\psi\), without taking its saturation in the target.
This image is a vector bundle of rank 2 on \(C\), and \(N\to\mathcal Q\) is surjective.
At an \(A_{r+1}\) point, put \(\alpha=(q+1)u\); the stalk is generated, in the frames of \(E_{q-1},E_q\), by \( (\alpha^{-1},0)\) and \( (-\alpha^{-2},\alpha^{-1})\).
At a point with stabilizer \(A_r\times\mu_e\), the normal differential is invertible and the stalk is \(E\).
At a cusp with parameter \(v\), the invertibility of \(J_c\) in \eqref{sc:theta-map}, after removing the radial direction, gives the stalk \(v^{-1}E\).
Projection onto the second coordinate is consequently surjective onto \(Q=E_q(D_R+S)\) at every point, and its kernel is \(E_{q-1}(D_R+S)\).
Thus
\[
 0\longrightarrow E_{q-1}(D_R+S)\longrightarrow\mathcal Q
 \longrightarrow Q\longrightarrow0,\qquad
 \deg\det\mathcal Q=-(2q-1)t+2R+2s=2h+t.
\]
These equalities of stalks exclude any additional quotient of finite length.
If both \(h<0\) and \(2h+t<0\), Proposition~\ref{sg:determinant} forces at least three coincident sections.
Proposition~\ref{sg:negative} then shows that the Hom space has dimension at least two, a contradiction.
Thus \(h<0\) implies \(2h+t\ge0\).
Substituting the definition of \(h\) into equation~\eqref{ew-curve-determinant-equation} gives
\[
 qh=a-\frac{(q-2)(q+1)}2t-\sigma .
\]
Substitute \(h\ge-t/2\); for \(h\ge0\), substituting \(h\ge0\) instead gives a stronger bound.
\end{proof}

\subsection{Integral classes and local flags}
\begin{proposition}\label{sc:chern-integrality}
On a smooth coarse ball quotient of dimension \(d\), let \(c_i^{\rm orb}\) denote the average of tangent Chern classes on a neat cover.
Then \(c_i^{\rm orb}=(-1)^i\binom{d+1}{i}L^i\).
If all stabilizers in codimension at most \(i\) have orders dividing \(M_i\), then \(M_ic_i^{\rm orb}\) is integral.
It extends over a removed subset of codimension at least \(i+1\).
If all reflections have order two,
\[
 \mathcal T=2NL=2K_W+B\in\operatorname{Pic}(W)
\]
for the reduced branch divisor \(B\).
For products of type \(A\), the denominators \(M_2=12\), \(M_3=24\) suffice.
On a closed \(\mathbb P^k\), \(k\ge3\), write \(\mathcal T=b\mathcal O(1)\), \(b>0\).
Then
\begin{equation}\label{sc:chern-divisibility}
 \frac{3(N-1)b^2}{2N},\
 \frac{(N-1)(N-2)b^3}{2N^2}\in\mathbb Z,\qquad
 N^2\mid b^3,\qquad L\cdot\text{line}\ge\frac1{2N^{1/3}} .
\end{equation}
\end{proposition}
\begin{proof}
The following identities do not require a lower bound on the number of labels.
Let \(U\) be a smooth coarse ball quotient of dimension \(d\), let \(p:Y\to U\) be a finite neat cover of degree \(v\), and let \(L=c_1(\mathcal L)\).
The Euler sequence, whose middle term is a flat bundle twisted by a line bundle, gives \(c(T_Y)=(1-p^*L)^{d+1}\) in rational cohomology.
If \(G_x\) is the full stabilizer at \(x\in U\), proper pushforward of constructible functions and the normalization of Chern--Schwartz--MacPherson classes give
\begin{equation}
\label{ew-comparison-chern}c_{\mathrm{SM}}\bigl(x\longmapsto|G_x|^{-1}\bigr)
=v^{-1}p_*\bigl(c(T_Y)\cap[Y]\bigr),
\qquad c_i^{\mathrm{orb}}=(-1)^i\binom{d+1}{i}L^i.
\end{equation}
Here the orbifold classes are defined by the average in \eqref{ew-comparison-chern}, and the cohomology equality uses the duality of the smooth \(U\).
We use the natural transformation \(C_*\) from constructible functions to Chern--Schwartz--MacPherson classes underlying \cite{Ohmoto}.
For a fixed ball, the identity from the Euler sequence holds on its neat cover, with its own dimension and the restricted pullback of the class \(L\) of the original lattice.
Set the constructible function to zero on strata of codimension greater than \(i\).
They do not change the dimension \(d-i\) component, and multiplying by \(M_i\) gives an integral function and integral Chern class.
For a removed support of codimension \(c\), local cohomology is Borel--Moore homology in degree \(2d-q\), so vanishes for \(q<2c\); this gives the stated extension.
The formula for the canonical class when all reflections have order two gives \(\mathcal T\); the small projective morphism \(\pi\) of Lemma~\ref{sg:resolution} has no exceptional divisor.
The low rank type \(A\) group orders give 12 and 24.
Substitution into \(12c_2^{\rm orb}\) and \(24c_3^{\rm orb}\) gives the two fractions.
For odd \(p\mid N\), the second fraction in \eqref{sc:chern-divisibility} forces \(3v_p(b)\ge2v_p(N)\).
For \(v_2(N)\ge2\), use \(v_2(N-2)=1\); for \(v_2(N)=1\), the first fraction in \eqref{sc:chern-divisibility} forces \(b\) even.
\end{proof}

\begin{proposition}\label{sc:root-theta}
For \(R=A_r,D_r,E_6,E_7\), use \(A_\tau=Q(R)\otimes(\mathbb C/(\mathbb Z+\tau\mathbb Z))\) with the polarization associated with the root lattice and with the equivariant structure for the Weyl group that is trivial at the origin.
Choose basic homogeneous polynomial invariants \(p_1,\ldots,p_r\) of the Weyl group, put \(d_i=\deg p_i\), and let \(m_i\) be the degrees of the corresponding generators of the ring of invariant sections, with distinguished generator \(\phi_0\) of degree one.
Then \(\phi_0(0)\ne0\), and the first derivatives of invariant sections of degree at most \(j\), divided by the corresponding powers of \(\phi_0\), project isomorphically onto \(\langle dp_i:m_i\le j\rangle\).
The statement holds, by lifts of translations to the theta group, at all common fixed points at which the Weyl group acts trivially on the fiber of the line bundle of the polarization.
The statement also retains the components along the differentials of basic invariants of higher degree and respects \(z\mapsto-z\).
In type \(D\), the direction of \(P_0\) is distinguished from the other invariant directions by its sign character.
\end{proposition}

\begin{proposition}\label{sc:parabolic-flags}
At the inclusions in the following table, use increasing degrees \(k_i\) of the polynomial invariants on the normal space, the coordinate \(v=u^e\) on the coarse quotient of the axis, and the frames \(p_i/u^{\epsilon_i}\) of the normal lines over the punctured axis.
Coordinate changes in the target alone make the normal differential upper triangular with the indicated successive orders:
\[
\begin{array}{c|c|c|c}
P\subset G&e&(\epsilon_i)&\text{orders}\\ \hline
D_5\subset E_6&1&0&(3,2,3,3,4)\\
E_6\subset E_7&2&k_i\bmod2&(2,2,2,2,3,3)\\
E_7\subset E_8&2&k_i\bmod2&(3,3,3,4,4,5,6)\\
G_{33}\subset G_{34}&6&(-2,0,-2,0,0)&(1,2,2,3,4).
\end{array}
\]
For each inclusion, the last row of the inverse differential has pole order at most the last entry of the tuple of orders for that inclusion; on the kernel of that row, the preceding row of the inverse differential has pole order at most the preceding entry of that tuple.
\end{proposition}
The degrees \(k_i\) in the four rows are, respectively, \((2,4,5,6,8)\), \((2,5,6,8,9,12)\), \((2,6,8,10,12,14,18)\), and \((4,6,10,12,18)\).
The frames and successive orders in the table refer to these same directions, rather than to new coordinates.
Lemma~\ref{sn:d5-flag} uses the first row; Theorem~\ref{st:ten-main} uses the second, Proposition~\ref{sh:eleven} the third, and Lemmas~\ref{sn:compact-g34} and~\ref{sn:nonuniform-g33} the fourth.
At the ends, Proposition~\ref{sc:root-theta} identifies the same directions with the successive projections of the first derivatives of invariant sections; Lemma~\ref{sa:g33-theta} supplies the exponents \((1,3,3,4,6)\) of the powers of \(q_*\) used when dividing the coefficients for the fourth row.
Appendix~\ref{sa:flags} gives the proofs of both statements, including the complete flags of ideals in the local algebras that justify elimination without changing these directions.

\subsection{\texorpdfstring{Curves without ends obtained from disks with full generic stabilizer \(G(m,p,r)\)}{Curves without ends obtained from disks with full generic stabilizer G(m,p,r)}}
\begin{lemma}\label{sc:compact-monomial}
If \(n=r+4\), \(r\ge5\), and \(X_\mu\) is the Baily--Borel compactification of an arithmetic ball quotient, then there is no curve obtained from a disk fixed by a finite reflection group, without ends and with full generic stabilizer \(G(m,p,r)\), whose points with a larger stabilizer are all of the types allowed in Proposition~\ref{sc:monomial-comparison}.
The ambient lattice may be cocompact or nonuniform.
\end{lemma}
\begin{proof}
The image of this compact curve lies in the smooth interior of the coarse ball quotient, where the small projective morphism \(\pi:W\to X_\mu\) is an isomorphism by Lemma~\ref{sg:resolution} and Proposition~\ref{sg:arithmetic}.
Thus all calculations take place on \(W\), even if the ambient quotient has cusps elsewhere.
Put \(j=mt\in\mathbb Z_{>0}\), \(A=r(r-1)/2+r/p\), \(\eta=r-1+1/p\), and \(k_0=k/m\).
Since the full normalizer preserves the indecomposable direction of degree \(k\), Proposition~\ref{sc:monomial-comparison} gives
\[
 (m-1)R+m\sigma=2j+2m(1-g),\quad
 a=2-2g-Aj+\eta R,\quad h=R-k_0j .
\]
For \(g\ge1\), \(R\le2j/(m-1)\), hence \(h<0\).
For any map to the target associated with the invariant of greatest degree, the coefficient on the tangent direction of the curve has negative degree \(R/m-(k_0-2/m)j-\sigma\).
After division by the zeros forced by the differential, the successive coefficients for the elementary symmetric invariants have degrees \(-(k_0-i)j<0\).
For \(p>1\), the coefficient for \(P_0\) has, after this division, degree at most
\[
 (1-1/p)R-(r-1-r/p)j
 \le((r+2)/p-r+1)j<0 .
\]
Thus the evaluation map from the Hom space has rank at most 1.
Proposition~\ref{sg:negative} gives exactly two coincident sections for the quotient \(D\), so \(a=re+2\delta>0\).
But
\[
 a\le(2+2m\eta/(m-1))(1-g)-(A-2\eta/(m-1))j<0 ,
\]
since \(A-2\eta/(m-1)\ge(r-1)(r-4)/2+(r-2)/p>0\).
Consequently \(C=\mathbb P^1\).
It follows from the coordinate \(g_1=v\) at the points with stabilizer \(G(m,p,r+1)\) and from the invertible normal differential at the points with stabilizer \(G(m,p,r)\times\mu_e\) that the curve is immersed.
At every stalk the flag in the directions of the elementary symmetric invariants splits from the direction of \(P_0\), so Lemma~\ref{sc:normal-splitting} gives
\[
 N_f=
 \begin{cases}
 \bigoplus_{i=1}^r\mathcal O(R-ij),&p=1,\\
 \bigoplus_{i=1}^{r-1}\mathcal O(R-ij)\oplus\mathcal O((R-rj)/p),&p>1 .
 \end{cases}
\]
If summands of negative degree occur, Propositions~\ref{sg:negative} and~\ref{sg:kernel} make all of them \(\mathcal O(-e)\) and every other summand of degree at least \(e\).
Two summands of negative degree associated with elementary symmetric invariants have different degrees and are impossible.
For \(p>1\), since the degrees \(R-ij\) decrease with \(i\), only the summand associated with the \((r-1)\)st elementary symmetric invariant can have negative degree; such a summand forces a summand of the same degree associated with \(P_0\), hence
\[
 j=(p-1)e,\qquad R=[r(p-1)-p]e .
\]
For \(p=2\), the summand associated with the \((r-2)\)nd elementary symmetric invariant has degree zero, a contradiction.
For \(p\ge3\), the Riemann--Hurwitz formula for \(C\) gives \(ce\le2m\), but
\[
 c=(m-1)[r(p-1)-p]-2(p-1),\quad
 c-2m\ge(4p-7)m-6p+7\ge4+(p-3)(4p-1)>0 .
\]
The remaining negative possibilities force \((r-1)j<R<rj\).
For \(m\ge3\) the bound from the Riemann--Hurwitz formula for \(C\) contradicts this; for \(m=2\), it leaves only \(r=5,6\), \(j=1\), and there is no integral \(R\) in that interval.
Thus \(N_f\) is nef, \(R\ge rj\), and \([r(m-1)-2]j\le2m\).
This forces \(m=2\), \(j=1\), \(r=5\) or 6.
For \(r=6\), \(R=6\), \(a=17\), contradicting the parity of \(a=\sum_im_i\) for ten points.
For \(r=5\), the only integral pairs are \((R,a)=(5,12),(6,17)\), with \(t=1/2\).
Nefness of \(N_f\) permits at most one negative degree \(m_i\).
With none, Proposition~\ref{sg:mori} decomposes the curve into coordinate curves, and \eqref{sc:coordinate-bound} gives \(t\ge a/7\).
With one, \(a=7e+2\delta\); the first pair is impossible, and the second has \(e=1,\delta=5\), giving \(t\ge2\delta/7=10/7\).
Both contradict \(t=1/2\).
\end{proof}

\section{Nine points}
\label{sn:nine}

Let $n=9$, and retain arbitrary weights $0<\mu_i<1$ of sum two.
Suppose that $X=X_\mu$ is the Baily--Borel compactification of an arithmetic quotient of $\mathbb B^6$.
Use the smooth GIT quotient \(W=X_\nu\) obtained by perturbing \(\mu\) to weights \(\nu\) in a nearby chamber, and the small projective morphism $\pi:W\to X$ of Lemma~\ref{sg:resolution}, and write $H=-K_W$.
For a curve $f:C\to W$ obtained from the full normalizer in the original lattice, put
\[
 t=\deg f^*L>0,\qquad a=\deg f^*H,\qquad g=g(C),\qquad m_i=\deg M_i.
\]
From Section~\ref{sn:nine} on, all line bundles for basic invariants are the line bundles of their invariant sections on the coarse curve, with their finite characters.
We allow curves without ends and of arbitrary genus.
The proof uses Lemma~\ref{sg:curve-bound} repeatedly, without proving separately for each reflection group that \(K_W\) plus a multiple of \(L\) is nef.

\subsection{Two bounds for complete curves in the cocompact case}

\begin{lemma}\label{sn:compact-curve}
Suppose that the arithmetic quotient of $\mathbb B^6$ is compact.
For a nonconstant complete curve, if $m_i\ge0$ for every $i$, then $a\le7t$.
If exactly one degree \(m_i\) is negative, write
\[
 m_0=-e,\qquad m_j=e+2\delta_j\ (j\ne0),\qquad \delta=\sum_{j\ne0}\delta_j.
\]
Then $a=7e+2\delta$ and $7t>6e+2\delta=a-e$.
\end{lemma}

\begin{proof}
By~\eqref{sc:coordinate-bound}, a general coordinate curve $\gamma_i$ satisfies $7L\gamma_i\ge2$.
Apply the decompositions of \([f_*C]\) obtained from the degrees $m_i$ in Proposition~\ref{sg:mori}.
They give the first assertion and reduce the second to $L\beta_0>6/7$.
Choose a maximal short set $J\ni0$, and set $q=|J|\le6$.
The same decomposition gives
\[
 [\beta_0]=[\beta_J]+\sum_{j\in J\setminus\{0\}}[\gamma_j].
\]
For $q=1$, Proposition~\ref{st:p6} gives $L\beta_0\ge1$.
For $q=2,3$, Proposition~\ref{st:compact-projective} gives, respectively,
\[
 L\beta_0\ge\frac23+\frac27=\frac{20}{21},
 \qquad
 L\beta_0\ge\frac13+\frac47=\frac{19}{21}.
\]
For $q\ge4$, ampleness gives $L\beta_J>0$, hence $L\beta_0>2(q-1)/7\ge6/7$.
\end{proof}

\begin{lemma}\label{sn:compact-g34}
If the arithmetic quotient of $\mathbb B^6$ is compact, it has no point with full stabilizer $G_{34}$.
\end{lemma}

\begin{proof}
If a point with full stabilizer $G_{34}$ exists, take a $G_{33}$ disk through it and the corresponding $G_{33}$ curve.
Let $R>0$ count $G_{34}$ points and take $\sigma$ over the points with stabilizer $G_{33}\times\mu_e$.
The line bundle associated with the invariant of degree six gives $j=6t\in\mathbb Z_{>0}$.
The Riemann--Hurwitz formula for \(C\), equation~\eqref{ew-curve-determinant-equation}, and the row of the inverse differential associated with the invariant of greatest degree give
\begin{equation}\label{sn:g33-compact-data}
 \frac56R+\sigma=2t+2-2g,\quad
 a=2-2g-50t+\frac{38}{3}R,\quad h=-18t+4R.
\end{equation}
The degrees of the five successive coefficients for a map to a line bundle of degree $d$ are
\begin{equation}\label{sn:g33-compact-coefficients}
 d+4t-\frac43R,\quad d+6t-2R,\quad
 d+10t-\frac73R,\quad d+12t-3R,\quad d+18t-4R.
\end{equation}
These are the degrees of the line bundles for the basic invariants of the normal representation, $(-4t+R/3,-6t,-10t+R/3,-12t,-18t)$, corrected by the successive zeros $(1,2,2,3,4)$ of Proposition~\ref{sc:parabolic-flags}.
If $g\ge1$, then $a,h<0$ and the first two coefficients for $d=h$ vanish.
Proposition~\ref{sg:negative} gives $\ell\le4$ and therefore $a=(9-2\ell)e+2\delta>0$, a contradiction.
If $g=0$ and $j\ge8$, the four lower coefficients vanish and the Hom space has dimension one.
The resulting quotient \(D\) of rank 1 requires $a+5h\ge0$, but
\[
 a+5h\le(402-308t)/5<0.
\]
For $1\le j\le7$, the conditions
\[
 R\le(2j+12)/5,\quad R\equiv2j\pmod3,\quad \sigma=0\text{ or }\sigma\ge1/2
\]
leave $(j,R)=(1,2),(2,1),(3,3),(4,2),(5,1),(6,3),(7,2)$.
The five rows with $h<0$ have $a<0$ and a Hom space of dimension one, again impossible.
The other rows are
\[
 (t,R,\sigma,a,h)=(1/6,2,2/3,19,5),\qquad(1/2,3,1/2,15,3).
\]
Formula~\eqref{sn:g33-compact-coefficients}, including the coefficient on the tangent direction of the curve, excludes every quotient onto a line bundle of negative degree.
Thus at most one $m_i$ is negative.
If none is negative, Lemma~\ref{sn:compact-curve} contradicts $a\le7t$.
Otherwise $a=7e+2\delta$ gives $(e,\delta)=(1,6)$ or $(1,4)$, and the weaker inequality $t>2\delta/7$ already contradicts the values $t=1/6$ and $t=1/2$.
\end{proof}

\subsection{\texorpdfstring{Successive kernels for $D_5$}{Successive kernels for D5}}

The sequence of kernels in Lemma~\ref{sn:d5-flag} is used in both the cocompact and nonuniform cases.

\begin{lemma}\label{sn:d5-flag}
Let a curve obtained from the full normalizer have full generic stabilizer $D_5$, larger irreducible groups $E_6,D_6$, and other interior groups $D_5\times\mu_e$ acting trivially on the root space.
Allow ends with the minimal $D_5$ polarization, namely the polarization associated with the root lattice in Proposition~\ref{sc:root-theta}, generators of degrees $(1^4,2^2)$, trivial character on the fiber of \(\mathcal V_*\), and width one.
Write $R,D$ for the reduced divisors of these $E_6$ points and $D_6$ points, using the same letters for their degrees.
Here $\sigma$ is taken over the points with stabilizer $D_5\times\mu_e$.
Set $E=f^*T_W$ and $T_f=\operatorname{sat}_E df(T_C)$.
Then $T_f=T_C(R)$ and $E/T_f$ has successive quotient line bundles $Q_2,Q_4,Q_5,Q_6,Q_h$ of degrees
\begin{align}
 q_2&=-2t+3R+D+s,&q_4&=-4t+2R+D+s,\\
 q_5&=-5t+3R+D/2+s,&q_6&=-6t+3R+D+2s,\\
 \label{sn:d5-degrees}
 h&=-8t+4R+D+2s.
\end{align}
In particular $E\twoheadrightarrow Q_h$, and
\begin{equation}\label{sn:d5-identities}
 2g-2+s+D/2+\sigma=2t,\quad
 a=4h+5t+D+\sigma,\quad
 \deg T_f=(h+D)/4+\sigma+s/2.
\end{equation}
If $R>0$, then $h>0$ and $E$ is nef.
\end{lemma}

\begin{proof}
The line bundles associated with these directions have degrees $-2t,-4t,-5t-D/2,-6t,-8t$.
The term $-D/2$ comes from the frame $P_0/u$ on the coarse quotient at a $D_6$ point, with its character included.
Here $P_0=x_1\cdots x_5$ is the invariant of degree five of $D_5$, while the invariant of degree six of $D_6$ is $Z=uP_0$ in~\eqref{sc:monomial-matrix}; since the residual action of $D_6$ is $(u,P_0)\mapsto(-u,-P_0)$, the frame $P_0/u=Z/v$, $v=u^2$, is invariant, and $d(P_0/u)=v^{-1}dZ$ on the axis gives its simple pole.
The full normalizer preserves their flag ordered by increasing degree.
Proposition~\ref{sc:parabolic-flags} gives successive zero orders $(3,2,3,3,4)$ at $E_6$.
At $D_6$, write $P_0=u\xi$; after removing the tangent column, the matrix of the normal differential is
\[
 \begin{pmatrix}
 v&1&0&0&0\\0&v&0&1&0\\0&0&v&0&0\\0&0&0&v&1\\0&0&0&0&v
 \end{pmatrix},
\]
so every quotient gains one zero.
At an end the distinguished generator \(\phi_0\) of Proposition~\ref{sc:root-theta} is nonzero and has degree one.
After its removal, the matrix is $\operatorname{diag}(q,q,q,q^2,q^2)A(q)$, where the leading $3\times3$ block of $A(0)$ is invertible.
The least orders of its minors using the first one, two, three, four, and five columns are $1,2,3,5,7$; their differences give $(1,1,1,2,2)$.
These are orders of minors from the initial columns, not merely Smith exponents after an arbitrary change of basis of the domain.
The differential on the $E_6$ axis has one common zero, and the differential on the coarse quotients of the other axes is invertible.
This proves the flag and~\eqref{sn:d5-identities}.

We record the local obstruction that also treats $h=0$.
Use coordinates $T_2,T_5,T_6,T_8,T_9,T_{12}$.
The $E_6$ axis is
\[
 T(u)=(-3u^2,3u^5,-3u^6/2,-3u^8/4,u^9,-3u^{12}/16).
\]
By Appendix~\ref{sa:d5-image}, the image of its ambient local ring is $\mathcal A=\mathbb C\{u^2,u^5\}$, whose missing nonnegative exponents are $1,3$.
We divide the rows of the inverse differential for the descending saturated quotients over \(\mathbb C\{u\}\) by the common orders of the zeros and poles of their entries; the resulting rows are
\begin{equation}\label{sn:d5-primitive-rows}
\begin{aligned}
 C_8&=dT_{12}+u^3dT_9/2+u^6dT_6/4,\\
 C_6&=dT_9+u\,dT_8+u^3dT_6+u^4dT_5/2+u^7dT_2/4,\\
 C_5&=dT_8+u^3dT_5/2+u^6dT_2/4,\\
 C_4&=dT_6+u\,dT_5+u^4dT_2,\qquad
 C_2=dT_5+5u^3dT_2/2.
\end{aligned}
\end{equation}
They annihilate $T'(u)/u$ and have a unit minor in columns $T_{12},T_9,T_8,T_6,T_5$.
We identify the successive kernels of these rows by the flag of ideals in the local Tjurina algebra used in Proposition~\ref{sc:parabolic-flags}.
No unit multiple of $C_8$ is $\mathcal A$-valued, since its coefficient ratio is $u^3/2\notin\mathcal A$.
If a regular row on the ambient quotient annihilates the tangent direction and the quadratic invariant, then the row has an extra zero in the coefficient of the invariant of degree four: compare the missing $u$ term in its $dT_5$ coefficient.
If the row also kills the invariant of degree four, comparison of the missing $u^3$ terms in $dT_6,dT_5$ shows that the value of the row at zero lies in $\mathbb C dT_{12}$.

Suppose $R>0$ and $h\le0$.
For a row whose target line bundle has degree $q\le h$, the degrees of the first two coefficients are bounded by
\[
 q-q_2\le3h/4-2R-3D/4-s/2<0,\qquad
 q-q_4\le(h-D)/2\le0.
\]
Equality in the second bound requires $q=h=D=0$.
If $h<0$, the independent quotients associated with pairs of coincident sections in Proposition~\ref{sg:negative} all have values in $\mathbb C dT_{12}$, hence $\ell=2$.
The quotient associated with the invariant of greatest degree factors through this pair quotient with a unit coefficient, contradicting $u^3/2\notin\mathcal A$.
For $h=0$, the same argument first excludes quotients onto line bundles of negative degree.
Two negative degrees $m_i$ would supply one; a unique negative degree $m_i$ would make the row associated with the invariant of greatest degree zero on the columns for the other labels, whose degrees are positive, and infinitesimal projective transformations then kill its remaining column.
Thus all degrees \(m_i\) are nonnegative and the row associated with the invariant of greatest degree only uses labels \(i\) with \(m_i=0\).
Their distinct positions are disjoint.
Three positions trivialize the family and make the relevant row constant; one position makes it a constant sum of pair differences.
Both contradict the same coefficient ratio.
If exactly two positions remain, each position carries at least two marked points and gives an independent quotient of degree zero associated with that pair.
Their first coefficients vanish by degree.
The second vanish by degree when $D>0$, and by their extra zero when $D=0$.
Both pair quotients therefore have value in $\mathbb C dT_{12}$, contradicting independence.
Consequently $h>0$.

Finally~\eqref{sn:d5-identities} gives
\[
 q_2=h/4+2R+3D/4+s/2>0,\quad q_4=(h+D)/2>0,
\]
\[
 q_6=3h/4+D/4+s/2>0,\quad \deg T_f>0,
 \qquad q_5=3h/4+g-1+\sigma/2\ge-1/4.
\]
Since $h,q_5$ are degrees of line bundles, hence integers, $q_5\ge0$.
The flag has quotients of nonnegative degree and remains so on every finite cover; hence $E$ is nef.
\end{proof}

\begin{proposition}\label{sn:positive-target}
In the cocompact case the curve of Lemma~\ref{sn:d5-flag} cannot have $R>0$.
\end{proposition}

\begin{proof}
We have $h>0$, $E$ nef, and at most one negative degree \(m_i\).
Put $Y=2t-h-D-\sigma$.
If all degrees \(m_i\) are nonnegative, Lemma~\ref{sn:compact-curve} gives $Y\ge3h>0$.
With one negative degree $m_i=-e$, the eight columns for the other labels generate $E$, so $e\le h$ and
\[
 S'=7t+e-a>0,\qquad Y=S'+3h-e>0.
\]
The column for each label, written in a frame pulled back from $W$, has coefficients in $\mathcal A$.
For a column killed by $C_8$, successive use of the rows in \eqref{sn:d5-primitive-rows} and the missing exponents $1,3$ gives extra zero orders $(R,2R,2R,3R)$ in $Q_6,Q_5,Q_4,Q_2$.
The equation $C_8v=0$ gives $v_9(0)=0$, then $C_6v=0$ gives $v_8(0)=0$; the $u^2,u^3$ terms of $C_5v=0$ give $[v_8]_{u^2}=v_5(0)=0$.
The $u^3$ terms of $C_6v=0$ and $C_4v=0$ give $v_6(0)=[v_5]_{u^2}=0$.
The adjusted degrees are
\[
 q_6-R=-t-\sigma/2-Y/2,\quad q_5-2R=-5t/2-\sigma/4-Y/4,
\]
\[
 q_4-2R=-2t-h-\sigma-Y,\qquad q_2-3R=-h-\sigma-Y.
\]
All are negative; hence, if the row associated with the invariant of greatest degree kills the column for a label $i$ with $m_i\ge0$, this column lies in $T_f$.
Let $p$ count these columns and put $\tau=\deg T_f=(h+D)/4+\sigma$.
At a general point, the kernel of their evaluation consists of projective vector fields vanishing at complementary marked positions.
There is at least one such position, hence $p\le3$.
With a unique negative degree $m_i$ there are two distinct complementary positions, hence $p\le2$.
Individual columns are nonzero because a nonconstant ordered configuration has at least four distinct positions.

If all degrees \(m_i\) are nonnegative, put $A_7=7t-a\ge0$.
For $\tau\le h$ use $B_0=9h-a\ge0$; for $\tau>h$ use $B_3=6h+3\tau-a\ge0$.
But
\[
 5A_7+4B_0=-10t-9D-9\sigma<0,\qquad
 2A_7+B_3=-t-21h/4-9D/4<0.
\]
With one negative degree $m_i$, $T_p=(8-p)h+p\tau-e-a\ge0$, but
\[
 S'+T_p=-3t-3ph/4+(p/4-2)D+(p-2)\sigma<0\qquad(0\le p\le2).
\]
Every case is impossible.
\end{proof}

\subsection{\texorpdfstring{Nefness of $K_W+10L$ in the cocompact case}{Nefness of K+10L in the cocompact case}}

\begin{lemma}\label{sn:compact-chern}
Once we exclude the proper irreducible factors satisfying the condition on full interior stabilizers in Lemma~\ref{sg:fixed} and the compact $G(m,p,5)$ curves with $m\ge4$ or $(m,p)=(3,1)$, the divisor $K_W+10L$ is nef.
\end{lemma}

\begin{proof}
The remaining groups of full rank 6 are $A_6,B_6,D_6,G(3,3,6),E_6,G_{34}$.
Their reflections have order two; write $B$ for the reduced branch divisor.
Then $14L=2K_W+B$ is integral and the groups of rank 2 are $A_1^2,A_2,B_2$.
Set
\begin{equation}\label{sn:chern-combinations}
 C_1=H+7L=B/2,\quad C_2=c_2(T_W)+6HL+21L^2,
 \quad A_0=3C_1(C_1+L)-6C_2.
\end{equation}
For a normal neat cover $p:Y\to W$ with group $G$, Burnside's fiber count gives
\[
 \boldsymbol1_W=|G|^{-1}\sum_{g\in G}p_*\boldsymbol1_{\operatorname{Fix}(g)}.
\]
Apply the naturality of Chern--Schwartz--MacPherson classes used in Proposition~\ref{sc:chern-integrality}.
A ball of codimension $r$ fixed by a finite reflection group has Chern polynomial $(1-Lz)^{7-r}$.
The triangular identities in codimensions one and two give respectively $C_1,C_2$ of~\eqref{sn:chern-combinations}, and
\begin{equation}\label{sn:rank-two-cycles}
 48C_2=12Z_{A_1^2}+16Z_{A_2}+18Z_{B_2},\qquad
 4A_0=4Z_{A_2}+9Z_{B_2}.
\end{equation}
Here $Z_P$ are ordinary reduced fixed cycles.
For clarity, the first coefficients are $48$ times the proportions $1/4,1/3,3/8$ of elements with no nonzero fixed vector.
For the second identity, the coefficients of $B^2+2LB$ are $2,4,6$, obtained from $4\nu_P(\nu_P-1)/|P|$, where $\nu_P=2,3,4$ is the number of mirrors of $P$.
The mirror normal class is $-L$, so its diagonal self-intersection cancels with $2LB$.
These observations derive both identities from the same formula for Chern--Schwartz--MacPherson classes.

On a closed locus $Z_I=\mathbb P^k$ where the sections labeled by $I$ coincide, put $q=|I|$, $\alpha=9-2q$, and $L|_{Z_I}=\lambda H_0$ for the hyperplane class $H_0$.
The exact sequence~\eqref{sg:tangent}, written on $W$, gives $c_1(T_W)|=\alpha H_0$ and $c_2(T_W)|=(\alpha^2-7)H_0^2/2$.
A plane in $Z_I$ is the intersection of $q-1$ divisors where pairs coincide and $k-2$ first Chern classes of the bundles $M_j$ with $j\notin I$.
Thus~\eqref{sn:rank-two-cycles} pairs integrally with this plane, including excess intersections.
For $b=14\lambda\in\mathbb Z$, substitution gives
\[
 48\int_{\mathbb P^2}C_2=24(\alpha^2-7)+\frac{36}{7}b(b+4\alpha),\qquad
 4\int_{\mathbb P^2}A_0=84+\frac67b(b+3\alpha).
\]
Their difference gives $7\mid6\alpha b$ and then $7\mid6b^2$, hence $7\mid b$.
The $K_W$-negative lines in the loci $Z_I$ with $q=2,3,4$ have $H$-degrees $5,3,1$ and $L$-degrees at least $1/2$.
The corresponding ratios are at most $10,6,2$.
Coordinate curves have ratio at most seven by Lemma~\ref{sn:compact-curve}.
Proposition~\ref{sg:mori} proves the result; for $W=\mathbb P^6$ use Proposition~\ref{st:p6}.
The divisibility argument applies to the restrictions to the loci $Z_I$ and does not assert descent of $2L$ on the compact $W$.
\end{proof}

\begin{theorem}\label{sn:compact-terminal}
No nine-point quotient under consideration is a compact arithmetic ball quotient.
\end{theorem}

\begin{proof}
By Proposition~\ref{sg:arithmetic}, stabilizers are reflection groups, so the compact coarse quotient is smooth and $\pi$ is an isomorphism.
Replace an irreducible factor successively by larger irreducible factors of full point stabilizers.
If this replacement stops below rank 6, Proposition~\ref{sg:persistent} gives basic invariants of equal degree $\delta$ and a closed image $\mathbb P^{6-s}$ with normal bundle $\mathcal O(-1)^{\oplus s}$ and $\delta L=\mathcal O(1)$.
The classification of groups whose basic invariants have equal degrees first gives $s\le2$; no analogue of Lemma~\ref{sg:marked} for four labels is needed.
For $s=1$, $\delta\ge2$ contradicts the bound $\lambda\ge2/3$ for $k=5$ in Proposition~\ref{st:compact-projective}.
For $s=2$ the groups are $G(m,2,2)$, with even $m\ge4$, and $G_7,G_{11},G_{19}$, all with $\delta\ge4$, contradicting its bound $\lambda\ge1/3$ for $k=4$.
Thus this replacement, started from any reflection, reaches rank 6.
Lemma~\ref{sc:compact-monomial} excludes $G(m,p,5)$ for $m\ge4$ or $(m,p)=(3,1)$ before any bound on $a/t$ is used.
Lemma~\ref{sn:compact-chern} now gives $a\le10t$ on every complete curve.
Lemma~\ref{sn:compact-g34} excludes $G_{34}$, and Proposition~\ref{sn:positive-target} excludes $E_6$ through the disk whose full generic stabilizer is the parabolic subgroup $D_5$ of $E_6$.
Next $B_6,D_6,G(3,3,6)$ give $B_5,D_5,G(3,3,5)$ curves, whose points with a primitive stabilizer have just been excluded; Lemma~\ref{sc:compact-monomial} excludes these curves.
Only $A_6$ remains.
The nontrivial lattice generated by reflections has a reflection, so replacing this reflection successively by larger irreducible factors of full point stabilizers supplies an $A_5$ curve.
Proposition~\ref{sc:type-a} gives $a\ge11t+\sigma>10t$, a contradiction.
We did not assume that the coefficient field is imaginary quadratic or that $2L$ descends on the compact $W$.
\end{proof}

\subsection{The nonuniform proof: bounds fixed once}
\label{sn:nonuniform}

Now suppose that the arithmetic quotient is nonuniform.
At a genuine wall vertex the local slice is $\operatorname{Cone}(\mathbb P^{r_1}\times\mathbb P^{r_2},\mathcal O(1,1))$, $r_1+r_2=5$, $r_1,r_2\ge1$.
The smaller factor is retained in the smooth GIT quotient \(W=X_\nu\) obtained by perturbing \(\mu\) to weights \(\nu\) in a nearby chamber, so $\dim\operatorname{Exc}(\pi)\le2$ and its codimension is at least four.
The morphism \(\pi\) is an isomorphism near every smooth cusp.
Proposition~\ref{sg:arithmetic} applies since $12>7$; the coefficient field $K$ is imaginary quadratic and $w_KL$ is an ample line bundle on $X$, $w_K\in\{2,4,6\}$.
After reflections of order greater than two are excluded in the proof of Lemma~\ref{sn:parameters}, Proposition~\ref{sg:arithmetic} shows that $2L$ descends to a line bundle on the original $X$.

We use Lemma~\ref{sg:curve-bound} in the following form throughout.
For an ample line bundle $A_X=cL$, put $b=ct$.
Then
\begin{equation}\label{sn:curve-comparison}
 a\le7b;\qquad
 \operatorname{Hom}(f^*T_W,Q)\ne0,\ \deg Q\le0\quad\Longrightarrow\quad a\le5b.
\end{equation}
If the common negative degree $m_i$ of the coincident sections is $-e$, Proposition~\ref{sg:negative} and Lemma~\ref{sg:curve-bound} give
\[
 a=(9-2\ell)e+2\delta,\quad e+\delta\le b,
\]
and every target of degree less than $e$ kills the summand $F$ of Proposition~\ref{sg:kernel}.
Thus, after $2L$ descends, the two bounds in~\eqref{sn:curve-comparison} are $a\le14t$ and $a\le10t$.
These statements hold for each complete curve separately, also for $W=\mathbb P^6$; no nefness of a divisor on $W$ is used.

\begin{lemma}\label{sn:enlargement}
Every irreducible reflection factor lies in a chain of successively larger irreducible factors of full point stabilizers reaching rank 5 or 6.
Every disk fixed by an irreducible group of rank 5 avoids genuine wall vertices, so all its ends are smooth ambient cusps.
\end{lemma}

\begin{proof}
If the chain of larger irreducible factors of full point stabilizers stops in rank $s\le4$, every full interior stabilizer along the corresponding fixed ball has the form \(T\times J_x\) required in Proposition~\ref{sg:persistent}.
Its image is a closed $Z_X=\mathbb P^{6-s}$ with equal degrees $\delta$, and classification gives $s\le2$.
The strict transform $Z\to Z_X$ has no exceptional divisor since $\dim Z\ge4$ and $\dim\operatorname{Exc}(\pi)\le2$.
A projective small morphism to a factorial target is an isomorphism: push down a relatively ample Cartier divisor and pull it back; their difference would be an exceptional divisor.
Thus $Z=Z_X$, $N_{Z/W}=\mathcal O(-1)^{\oplus s}$, and $\delta L|_Z=\mathcal O(1)$.
The common fixed space \(V_0\) of the factor at which the chain stops is defined over the original imaginary quadratic field and has dimension at least three.
Meyer supplies a rational isotropic line in that space, so the quotient \(\mathbb B(V_0)/\Pi\) of Lemma~\ref{sg:fixed} is nonuniform.
Descent of the original $w_KL$ then gives $\delta\mid w_K$.
Let \(P_\delta\subset\Sym^\delta E^*\) be the subspace of invariants of \(T\), and put \(U_\delta=P_\delta^*\); since all basic invariants of \(T\) have degree \(\delta\), \(P_\delta\) is their span.
By Lemma~\ref{sg:fixed}, the normalizer, with the action of \(C_0\) retained, acts on \(U_\delta\) through a finite group.
On a finite cover that makes these finite coefficients trivial, \(U_\delta\) gives a trivial flat bundle, and the finite character remaining in its combination with \(L^{-\delta}\) is retained.
The lowest invariant term of the ideal of \(Z\) gives there a generically invertible map
\[
 L^{-\delta}\otimes U_\delta\longrightarrow N_{Z/W}.
\]
After canceling equal degrees it is, on this cover, a map from the trivial flat bundle with fiber \(U_\delta\) to the pullback of $\mathcal O_Z^{\oplus s}$.
On complete curves away from cusps its coefficients become constant after a finite cover.
Interior inertia is therefore trivial; simple connectedness and Hartogs make the matrix constant and invertible on $Z$, so the map descends to $Z$ and extends over it.
At a smooth ambient cusp every vanishing function and the lowest coefficient of its expansion in the normal directions are divisible by \(q_*\), a contradiction.
All cusps of $Z$ consequently lie at genuine wall vertices.
By surjectivity on local rings and the generators of degree one in Theorem~\ref{sg:segre}, the full cusp stabilizers of \(Z\) and of the ambient quotient contain the same central translations: if one of these two groups of central translations had index greater than one in the other, every nonconstant restricted generator would vanish.
The induced smooth cusps therefore have only generators of degree one and trivial character on the fiber of \(\mathcal V_*\).

In rank 2, the possibilities are $G(4,2,2),G(6,2,2)$.
For this factor $T$ of rank 2, the normal directions corresponding to the two basic invariants of $T$ of degree $\delta$, with $\delta=4$ for $G(4,2,2)$ and $\delta=6$ for $G(6,2,2)$, require every larger point stabilizer to have two basic invariants of degree $\delta$.
The lift of \((F_P)_x\) in Lemma~\ref{sg:cusp-lift} permits this test at a point of a cusp that is a genuine wall vertex.
Consider the actions of $G(e,1,k)=\mu_e^k\rtimes S_k$ on $E^k$, by automorphisms of an elliptic curve $E$ fixing the origin in each coordinate and by permutations of the coordinates, the actions of $S_{k+1}$ by permutations on $\{x\in E^{k+1}:x_1+\cdots+x_{k+1}=0\}$, and their products; these are the two series of~\cite[Theorem~1.1(4)(a),(b)]{ALA1} and, in dimension 2, of~\cite[Theorem~1.1(3)(a),(b)]{ALAQ}.
For the point stabilizers of these actions, the irreducible factors of the reflection representation with the fixed part removed are symmetric groups and groups $G(e,1,k)$, and within each factor the basic invariants have distinct degrees, namely \(2,\ldots,k+1\) for \(S_{k+1}\) and \(e,2e,\ldots,ke\) for \(G(e,1,k)\).
Among the actions of finite groups on abelian varieties of dimension 2 with smooth quotient, only the action of $G(4,2,2)$ remains, and this excludes $\delta=6$.
For $\delta=4$, the induced groups on $\mathbb P^4$ are $A_3,B_3,G(4,1,3)$.
Over a general point of the branch divisor of the induced quotient of $\mathbb B^4$, the full stabilizer in the original ambient group is $G(4,2,2)\times\mu_e$, $e=2,4$, where $G(4,2,2)$ fixes the fixed subspace pointwise and $\mu_e$ acts transversely to a mirror in this subspace.
Let \(B_2,B_4\) be the reduced branch divisors corresponding to these additional factors \(\mu_2,\mu_4\), respectively.
With these branch divisors, the formula for the canonical class of $\mathbb P^4$ becomes $2B_2+3B_4=25H_0$, so $B_4\ne0$.
Meyer, applied in that quotient of a fixed ball of higher dimension by its full normalizer, yields a $G(4,1,3)$ cusp and the disk at this cusp fixed by $F_P=G(4,1,3)$.
Writing $u=4t$, its equations are
\[
 46t=4s+15R/4+\sigma,\qquad -12t+s+R\ge4t.
\]
The first right side is at least $15(s+R)/4\ge60t$, impossible.

For rank 1, $\delta=4,6$ are excluded by Propositions~\ref{st:p5-quarter} and~\ref{st:p5-sixth}.
For $\delta=2,3$, put $u=\delta t$ on $Z=\mathbb P^5$.
All induced cusps have five generators of degree one.
The field and cusp classifications leave $A_4$ and $G(m,1,4)$ with $m=2,3,4$ for $\delta=2$, and $m=2,3$ for $\delta=3$.
The case $m=6$ is excluded by the general inequality for \(P_0\) in rank 4 in the proof of Proposition~\ref{st:p5-sixth}:
\[
 a\ge(17p/6+1)h_P+30t+R+\sigma>6u
 \qquad(h_P\ge u,\ t=u/\delta,\ \delta=2,3).
\]
This uses that inequality itself, rather than the conclusion of the proposition for $L=H/6$.
A $G_{32}$ curve has no point with an irreducible stabilizer of rank 5, and it has no ends, since the actions of the groups \(F_P\) whose rings of invariant sections are generated in degree one have no full stabilizer $G_{32}$.
Equation~\eqref{ew-curve-determinant-equation} and the Riemann--Hurwitz formula for \(C\) give $a=2-2g-84t<0$, since $t=u/\delta\ge1/3$, contrary to $a=6u>0$.
For $G(m,1,4)$, equation~\eqref{ew-comparison-canonical} and Lemma~\ref{st:euler} give
\[
 (10m+2+6\delta)t=5s+(5-1/m)R+\sigma,\qquad
 -4mt+s+R\ge\delta t.
\]
The difference between the forced lower coefficient and the left coefficient is
\[
\begin{array}{c|ccc}
 &m=2&m=3&m=4\\\hline
 \delta=2&11&64/3&63/2\\
 \delta=3&19/2&20&\text{absent}
\end{array}
\]
and is positive.
Curves of types $D_4,G(3,3,4),G(4,2,4),G(4,4,4),F_4$ have no ends, by a check of the orders and characters of their groups against the smooth action of \(F_P\).
For $F_4$, $(28+6\delta)t=2-2g$ contradicts $t\ge1/\delta$.
For $D_4$, with $R,M$ counting $G_{33},D_5$ points, the two degree formulas give
\[
 g=0,\quad u\le2,\qquad (36+96/\delta)u-12=28R+21M.
\]
Modulo seven excludes $\delta=2$; $\delta=3$ leaves $(u,t,R,M,\sigma)=(1,1/3,2,0,1)$.
Here the image is a line with normal bundle $\mathcal O(1)^{\oplus4}$.
The four line bundles for the invariants on the coarse quotient have degree $-2$, with zero orders $(1,0,2,1)$ and $(1,2,0,1)$ at the two $G_{33}$ points; the other inertia consists of two points with inertia of order two acting trivially on the root space.
This is exactly the row~\eqref{st:d4-residual}, whose exclusion by triality uses only these data.
The identity $S_{4a}S_{4b}=S_2S_6$ in that exclusion makes the two products of exchanged ratios equal to $1/12$; the determinant of the four linear columns is therefore zero.
This excludes $D_4,G_{33},D_5$ before the pure $A_4$ formula is used.
For $G(m,p,4)$ with $p>1$, put $D_0=6m+4m/p$, $k=3m$, and $\beta=4+1/p-1/m$.
The equations for a curve without ends give
\[
 (D_0+2+6\delta)t=\beta R+\sigma,\qquad R\ge(k+\delta)t.
\]
For $\delta=2$, the differences $\beta(k+\delta)-(D_0+2+6\delta)$ are $8,27/2,14$.
These correspond respectively to $(m,p)=(3,3),(4,2),(4,4)$.
For $\delta=3,(m,p)=(3,3)$ the difference is six.
All are positive.
The ample tangent bundle of $\mathbb P^5$ excludes irreducible factors of smaller rank satisfying the condition on full interior stabilizers in Lemma~\ref{sg:fixed}, so an $A_4$ curve remains.
Its target associated with the invariant of greatest degree satisfies $5\deg Q=(6\delta-9)t-\sigma<5\delta t$, contrary to Lemma~\ref{st:euler}.
This excludes every proper irreducible factor satisfying the condition on full interior stabilizers in Lemma~\ref{sg:fixed}.

Using the restrictions on fields of definition in the classification of finite reflection groups, we obtain
\[
\begin{array}{c|l}
\text{rank 5}&A_5,\ G(m,p,5),\ G_{33}\\
\text{rank 6}&A_6,\ G(m,p,6),\ E_6,\ G_{34},
\end{array}\qquad
m\in\{2,3,4,6\},\quad p\mid m.
\]
Every group of rank 6 has a parabolic subgroup of rank 5.
An irreducible root space of rank 5 cannot lie in one of the factors \(\mathbb C^{r_1},\mathbb C^{r_2}\) of \(\mathbb C^r\) at a genuine wall vertex, for which \(\{r_1,r_2\}\) is \(\{1,4\}\) or \(\{2,3\}\).
Thus every end of a curve fixed by a group of rank 5 is a smooth ambient cusp.
\end{proof}

\subsection{\texorpdfstring{Remove $G(6,p,5)$, $G(4,1,5)$, $G(3,1,5)$ before using $2L$}{Remove G(6,p,5), G(4,1,5), G(3,1,5) before using 2L}}

\begin{lemma}\label{sn:parameters}
For $T=G(6,p,5),G(4,1,5),G(3,1,5)$, there are no $T$ curves and no smooth cusps with $F_P=T$.
After $2L$ descends, the same holds for $G(4,2,5),G(4,4,5)$.
\end{lemma}

\begin{proof}
Apply Proposition~\ref{sc:monomial-comparison} with the periods at the ends, polarizations, and ages of the characters on the fibers.
For $p>1$, the degrees of the target line bundles for the elementary symmetric invariant of greatest degree and for \(P_0\) are
\[
 h=-4mt+R+\nu,\qquad h_P=(-5mt+R+\nu)/p,\qquad \nu=\sum_c\nu_c;
\]
for $p=1$ the target associated with the invariant of greatest degree is $h=-5mt+R+\nu$.
The age of the line bundle corresponding to \(P_0\) is retained separately.

For $m=6$, put $j=6t\in\mathbb Z_{>0}$ and use $A_X=6L$.
At every end,
\[
 (B_c-\nu_c/p,\nu_c)\in\{(5,1),(9,2),(13,3),(23,6)\},\qquad
 \operatorname{age}_c=\lceil\nu_c/p\rceil-\nu_c/p.
\]
It follows from the degrees $(2,3,6/p,6,6,6)$ of the generators before normalization, with their character congruences.
Set $\gamma=23/6+1/p$.
Since $B_c\ge\gamma\nu_c$, equation~\eqref{ew-comparison-canonical} gives
\[
 a\ge53j/6+R+\sigma+p\gamma h_P.
\]
Formula~\eqref{sn:curve-comparison} first gives $a\le7j$, then $h_P<0$ and $a\le5j$.
Every nonzero line target has degree at least $-j$.
The Riemann--Hurwitz formula for \(C\), using $\nu_c\le6$, gives
\[
 (\bar k-3)j+4R+6\sigma\le12(1-g),\qquad
 \bar k=5\ (p=1),\quad\bar k=4\ (p>1).
\]
Thus $g=0$, $j\le6$ for $p=1$ and $j\le12$ otherwise; also $s\le\lfloor2+j/3\rfloor$ and $5R/6+s\le2+j/3$.
Here $\sigma=0$ or $\sigma\ge1/2$, with no bound on the effective inertia orders in the definition of \(\sigma\).

We give the finite check of the remaining cases to specify its scope.
Let $V=\sum V_c$ be the sum of the orders of the zeros initially required at the cusps: $V_c=1$, except that the two $p=1$ end sums $11,16$ require $V_c=2,3$.
For a negative target $q=\min(h,h_P)$, the successive symmetric bounds are $q+ij-R-V$.
The product bound is $q+(5j-R)/p+\sum\operatorname{age}_c-V$.
Require $h,h_P\in\mathbb Z$, $h,h_P\ge-j$, and integrality of the degree of the line bundle corresponding to \(P_0\)
\[
 -5j/p-(1-1/p)R-\sum\operatorname{age}_c.
\]
For $-q\le e\le j$ and $2\le\ell\le6$, require $a=(9-2\ell)e+2\delta$, $e+\delta\le j$, and $\delta\ge0$ integral.
The coefficient bounds must allow the evaluation map from the Hom space to have rank $\ell-1$ and the Hom space to have dimension $(\ell-1)(e+q+1)$.
The independent inverse directions on $F$ also impose
\[
 \#\{0\le i<4:h+ij<e\}+\boldsymbol1_{h_P<e}\le\ell-1
\]
for $p>1$, with the directions of five elementary symmetric invariants and no term for \(P_0\) for $p=1$.
Substitution in this proved finite range gives no row for $p=1$.
All remaining rows have $R=g=0$, $a=e=j$, $\ell=4$, $h=-j$, $\delta=0$; their complete end table is
\[
\begin{array}{c|c|c|c|c}
 p&j&\{B_c\}&\sigma&h_P\\\hline
 2&2&\{26\}&5/3&-2\\
 2&3&\{29/2,26\}&1&-3\\
 2&3&\{11/2,10,26\}&0&-3\\
 3&6&\{29/3,29/3,25,25\}&0&-4\\
 6&6&\{24,24,24\}&1&-2\\
 6&9&\{27/2,24,24,24,24\}&0&-3.
\end{array}
\]
This includes curves without ends and every sign of $a$; positivity was not imposed.
For $p=3,6$, the line bundle corresponding to \(P_0\) is a direct summand because $R=0$.
Its degree strictly between $-j$ and zero contradicts $f^*T_W=\mathcal O(-j)^{\oplus3}\oplus F$, whose other summands have degree at least $j$.
For $p=2$, every row has an end with $B_c=26$.
Let $x_*$ be the fixed point of $\mu_6$ on the elliptic curve factor of the theta product, and let $\mathcal O([x_*])$ be the positive line bundle on this factor.
Choose a local coordinate $z$ at $x_*$ and a local frame of $\mathcal O([x_*])$ in which $\mu_6$ acts by a character.
For sections $u,x,y$ of its first, second, and third tensor powers, with $u$ vanishing simply at $x_*$, write $u=z(u_0+O(z^6))$, $x=x_0+x_1z^6+O(z^{12})$, and $y=y_0+y_1z^6+O(z^{12})$, where $u_0x_0y_0\ne0$.
These sections obey $y^2=4x^3-g_3u^6$, $g_3\ne0$, and
\[
 2y_1/y_0-3x_1/x_0=-g_3u_0^6/y_0^2\ne0.
\]
The derivatives in the radial direction and of the first elementary symmetric invariant are independent.
Once the coefficient for the first elementary symmetric invariant vanishes, all coefficients for the elementary symmetric invariants are divisible by the sixth power of \(q_*\).
Only the direction of one elementary symmetric invariant and that of \(P_0\) can remain, so the evaluation map from the Hom space has rank at most 2 instead of $\ell-1=3$.

For $m=4,p=1$, use $j=4t$ and $A_X=4L$.
The degrees $(d_1,d_2,2d_3,4,4,4)$ of the generators, $d_1,d_2\mid4$, $d_3\mid2$, give $B_c\ge19\nu_c/4$ for every torsor and linearization.
Consequently
\[
 a\ge19h/4+33j/4+R+\sigma.
\]
Starting with $a\le7j$ gives $h<0$, then $a\le5j$, $h\le-13j/19$, and $e\ge13j/19$.
We use the coordinates $Z,g_5$ of~\eqref{sc:monomial-matrix} and put $v=u^4$.
On $F$ the equations $dZ=v\,df_5$, $dg_5=df_5+v\,df_4$ reduce the order of the pole in the row of the inverse differential associated with $f_4$, giving a second independent target of degree $h+j<e$.
Thus $\ell\ge3$ and $a\le3j$, but $h\ge-j$ gives $a\ge7j/2$.

For $m=3,p=1$, put $j=3t$.
The integrality of $j$ comes from the line bundles for the elementary symmetric invariants, not from $6L$ alone.
After the case $m=6$ has been excluded, the degrees of the generators at the ends are $(1^6),(1^2,3^4),(1,3^5)$, with counts $x,y,z$.
Set $h=-5j+R+x+3y+3z$.
The equations give
\[
 a-5h=10j+2-2g-2y>4j,\qquad j+R+3\sigma\le6(1-g).
\]
Here $y\le2+2j/3-2g$ proves the strict inequality.
Lemma~\ref{sg:curve-bound} gives $g=0$, $1\le j\le6$, $h<e\le2j$, and factorization through the quotient $D$ of Proposition~\ref{sg:negative}.
After division by the zeros forced by the differential, the coefficient for the first elementary symmetric invariant has degree $-4j+2y<0$; its sole possible exception $j=1,y=s=2,R=0,e=2$ gives $a=13>10j$.
At $(1^2,3^4)$, use the cubic $E=\{u^3-v^3=w^3\}\subset\mathbb P^2$.
Let $u_i,v_i,w_i$ be its coordinate sections on the $i$th factor of the finite \'{e}tale cover $E^5$ of the abelian variety \(A_P\).
The sections $s_u,s_v$ pull back to $\prod_{i=1}^5u_i,\prod_{i=1}^5v_i$, respectively, and $z_i=w_i/v_i$ near the fixed point, where all $v_i$ are nonzero.
The local equation $(s_u/s_v)^3=\prod_{i=1}^5(1+z_i^3)$ makes the derivatives in the radial direction and of the first elementary symmetric invariant independent at every fixed point.
At $(1,3^5)$, the section of degree one for width one is nonzero.
Thus all coefficients for the elementary symmetric invariants are divisible by the third power of \(q_*\); after this division their degrees are $(i-5)j$.
The Hom space has dimension one, so $\ell=2$, $h=-e$, and $5e+\delta=5j+1-y$.
The only integer case with $2e\le j$ is $(j,e,y,R)=(2,1,3,0)$, forcing $\sigma=1/3$, which is impossible because each term of \(\sigma\) is at least \(1/2\).
Hence $2e>j$ and the next inverse target $h+j=j-e<e$ gives two independent rows through the quotient \(D=M\) of Proposition~\ref{sg:negative}.

These exclusions remove every reflection of order greater than two, including those in factors of \(F_P\) of smaller rank by Lemma~\ref{sn:enlargement} and the lift of \((F_P)_x=F_P\) in Theorem~\ref{sg:segre}.
The normalized lifts now have determinant $\pm1$, so a scalar $zI_7$ in their kernel satisfies $z^{14}=1$.
Its order also divides $w_K$, hence divides two.
Proposition~\ref{sg:arithmetic} shows that $2L$ descends to a line bundle, with the full cusp stabilizer included.

For $m=4,p=2,4$, we may therefore put $j=4t\in2\mathbb Z_{>0}$, $b=j/2$.
The line bundle corresponding to \(P_0\) gives $h_P=(h-j)/p$.
The degrees of the generators before normalization imply $5\nu_c+1-B_c\le4$ for $p=2$ and $\le5$ for $p=4$.
For $p=2$ the stronger bound uses the even degree $2d_3$ of the generator $(\prod_iy_i)^{d_3}$ of the ring of invariant sections, built from the product of the coordinate sections in the Gaussian model $x_1^4-x_2^4=y^2$ of~\cite[Example~3.10]{Rains} and measured in its grading with $\deg y=2$, before the scalar kernel is removed; this degree is not the polynomial degree of $P_0$.
Interchanging $x_1,x_2$ if necessary, write the six degrees of the generators before normalization as $d_1,d_2,2d_3,4,4,4$, where $d_1,d_2$ correspond to the branch points $x_1=0$ and $x_2=0$ of ramification index four.
For $p=2$, $d_1$ is even; if the greatest degree after normalization is four, the case in which all degrees are divisible by two is excluded, so $d_2$ is odd.
Hence \(\sum_ib_{c,i}=d_1+d_2+2d_3+12\ge2+1+2+12=17\), and \(B_c=\sum_ib_{c,i}\) because the width is one when \(\nu_c=4\).
Elimination of $s$ gives
\[
 a-5h\ge11j/2-6+6g+5R/2+4\sigma\quad(p=2),
\]
\[
 a-5h\ge25j/4-8+8g+3R+5\sigma\quad(p=4).
\]
Both exceed $2b=j$.
Thus $h<e\le j/2$, and the independent target line bundle corresponding to \(P_0\) also has degree less than $e$; hence $\ell\ge3$ and $a\le3j/2$.
For $p=2$, $h_P\ge-j/2$ gives $h\ge0$, contradicting the first bound.
For $p=4$, $h\ge-j/2$ gives $a-5h\le4j$, contradicting the second bound for $j\ge4$.
For $j=2$, $h=-1,0$ and $(h-2)/4$ is not integral.
All corresponding exclusions of cusps and groups of rank 6 use the lifts of the groups \((F_P)_x\) in Lemma~\ref{sg:cusp-lift}, preserving period and character.
\end{proof}

\subsection{The remaining five types}

\begin{lemma}\label{sn:nonuniform-g33}
There are no $G_{34}$ points, cusps with $F_P=G_{33}$, or $G_{33}$ curves.
\end{lemma}

\begin{proof}
The remaining groups \(F_P\) are $A_5,B_5,D_5,G(3,3,5),G_{33}$.
Primitivity forces every $G_{33}$ end to have $F_P=G_{33}$, trivial character on the fiber of \(\mathcal V_*\), and width one.
The degrees of its generators are $(1,1,3,3,4,6)$.
The four fixed points form the kernel of the polarization; lifts of translations to the theta group preserve the characters at all four, by Lemma~\ref{sa:g33-theta}.
The only larger interior groups are $G_{34},G_{33}\times A_1$.
Let $R$ count the former and take $\sigma$ over the latter, so $2\sigma\in\mathbb Z_{\ge0}$.
The equations are
\begin{equation}\label{sn:g33-nonuniform-data}
 2g-2+s+5R/6+\sigma=2t,\quad
 a=36(1-g)-16t-3R/2-17\sigma,\quad
 h=12(1-g)-6t-R-6\sigma.
\end{equation}
Before correcting at the cusps, we use line bundles of degrees $(-4t+R/3,-6t,-10t+R/3,-12t,-18t)$ for the basic invariants of the normal representation.
To bound the dimension of the Hom space, we successively divide the coefficients by powers of \(q_*\) with exponents $(1,3,3,4,6)$.
If $g\ge1$, after division by the zeros forced by the differential, the first two coefficients for the target associated with the invariant of greatest degree have degrees
\[
 10(1-g)-4t-3R/2-5\sigma,\qquad
 10(1-g)-2t-13R/6-5\sigma.
\]
These degrees and that of the coefficient on the tangent direction of the curve are negative.
The remaining three directions bound the rank \(\ell-1\) of the quotient from the \(\ell\) coincident sections in Proposition~\ref{sg:negative}, so $\ell\le4$, thus $a>0$, contrary to~\eqref{sn:g33-nonuniform-data}.
Hence $g=0$.

For $R>0$, put $j=2t$, $R=3r$, $q=2\sigma$; integrality of the degrees of the line bundles on the coarse curve makes these integers.
Then
\[
 r+q\equiv0\pmod2,\quad s=2+j-(5r+q)/2\ge0,\quad
 a=36-8j-(9r+17q)/2\le7j.
\]
If $r=1$, then $q$ is odd and $j\ge2$; if $r\ge2$, then $j\ge3$.
After division by the zeros forced by the differential, the degrees of the first two coefficients are strictly negative, and when $h<0$ the same bound $\ell-1\le3$ from Proposition~\ref{sg:negative} gives $a>0$.
When $h\ge0$, use $a=3h+j+9r/2+q/2>0$ instead.
For $r\ge3$, the inequality for $s$ gives $a\le52-49r/2-25q/2<0$.
Positivity leaves exactly
\[
 (t,R,s,\sigma,a,h)=(3/2,6,0,0,3,-3),\qquad(1,3,1,1/2,7,0).
\]
In the first row a third coefficient vanishes, so $\ell\le3$, $e\ge3$, and $a\ge3e\ge9$.
In the second, after division by the zeros forced by the differential, the five coefficients for a target $\mathcal O(-1)$ have degrees $(-2,-4,-1,-2,-1)$, so no quotient onto a line bundle of negative degree exists.
All degrees \(m_i\) have the same parity, and their sum is seven.
They cannot all be nonnegative, since then all nine would be positive odd integers.
Two negative degrees $m_i$ give a quotient onto a line bundle of negative degree.
With one, $7=7e+2\delta$ gives $e=1$, $\delta=0$, and $f^*T_W=\mathcal O(2)\oplus\mathcal O(1)^{\oplus5}$, which cannot map nontrivially to the quotient line bundle $Q_0=E_{18}(4D_R+6S)$ for the invariant of degree eighteen, where $E_{18}$ has degree $-18t$ and $D_R$ is the reduced divisor of the $G_{34}$ points; since $\deg Q_0=-18t+4R+6s=0$ and $C=\mathbb P^1$, we have $Q_0\simeq\mathcal O_C$.
Thus $G_{34}$ is excluded.

Now $R=0$.
The bound $a\le10t$ follows from $a\le14t$: it is automatic for $t\ge3/2$, and for $t=1/2,1$ the Riemann--Hurwitz formula for \(C\) leaves respectively $a\le0,3$.
The central scalar $\omega$ has characters $(\omega,1,\omega,1,1)$ on the normal directions.
Its projector preserves the local lattices of the normal bundle on the coarse curve and splits the quotient \(N_f^{\rm sat}\) of \(f^*T_W\) by the saturated image of \(df\) as $U\oplus V$ of ranks 2 and 3.
We have
\[
 \deg U=-14t+4s=8-6t-4\sigma<0,\qquad
 0\longrightarrow E_4(S)\longrightarrow U\longrightarrow E_{10}(3S)\longrightarrow0.
\]
The splitting of Proposition~\ref{sg:kernel} forces $U=\mathcal O(-e)^{\oplus2}$; a summand of $U$ of positive degree would make the determinant of $U$ nonnegative.
Thus $e=3t+2\sigma-4>0$ and $t+\sigma\le2$.
Integrality leaves $(t,s,\sigma,a,h,e)=(1,3,1,3,0,1),(2,6,0,4,0,2)$.
After division by the zeros forced by the differential, the coefficients of maps $V\to\mathcal O(-e)$ have degrees $(-4,-1,-1)$ or twice these values, so $\ell=3$.
The second row contradicts $a\ge3e$; the first gives $V=\mathcal O(1)^{\oplus3}$ and a nonzero map $V\to\mathcal O$ associated with the invariant of greatest degree, also impossible.
A $G_{33}$ curve without ends has $h=-18t$ and $a=2-2g-50t<0$, since $2t$ is a positive integer.
Only its last coefficient can survive, giving $\ell=2$ and $a\ge90t>0$.
\end{proof}

\begin{lemma}\label{sn:nonuniform-b5}
There are no $B_5$ curves or smooth cusps with $F_P=B_5$.
\end{lemma}

\begin{proof}
The remaining groups \(F_P\) force ends with $F_P=B_5$, trivial character, and width one.
The complete pairs are $(B_c,\nu_c)=(6,1),(9,2),(10,2),(11,2)$; the only larger irreducible stabilizer is $B_6$.
For $h=-10t+R+\sum\nu_c$, equation~\eqref{ew-comparison-canonical} gives
\[
 a\ge9h/2+13t+R+\sigma.
\]
First $a\le14t$ gives $h\le2t/9$ and $a-5h\ge116t/9>4t$.
Then Lemma~\ref{sg:curve-bound} gives $h<e$, $\ell\ge2$, $a\le10t$, and $h\le-2t/3$.
After division by the zeros forced by the differential, the lowest two coefficients have degrees $-8t+\nu-s,-6t+\nu-s$, with $\nu-s\le s\le2+2t-2g$.
They are negative for $t\ge1$; for $t=1/2$ use integrality of $h<0$ to get $R+\nu\le4$.
In each of the three models with $\nu_c=2$, the sections of degree one vanishing at a fixed point project injectively to the directions of the first two elementary symmetric invariants.
Use the four sections $x_1,\ldots,x_4$, with $x_j^2=\ell_j(T,U)$, where $[T:U]\in\mathbb P^1$ and the four linear forms have distinct zeros.
At the chosen fixed point, let $x_j=0$, choose $x_b\ne0$, set $z=x_j/x_b$, and let $c,d$ be the other two indices.
Set $\gamma_c=(x_c/x_b)(0)$ and $\gamma_d=(x_d/x_b)(0)$, and define $a_c,a_d$ by $\ell_c/\ell_b=\gamma_c^2(1+a_cz^2)$ and $\ell_d/\ell_b=\gamma_d^2(1+a_dz^2)$.
Their ratios are $x_c/x_b=\gamma_c(1+a_cz^2)^{1/2}$ and $x_d/x_b=\gamma_d(1+a_dz^2)^{1/2}$, with $a_ca_d(a_d-a_c)\ne0$.
Their three product sections have an invertible matrix of values and derivatives in degrees two and four, and this remains injective on the descending subspace.
Theta translation retains all torsors and linearizations at every fixed point.
Thus the two vanished coefficients give a second division by powers of \(q_*\); the $\nu_c=1$ model needs only the first such division.
The five remaining degrees are $(2i-10)t$, so the Hom space has dimension one and $\ell=2$.
Formula~\eqref{ew-git-degree} then yields
\[
 h+2t\le2e/3-2(R+\sigma)/9<e.
\]
On $F$, the row associated with the invariant of greatest degree vanishes and the $B_6$ equations $dg_6=v\,df_5$, $dg_5=df_5+v\,df_4$ leave only a simple pole in the row associated with the invariant of the second greatest degree.
Its target has degree $h+2t<e$.
Two independent rows cannot factor through the quotient \(D=M\) of Proposition~\ref{sg:negative}.
The lift of \((F_P)_x\) in Lemma~\ref{sg:cusp-lift} excludes the cusps with \(F_P=B_5\) by the same argument.
\end{proof}

\begin{lemma}\label{sn:nonuniform-e6}
There are no points with full stabilizer $E_6$.
\end{lemma}

\begin{proof}
If a point with full stabilizer $E_6$ exists, the $D_5$ curves through it have precisely the models of Lemma~\ref{sn:d5-flag}: $G_{33}$ has no reflection subgroup of full rank $D_5$, and the other possibilities for \(F_P\) have been excluded.
Since $R>0$, that lemma gives $h>0$ and $E$ nef.
If all $m_i\ge0$, the decomposition \([f_*C]=\tfrac12\sum_i m_i[\gamma_i]\) of Proposition~\ref{sg:mori} and integrality of $(2L)\gamma_i$ give $a\le4t$, contrary to $a=4h+5t+D+\sigma$.
Otherwise nefness leaves one negative degree $m_i=-e$, and the columns for the other labels generate $E$, giving $h\ge e$.
We need $(2L)\beta_i\ge2$.
If this degree were one, on the open set where the marked point labeled $i$ is distinct from the others the polarization would be $\mathcal O(1)$ on $\mathbb P^6$ minus linear subspaces of codimension two.
Represent $2L$ by multidegrees $c_i$, $\sum c_i=2q$, so $q-c_i=1$.
For $j,k\ne i$, we have $m_j,m_k\ge e$ by~\eqref{sg:tangent}, so the sections $\sigma_j,\sigma_k$ either coincide or have a bracket of positive degree $(m_j+m_k)/2$; in both cases they meet in some fiber, whose stability makes the pair $\{j,k\}$ allowable.
Hence, if the pair $\{i,j\}$ is allowable, every pair containing $j$ is allowable, so the coordinate curve \(\gamma_j\) of Proposition~\ref{sg:mori} exists; since $K_W\cdot\gamma_j=-2$ and $K_W$ is \(\pi\)-nef, \(\pi\) does not contract it, and $c_j=(2L)\cdot\gamma_j\ge1$.
For the other pairs, the bracket $[ij]$ is a unit on the stable locus; adding its multidegree makes all $c_j\ge1$, preserving $q-c_i=1$.
For $j,k\ne i$, the global invariant
\[
 [jk]\prod_{\ell\ne i}[i\ell]^{c_\ell-\delta_{\ell j}-\delta_{\ell k}}
\]
restricts to $z_j-z_k$, where $z_j$ is the affine coordinate of the marked point labeled $j$ when the marked point labeled $i$ is placed at infinity.
Normality and multiplication identify the section ring with a polynomial ring on seven generators of degree one.
Hence $(X,2L)=(\mathbb P^6,\mathcal O(1))$, excluded independently by Proposition~\ref{st:p6}.
The small projective morphism \(\pi\) to factorial $\mathbb P^6$ is an isomorphism, so this exclusion applies to $X$ as well as $W$.
The first decomposition in~\eqref{sg:curve-classes} now gives $t\ge e+\delta/2$, and
\[
 4h=7e+2\delta-5t-D-\sigma\le2e-\delta/2-D-\sigma<4e,
\]
contrary to $h\ge e$.
\end{proof}

\begin{lemma}\label{sn:nonuniform-d5}
There are no $D_5$ curves or smooth cusps with $F_P=D_5$.
\end{lemma}

\begin{proof}
Only points with stabilizer $D_6$ or $D_5\times\mu_e$ remain, and every cusp has $F_P=D_5$ and width one.
The line bundle $E_P$ associated with \(P_0\) has degree $-5t-D/2$ and defines a direct summand
\[
 E=N_P\oplus E',\qquad N_P=E_P(D+S),\qquad p=\deg N_P=-5t+D/2+s.
\]
At $D_6$, the frame on the coarse quotient is $\psi=P_0/u$, and the invariant of degree six is $Z=uP_0=v\psi$.
At an end its column is the odd section $q\theta_P$.
Divide that column and the corresponding row of the inverse differential by the same parameter; their composite is the identity.
Theta translation proves this at all sixteen fixed points, so the splitting is global.
The formulas become
\[
 2g-2+s+D/2+\sigma=2t,\quad
 p=2(1-g)-3t-\sigma=h/2-t,\quad a=8p+13t+D+\sigma.
\]
Every nonzero target has degree at least $-2t$, so $a-5p\ge7t>4t$.
Lemma~\ref{sg:curve-bound} gives $p<e$, $\ell\ge2$, $a\le10t$, and $p<0$.
The conclusion on direct summands in Proposition~\ref{sg:kernel} gives $p=-e$.
Therefore $g=0$ and $k=2-t-\sigma=2t-e\in\{0,1\}$.
For $k=0$, $e=2t$, $\delta=0$, $h=-e$, and $s+D/2=3t$.
After division by the zeros forced by the differential, the coefficients for the first two elementary symmetric invariants have degrees $(2i-8)t+s$, $i=1,2$, which are negative.
The projections of the first derivatives of invariant sections give the second division by powers of \(q_*\) at the cusps; the degrees of the coefficients for the four elementary symmetric invariants become $(2i-8)t$.
The directions of \(P_0\) and of the elementary symmetric invariants contribute to a Hom space of dimension at most two, hence $\ell=3$.
For the third elementary symmetric invariant, the inverse differential restricted to $F$ has a row with only simple $D_6$ poles and target $h+2t=0<e$.
Three independent rows would factor through the quotient \(D=M^{\oplus2}\) of Proposition~\ref{sg:negative}.
For $k=1$, $(t,\sigma,e,p,h)=(1,0,1,-1,0)$.
The independent rows associated with \(P_0\) and the elementary symmetric invariant of greatest degree give $\ell\ge3$.
The equations $D=8-2s$, $a=13-2s=(9-2\ell)+2\delta$, $\delta\le1$, leave $(s,D,\ell,\delta)=(4,0,3,1)$.
For the target $N_P=\mathcal O(-1)$, after division by the zeros forced by the differential, the first two coefficients have degrees $-3,-1$, and after the second division the coefficients have degrees $-7,-5,-3,-1$.
Only the direction of \(P_0\) contributes to Hom, of dimension one; the quotient \(D=M^{\oplus2}\) of Proposition~\ref{sg:negative} requires dimension two.
The lift of \((F_P)_x\) in Lemma~\ref{sg:cusp-lift} again excludes all cusps with \(F_P=D_5\).
\end{proof}

\begin{lemma}\label{sn:nonuniform-g335}
There are no $G(3,3,5)$ curves or smooth cusps with $F_P=G(3,3,5)$.
\end{lemma}

\begin{proof}
Only points with stabilizer $G(3,3,6)$ or $G(3,3,5)\times\mu_e$ remain; the cusps have width one and, for $F_P=G(3,3,5)$, generators of degrees $(1^3,3^3)$.
Integrality of the degrees of the line bundles on the coarse curve gives $3t\in\mathbb Z$, $-5t-2R/3\in\mathbb Z$.
Together with $2t\in\mathbb Z$, they imply $t\in\mathbb Z_{>0}$ and $3\mid R$; the Riemann--Hurwitz formula for \(C\) makes $\sigma$ integral.
The formulas are
\[
 2g-2+s+2R/3+\sigma=2t,\quad
 a=24(1-g)-13t-3R-11\sigma,\quad h=6(1-g)-6t-R-3\sigma.
\]
Since $h\ge-2t$, we have $4t+R+3\sigma\le6(1-g)$.
Only $(g,t,R,\sigma,s,a,h)=(0,1,0,0,4,11,0)$ remains.
By the nonzero map to a line bundle of degree zero associated with the invariant of greatest degree, we have $a\le10t$, contrary to $a=11$.
\end{proof}

\begin{theorem}\label{sn:nonuniform-terminal}
No nine-point quotient under consideration is the Baily--Borel compactification of a nonuniform arithmetic ball quotient.
\end{theorem}

\begin{proof}
The nontrivial lattice is generated by reflections by Proposition~\ref{sg:arithmetic}.
Start with a reflection and pass to larger irreducible factors of full point stabilizers by Lemma~\ref{sn:enlargement}.
Lemmas~\ref{sn:enlargement}--\ref{sn:nonuniform-g335} force $A_5$ or $A_6$, and the latter has a parabolic subgroup $A_5$ that is the full generic stabilizer of a disk.
All ends of this $A_5$ curve are smooth cusps with $F_P=A_5$, with width one, trivial character on the fiber of \(\mathcal V_*\), and six generators of degree one.
Proposition~\ref{sc:type-a} applies without genus or immersion assumptions.
For the invariant of greatest degree, the nonzero target line bundle satisfies $6h=a-14t-\sigma$.
First $a\le14t$ gives $h\le0$, and then~\eqref{sn:curve-comparison} gives $a\le10t$.
The quotient \(\mathcal Q\) of rank 2 constructed in the proof of Proposition~\ref{sc:type-a} gives $a\ge11t+\sigma$, a contradiction.
Combining this with Theorem~\ref{sn:compact-terminal} proves the case of nine points with arbitrary weights.
\end{proof}

\section{Ten points}
\label{st:ten}
Here the ball has dimension seven.
We use one argument for the three arithmetic cases, treating them separately only until all their stabilizers have been shown to be of ADE type.
For a curve \(f:C\to W\) obtained from the full normalizer put $a=-K_W\cdot f_*C$, $t=L\cdot f_*C>0$, and $g=g(C)$.
On a projective space, \(H\) denotes the hyperplane class.
Propositions~\ref{sg:negative} and~\ref{sg:kernel} give, whenever a nonzero row of the inverse differential has a target of negative degree,
\begin{equation}\label{st:ten-git}
 a=(10-2\ell)e+2\delta,\qquad
 e\in\mathbb Z_{>0},\quad\delta\in\mathbb Z_{\ge0},\qquad
 f^*T_W\twoheadrightarrow M^{\oplus(\ell-1)},\quad\deg M=-e.
\end{equation}
Moreover $a$ is even, since the anticanonical class descends from $\mathcal O(2,\ldots,2)$.
The finite enumerations in this section include the cases $a\le0$ and $s=0$, and retain every integer $e$ allowed by~\eqref{st:ten-git}.

\subsection{Initial degree bounds}

\begin{lemma}\label{st:initial-nef}
In the cocompact case $K_W+(32/5)L$ is nef.
In the nonuniform Gaussian case $K_W+8L$ is nef.
In the remaining nonuniform cases $K_W+4wL$ is nef, where $w=|\mu(K)|=2$ or $6$.
\end{lemma}
\begin{proof}
Proposition~\ref{sh:projective} excludes $W=\mathbb P^7$.
In that case $\pi^*A_X=\mathcal O(k)$ with $k>0$ is ample, so $\pi$ is finite birational onto normal $X_\mu$ and hence is an isomorphism.
The $K_W$-negative rays of Proposition~\ref{sg:mori} have respective degrees against \(-K_W\) $2,6,4,2$: a coordinate curve, and a line where $q=2,3,4$ sections coincide.
The closed loci where these sections coincide are $\mathbb P^{8-q}$ with normal bundle $\mathcal O(-1)^{\oplus(q-1)}$.
Relative nefness of \(K_W\) prevents their contraction, and the maps of these loci to their normal images are isomorphisms.
Lemma~\ref{st:curvature} and Proposition~\ref{st:p6} give $L$-degree at least one for $q=2$.
For cocompact arithmetic lattices, Proposition~\ref{st:compact-projective} gives the bounds $1,2/3,1/3$ on the lines where $q=2,3,4$ sections coincide.
Their ratios are at most six, but Lemma~\ref{sc:moving-coordinate} gives $L\cdot\gamma\ge5/16$ on a coordinate curve, hence ratio at most $32/5$.
For nonuniform groups $wL$ is a line bundle.
This gives ratios at most $2w,6,4w,2w$.
When $w=4$, a slope $1/4$ on the locus $\mathbb P^5$ where three sections coincide would induce a quotient by Lemma~\ref{st:curvature}.
The induced quotient is nonuniform: its fixed Hermitian space over the original imaginary quadratic field has dimension at least three, so Meyer supplies a rational isotropic line.
Proposition~\ref{st:p5-quarter} excludes it.
Therefore $L|_{\mathbb P^5}\ge H/2$, and all four ratios are at most eight.
\end{proof}

\subsection{Stabilizers and nefness in the nonuniform case}

\begin{lemma}\label{st:ade-induction}
If the arithmetic quotient for ten points is nonuniform, all point stabilizers are products of ADE groups, all ends relevant to $D_6$ and $E_6$ have width one, and $4L$ is integral.
If the field is not Gaussian, $2L$ is integral.
\end{lemma}
\begin{proof}
Apply Proposition~\ref{sc:monomial-comparison}, using Lemma~\ref{st:initial-nef}, before reducing any multiple of \(L\) in that lemma.
The domains for rank 6 are given in this proof, and those for the auxiliary projective spaces in Appendix~\ref{sf:projective}; the shared lists of cusps and the conditions on the common degree \(-e\) of the coincident sections are in Appendices~\ref{sf:records}--\ref{sf:finite-domain}.
For clarity, we give the reductions determining those domains.
The primitive group $G_{34}$ is treated separately before Proposition~\ref{sc:monomial-comparison} is applied to the monomial groups.
Its polynomial invariants have degrees $(d_1,\ldots,d_6)=(6,12,18,24,30,42)$, and the generators of its ring of invariant sections have degrees $(1,3,4,6,7,9,12)$, with no torsor and width one.
Its principal theta divisor is irreducible, and the multiplicity of this divisor is at most four~\cite[Theorem 1.4]{Lozovanu}, below the first invariant degree six; hence, in the ring of invariant sections of $G_{34}$, the generator of degree one is nonzero at the origin.
After imposing three zeros at each end, the coefficient associated with the invariant of degree $d_i$ has degree $(d_i-42)t+9s$.
We have $h=-42t+12s$, $a=84(1-g)-50t-41\sigma$, with even $s$ and $j=6t$ integral.
The weak inequality $a\le24t$ gives $g=0$, $j\le8$ for $h\ge0$ and $j\le11$ for $h<0$.
The Riemann--Hurwitz formula for \(C\) and~\eqref{st:ten-git} exclude every row; the sole remaining value $\sigma=1/3$ is not an allowed sum of $1-1/e$.
For $m=6$, set $j=6t$, $h_P=(-36t+R+\nu)/p$ and $h=h_P$ for $p=1$, $h=ph_P+6t$ otherwise.
The line spanned by \(P_0\) is preserved since $6^5>6p$.
The degree and rank bounds reduce the calculation to $g=0$, $1\le j\le9$, $s\le\lfloor2+j/3\rfloor$, $R\le\lfloor(6/5)(2+j/3-s)\rfloor$.
The twelve remaining rows have respectively $p=1$ (two rows), $p=2$ (five), or $p=3,6$ (five).
The first two violate $a\ge6(-h)$; the five $p=2$ rows require $a=4j$, with three further contradicted by a nonzero map from the summand \(F\) of Proposition~\ref{sg:kernel} to a line bundle of degree zero, below $e=j$; the last five have $h<h_P<0$, contrary to the splitting of Proposition~\ref{sg:kernel}.
The tests for \(m=3\) use $G(3,1,6)$ with degrees of the generators at the ends $(1^7),(1^2,3^5),(1,3^6)$ and $1\le3t\le4$, followed by $G(3,3,6)$ with $(1^3,3^4)$ and $1\le3t\le3$.
They leave no row.
In the Gaussian case use $G(4,p,6)$, $p=1,2,4$, and $j=4t$.
The negative coefficient rank bound gives $j\le5,6,8$, respectively.
For \(p=4\), it suffices to retain the following end types, where \(\alpha\) is the tuple of degrees of the generators and \(\kappa\) the width:
\[
\begin{array}{c|c|c}
 \text{type}&\alpha&\kappa\\\hline
 A&(1^3,2^4)&1,2\\
 B&(1^2,2^5)&1,2\\
 C&(1,2^2,4^4)&1,2\\
 D&(1,2,4^5)&1,2\\
 E&(1^2,2,4^4)&1\\
 X&(1^2,4^5)&1.
\end{array}
\]
The type \(X\) is the nontrivial torsor; width two would require a point fixed by all of \(G(4,2,6)\) and is therefore excluded.
For each multiset of \(s\) ends, put \(B=\sum_c\sum_i\alpha_{c,i}/\kappa_c\) and \(\nu=\sum_c\max_i\alpha_{c,i}/\kappa_c\).
The complete necessary numerical conditions are
\[
\begin{gathered}
 1\le j\le8,\quad0\le s\le\lfloor2+j/2\rfloor,\quad
 0\le R\le\lfloor(4/3)(2+j/2-s)\rfloor,\\
 \sigma=2+j/2-s-3R/4\in\{0\}\cup[1/2,\infty),\\
 a=B-s+2-33j/2+21R/4\in2\mathbb Z,\quad a\le2j,
 \qquad h=-5j+\nu+R.
\end{gathered}
\]
There is no solution with \(h\ge0\).
For \(h<0\), put \(u=\max(0,\min(5,\lfloor(\nu-s)/j\rfloor+1))\), and use \(2\le\ell\le\min(7,1+u+\epsilon_P)\).
Here \(\epsilon_P\) is the number of possible coefficients for \(P_0\), so \(\epsilon_P\le1\).
When all widths are one, the degree \(E_P=-3j/2-3R/4\) must be integral and \(h-E_P-s<R\) forces \(\epsilon_P=0\).
Finally require \(e\ge\max(1,-h)\) and \(a-(10-2\ell)e\in2\mathbb Z_{\ge0}\).
Enumerating the finitely many choices of the data other than \(e\), and keeping every \(e\) that satisfies these conditions, leaves only the following signatures; all have \(R=0\), \(\ell=5\), and at most two ends of width two:
\[
\begin{array}{c|c|c|c|c|c}
 j&s&\nu&h&a&\sigma\\\hline
 2&3&7&-3&2,4&0\\
 3&3&9&-6&0&1/2\\
 3&3&10&-5&0,2,4,6&1/2\\
 4&4&12&-8&0&0\\
 4&4&13&-7&0,2,4,6&0.
\end{array}
\]
At such an end the least degree of a generator that does not vanish there is the width; thus the coarse curve is immersed at the end, with no fiber correction.
Condition~\eqref{sc:sorted-test} for the injection~\eqref{sc:normal-test}, for degrees $(4,8,12,16,20,6)$, requires
\[
 2a-4+45j-12\nu\le0.
\]
The five rows give $2a+2$, $2a+23$, $2a+11$, $2a+32$, $2a+20$, all positive.
This excludes the remaining groups and cusps with \(m=4\).
For an irreducible factor of rank at most 5 satisfying the condition on full interior stabilizers in Lemma~\ref{sg:fixed}, apply Meyer to its fixed Hermitian space of dimension at least three; the induced quotient isomorphic to projective space is nonuniform over the same field.
The reduction for these factors, Proposition~\ref{st:p6}, and Propositions~\ref{st:p5-sixth}--\ref{st:p5-quarter} now reduce the possibilities to parabolics of $A_7,B_7,D_7,E_7$.
For $B_6$, the four tuples of degrees of the generators at cusps of width one are $(1^7),(1^3,2^4),(1^2,2^5),(1,2^6)$.
With $\nu=\sum_c\max_i\alpha_{c,i}$, its equations are
\[
 2g-2+s+R/2+\sigma=2t,\quad
 a=\sum_c(B_c-1)+2-2g-42t+6R,\quad h=-12t+\nu+R.
\]
The integrality of $42t$ and $4t$ gives $2t\in\mathbb Z_{>0}$.
For $h\ge0$ one obtains $t=1/2$, $g=\sigma=0$, and $a\ge17-2s\ge11>8t$.
For $h<0$ only $g=0$, $2t=1,2$ remain, and substitution of the four tuples in~\eqref{st:ten-git} gives no row.
The stabilizers are therefore ADE products.
Normalized reflection lifts in dimension eight have determinants $\pm1$; the group of scalars among them has order dividing $\gcd(16,w)$.
The descent at cusps uses the invariant corresponding power of the frame $h(v,\ell)=1$.
It gives $2L$ outside the Gaussian case and $4L$ in that case.
\end{proof}

\begin{lemma}\label{st:ade-nef}
Once the ambient stabilizers are ADE products, $K_W+6L$ is nef.
\end{lemma}
\begin{proof}
On the closed locus $Z=\mathbb P^5$ where three sections coincide, a restriction $L|_Z=\lambda H$ with $\lambda<1$ induces a quotient by Lemma~\ref{st:curvature}.
Its pointwise stabilizer $F$ is $A_2$ or $2A_1$; let $J\supset F$ be the ambient stabilizer.
The effective actions of the normalizers $N_J(F)/F$ on the fixed subspace of $F$ are products of classical $A,B,D$ groups; for $J=E_6,E_7$ they are
\[
\begin{array}{c|c|c}
 J&F&N_J(F)/F\\\hline
 E_7&A_2&A_5\times\mathbb Z/2\\
 E_7&2A_1&A_1+B_4\\
 E_6&A_2&A_2\wr(\mathbb Z/2)\\
 E_6&2A_1&B_3
\end{array}.
\]
If \(F=2A_1\) has one factor in \(E_6\) and the other in a separate direct factor of \(J\), the effective action of \(N_{E_6}(A_1)/A_1\) on the fixed hyperplane is \(A_5\).
Smoothness shows that the effective restrictions are generated by reflections.
The kernel of the character on the fiber of \(\mathcal V_*\) at the cusp is a finite group of this type and rank at most 4, hence of order at most $384$.
This excludes the cusp types other than $A_4$, $B_4$, and $D_4$: the Gaussian $G(4,4,4)$ lift has order $1536$; \(m=3\) or \(m=6\) gives a nonreal character; for $F_4$ the kernel has order at least $576$; $G_{32}$ lifts in its entirety; and the $G_{29},G_{31}$ kernels have order at least $480,2880$, respectively.
For the $G_{31}$ torsor, put $J_0=2^{2+4}$, a normal subgroup of $G_{31}$; the sixteen points fixed by $J_0$ form orbits of sizes six and ten under $G_{31}/J_0$.
Choose $x$ in the orbit of size six, so that $|(G_{31})_x|=46080/6=7680$ and $(G_{31})_x/J_0\simeq S_5$.
The image of $J_0$ under the character $\chi_x$ on the fiber of \(\mathcal V_*\) has order at most $8$ and the abelianization of $S_5$ has order $2$, so $|\ker\chi_x|\ge7680/(8\cdot2)=480$; by Lemma~\ref{sg:cusp-lift}, this kernel is the full stabilizer of an interior point near the same ambient cusp.
Thus only $A_4,B_4,D_4$ cusp types remain.
The integrality of $2L$, or of $4L$ and $24c_2^{\rm orb}=360L^2$, makes $\lambda=1/2$ the only possible value below one.
For $\lambda=1/2$ put $t=u/2$ and $a=6u$.
For $B_4$, the bound $h\ge2t$ from Lemma~\ref{st:euler} and $h\le-4t+4(1-g)-2\sigma$ leave $g=0,u=1$; the four possible tuples of degrees of the generators at the ends give $h=0<u$.
For $D_4$ the equations give
\[
 7u+2s+7\sigma=16(1-g),\qquad
 (u,t,s,R,\sigma)=(1,1/2,1,2,1),(2,1,1,6,0).
\]
The curve is a line or smooth conic, whose normal bundle has largest summand $1$ or $4$.
The column of the differential of the quotient map associated with the quadratic invariant gives a nonzero map into that normal bundle from a line bundle of degree $-2t+R+s=2$ or $5$, a contradiction.
Finally $A_4$ gives $5h=3t-\sigma<5u$.
Consequently $L|_Z\ge H$.
In the Gaussian case, on the closed locus $\mathbb P^4$ where four sections coincide, the pointwise stabilizer, an ADE group of rank 3, has a normalizer whose effective restriction to the fixed space is real.
The induced group is generated by reflections, and the group of scalars among its lifts in dimension five has order dividing $\gcd(10,4)=2$.
Therefore, as in the proof of Proposition~\ref{sg:arithmetic}, the invariant frame for the full cusp stabilizer proves that the restriction of the original $2L$ is integral.
Hence $L|_{\mathbb P^4}\ge H/2$.
The lower bound $5/16$ for the $L$-degree of a coordinate curve, combined with integrality of $2L$ or $4L$, gives $L$-degree at least $1/2$ on coordinate curves.
The four degrees against \(-K_W\) $2,6,4,2$ now have $L$-degree bounds $1/2,1,1,1/2$, so Proposition~\ref{sg:mori} proves the assertion.
\end{proof}

\subsection{Curves obtained from disks with full generic stabilizer \texorpdfstring{\(E_6\)}{E6} and the conclusion}

\begin{lemma}\label{st:e6-terminal}
For a curve in the quotient of ten points obtained from a disk fixed by a finite reflection group, with full generic stabilizer $E_6$, points with stabilizer $E_6\times\mu_e$, cusps with $F_P=E_6$ and width one, and possibly $R$ points of type $E_7$, the bound $a\le32t/5$ excludes the curve.
\end{lemma}
\begin{proof}
For $R>0$ the coordinates of odd degree on the coarse quotient in the normal directions are $p_5/u,p_9/u$, with $v=u^2$.
Their residues give
\[
 2g-2+s+R/2+\sigma=2t,\quad
 a=12s+27R/2+\sigma-44t,\quad h=-12t+3R+3s.
\]
Here $2t\in\mathbb Z_{>0}$, $2t+R\in2\mathbb Z$, and $a$ is even.
The case $h\ge0$ reduces to $2t=1,2,3$ and has no row satisfying the parity conditions.
For $h<0$, after division by the zeros forced by the differential, the first five coefficients have degrees
\[
 -10t+R+2s,\ -7t+3R/2+2s,\ -6t+R+s,\ -4t+R+s,\ -3t+R/2+s.
\]
They are negative for $g>0$.
For $g=0$, one of these degrees can be nonnegative only if $t\le2$: the first, third, and fourth degrees are negative, the second is less than $2-t$, and the fifth is at most $2-t$.
All vanish except at $(t,R,s,\sigma,a,h)=(3/2,1,4,1/2,-4,-3)$, where the rank of the evaluation map from the Hom space violates~\eqref{st:ten-git}.
For the target associated with the invariant of greatest degree, the Hom space has dimension one and $e\ge-h$.
On its kernel, the row associated with the invariant of degree nine extends to a line bundle of degree $h_9=-9t+5R/2+2s$.
Since $h+h_9\le-R/8-5\sigma/12<0$, one has $h_9<e$, contradicting Proposition~\ref{sg:kernel}.
For $R=0$ the degree equations are $a=24(1-g)-20t-11\sigma$, $h=6(1-g)-6t-3\sigma$, and $t\in\mathbb Z_{>0}$.
After two successive divisions by powers of \(q_*\) at the ends, the coefficient for the basic invariant of degree \(d\) has degree $(d-12)t+s$, and these degrees exclude every case except $(g,t,s,\sigma,a,h)=(0,1,4,0,4,0)$.
Here the normal bundle splits under the involution induced by \(-I\) as $N^+\oplus N^-$, with $\operatorname{rank}N^-=2$ and $\deg N^-=-14t+3s=-2$.
A quotient line bundle of negative degree exists, but the flags defined by the theta functions give $\operatorname{Hom}(N^+,\mathcal O(-1))=0$.
The splitting in Proposition~\ref{sg:kernel} then forces every summand of $N^+$ to have positive degree.
The nonzero map to $E_{12}(3S)=\mathcal O$ associated with the invariant of greatest degree is impossible.
\end{proof}

\begin{theorem}\label{st:ten-main}
For every ten weights satisfying~\eqref{split-weights}, the quotient is not the Baily--Borel compactification of an arithmetic ball quotient.
\end{theorem}
\begin{proof}
For cocompact arithmetic lattices, we remove the monomial factors outside the parabolics of $A_7$, $B_7$, $D_7$, $G(3,3,7)$, $E_7$, $G_{34}$ by Lemma~\ref{sc:compact-monomial} and the reduction for maximal irreducible factors satisfying the condition on full interior stabilizers in Lemma~\ref{sg:fixed}.
Factors of rank at most 5 are excluded by Propositions~\ref{st:p6} and~\ref{st:compact-projective} and equality of the degrees of the basic invariants.
For cocompact arithmetic lattices, Lemma~\ref{sc:compact-monomial} includes the case of full generic stabilizer $B_6$.
At this stage $a\le32t/5$ by Lemma~\ref{st:initial-nef}.
In the nonuniform cases, Lemmas~\ref{st:ade-induction} and~\ref{st:ade-nef} give ADE stabilizers and $a\le6t<32t/5$.
In both cases, Lemma~\ref{st:e6-terminal} excludes $E_7$, then the remaining $E_6$ curves.
For cocompact arithmetic lattices, $G_{34}$ gives $a\ge252t$, and Proposition~\ref{sc:monomial-comparison} then gives $a\ge74t/5$ for $G(3,3,6)$ and $a\ge44t/5$ for $D_6$.
Each contradicts the bound $a\le32t/5$ for cocompact lattices.
For nonuniform $D_6$, the cusps have width one and generators of degrees $(1^4,2^3)$, and the equations of Proposition~\ref{sc:monomial-comparison} are
\[
 2g-2+s+R/2+\sigma=2t,\quad
 a=9s+2-2g-36t+11R/2,\quad h=-10t+R+2s.
\]
They reduce to $g=0$, $j=2t=1,2,3$.
Equation~\eqref{st:ten-git}, with every allowed integer \(e\) retained, leaves only
\[
 (j,s,R,\sigma,a,h)=(1,2,0,1,2,-1),(2,2,4,0,6,-2),
 (2,3,2,0,4,-2),(2,4,0,0,2,-2).
\]
The product contributes one map.
The flag for the minimal polarization of Proposition~\ref{sc:root-theta}, after two successive divisions by powers of \(q_*\) at the ends, leaves degrees $h+ij-2s$ for the coefficients of the elementary symmetric invariants, less than $R$ for $i<5$ and equal to $R$ for $i=5$.
The interior zeros therefore leave at most one symmetric map.
But~\eqref{st:ten-git} requires $\ell\ge4$ and at least three independent maps, a contradiction.
Finally Proposition~\ref{sc:type-a} gives $a\ge33t/2+\sigma$ for the last $A_6$ curve.
This contradicts the upper bound for $a$ available in each case and completes the ten-point case without any assumption on the support of the branch divisor or on cusp width.
\end{proof}

\section{Eleven or more points}\label{sh:section}
Write \(d=n-3\), \(r=d-1\), and \(N=d+1\).
On a projective space, \(H\) denotes the hyperplane class.
On the smooth GIT quotient \(W=X_\nu\) obtained by perturbing \(\mu\) to weights \(\nu\) in a nearby chamber, with the small projective morphism \(\pi:W\to X_\mu\), put \(L=\pi^*\mathcal L\), and in the nonuniform case put \(w=w_K=|\mu(K)|\) for the coefficient field \(K\) of Proposition~\ref{sg:arithmetic}.
For a disk fixed by a finite reflection group, let \(f:C\to W\) be the curve obtained from the quotient of the disk by its full normalizer, and put
\[
 a=-K_W\cdot f_*C,\qquad t=L\cdot f_*C>0,\qquad g=g(C).
\]
No sign condition on \(a\) will be used before it follows from the comparison using the tangent bundle.
We use strong induction on the number of labels; the cases of nine and ten points are proved in Sections~\ref{sn:nine} and~\ref{st:ten}.
If an irreducible factor of rank at most \(d-2\) satisfies the condition on full interior stabilizers in Lemma~\ref{sg:fixed}, Proposition~\ref{sg:persistent} identifies the Baily--Borel compactification of the quotient of its fixed ball by the full normalizer with the GIT quotient of the original merged weights on \(n-s\) labels, where \(s\le2\).
Thus \(9\le n-s<n\), and the induction hypothesis excludes every such factor.
This uses the original weights and the quotient of the fixed ball by the full normalizer; no ball structure is transferred to \(W\).
Each nefness statement in Propositions~\ref{sh:eleven}--\ref{sh:large} for \(K_W\) plus a smaller multiple of \(L\) is established only after the stabilizers it requires have been excluded.
After the calculation for monomial reflection groups, we complete the induction in the order eleven points, twelve points, and \(n\ge13\).

\subsection{Monomial reflection groups}

For \(n=11,12\), Proposition~\ref{sh:projective} excludes \(W=\mathbb P^d\).
If the source \(W\) is projective space, then the pullback of an ample line bundle \(A_X\) is \(\pi^*A_X=\mathcal O(k)\) with \(k>0\), hence is ample.
Thus \(\pi\) has finite fibers, and a finite birational morphism onto normal \(X_\mu\) is an isomorphism.
Corollary~\ref{sg:initial-nef} then gives \(a\le rwt\) in the nonuniform case.
For \(n\ge13\), we retain the weaker bound \(a\le(r+2)wt\) of Lemma~\ref{sg:curve-bound} for each curve until a homomorphism from the tangent bundle to a line bundle of nonpositive degree has been found.
Thus the argument for large \(n\) needs no auxiliary theorem for projective space.

For a \(G(m,p,r)\) curve, use
\[
 D=\frac{mr(r-1)}2+\frac{mr}{p},\quad
 \eta=r-1+\frac1p,\quad
 k=\begin{cases}mr&p=1,\\m(r-1)&p>1.\end{cases}.
\]
At an end let \(\alpha_c\) be the normalized tuple of degrees of the generators of its ring of invariant sections for the full stabilizer, \(\kappa_c\) the order of its character on the fiber of \(\mathcal V_*\), \(B_c=\sum_i\alpha_{c,i}/\kappa_c\), and \(\nu_c=\lfloor\max_i\alpha_{c,i}/\kappa_c\rfloor\).
Write \(R\) for the number of points with a larger irreducible stabilizer, and take \(\sigma\) over the remaining points of \(C\).
Proposition~\ref{sc:monomial-comparison} gives
\begin{align}
 2g-2+s+(1-1/m)R+\sigma&=2t,\label{sh:monomial-log}\\
 a&=\sum_c(B_c-1)+2-2g-Dt+\eta R,\label{sh:monomial-a}\\
 \label{sh:monomial-h}
 j=mt&\in\mathbb Z_{>0},\qquad h=-kt+\sum_c\nu_c+R.
\end{align}
Here \(h\) is the degree of a target line bundle of the inverse differential, with its finite character and its correction on passing to the coarse curve retained.
These equations are used together with the homomorphism from the tangent bundle to a line bundle of negative degree
\begin{equation}\label{sh:git}
 a=(n-2\ell)e+2\delta,\qquad
 e\ge-h,\quad e>0,\quad \delta\in\mathbb Z_{\ge0}.
\end{equation}
We record once in Appendix~\ref{sf:finite-domain} the bounds on \(\ell\) obtained by counting coefficients, the table of the degrees of the generators at the cusps before normalization, and the exact integer ranges.

\begin{proposition}\label{sh:monomial}
If \(r=n-4\ge9\), no disk with monomial full generic stabilizer or irreducible smooth monomial cusp occurs.
\end{proposition}
\begin{proof}
Disks without ends are excluded by Lemma~\ref{sc:compact-monomial}.
For a nonuniform arithmetic lattice, choose coordinates on the normal space \(E\) of Lemma~\ref{sg:fixed} in which the full generic stabilizer \(G(m,p,r)\) consists of monomial matrices, and take the product of \(\operatorname{diag}(\zeta,\zeta^{-1},1,\ldots,1)\), with \(\zeta\) a primitive \(m\)th root of unity, and the interchange of the second and third coordinates.
The product of the nonzero entries of this element is 1, so it lies in \(G(m,p,r)\) for every \(p\mid m\), and its trace on \(V\) is \(\zeta\) plus an integer.
As a product of normalized reflections, this element has its trace in the imaginary quadratic field \(K\), which restricts \(m\) to \(2,3,4,6\).
At an end \(c\), Lemma~\ref{sg:cusp-lift} gives \(G(m,p,r)\subset(F_P)_{x_c}\subset F_P\).
Since \(r\ge9\), the reflection representation of \(G(m,p,r)\) on \(\mathbb C^r\) is irreducible, and hence so is the linear action of \(F_P\) on \(\mathbb C^r\).
By Lemma~\ref{sg:resolution}, the singular points of \(X_\mu\) lie at genuine wall vertices, with local model \(\operatorname{Cone}(\mathbb P^{|I|-2}\times\mathbb P^{|I^c|-2},\mathcal O(1,1))\), at which Lemma~\ref{sg:cusp} splits the linear action of \(F_P\) on \(\mathbb C^r\) into two irreducible factors; hence \(c\) is a smooth cusp.
Since \(r\ge9\), the irreducible reflection group \(F_P\) of Lemma~\ref{sg:cusp} is \(S_{r+1}\) or \(G(m_c,q_c,r)\) by the classification of finite complex reflection groups \cite{GeckMalle}.
Comparing orders excludes \(S_{r+1}\), since \(m^r\,r!/p\ge2^{r-1}r!>(r+1)!\), and gives \(m^r/p\mid m_c^r/q_c\), which forces \(m\mid m_c\) because \(r\ge9\) and \(m\in\{2,3,4,6\}\).
Since we descend in \(m\), ends with \(m_c>m\) have already been excluded, so \(m_c=m\), the orders give \(q_c\mid p\), and \(\kappa_c=[(F_P)_{x_c}:G(m,p,r)]\) divides \([F_P:G(m,p,r)]=p/q_c\).
The bounds
\[
 (r-1)\max\alpha_c\le\sum_i\alpha_{c,i}\le rm+1
\]
of~\eqref{sf:cusp-bounds} and the bound \(a\le(r+2)wt\) of Lemma~\ref{sg:curve-bound} give
\begin{align*}
 (r-1)h&\le\{(r+2)w+D+2-(r-1)k\}t
 -(1+1/p-1/m)R-\sigma,\\
 a&\le(2rm+2)(1-g)
 -\{r(r-5)/2+r/p\}j
 +\{r-1+1/p-r(m-1)\}R.
\end{align*}
For \(m\ge3\), the first coefficient is at most \(-3r^2/2+21r/2+14\) if \(p=1\), or \(-3r^2/2+12r+11\) otherwise.
For \(m=2\), \(h\le-(r-3)j+4(1-g)\le-2\).
Thus \(h<0\) without a prior assumption on \(a\) or \(W\ne\mathbb P^d\).
Lemma~\ref{sg:curve-bound} now strengthens the bound to \(a\le rwt\), so the finite range of parameters obtained by counting coefficients is unchanged.
Counting coefficients as in Appendix~\ref{sf:finite-domain} gives \(\ell\le4\) for \(g\ge1\), or \(\ell\le6\) for \(j\ge6\), and then \eqref{sh:git} gives \(a>0\).
The bound \(a\le(2rm+2)(1-g)-\{r(r-5)/2+r/p\}j+\{r-1+1/p-r(m-1)\}R\) excludes \(g\ge1\) and \(j\ge6\), the latter by \(3r^2-26r-1\ge8\).
For \(r\ge29\), the remaining bound \(\ell\le14\) contradicts \(a\le-r^2/2+43r/3+2<0\).
For \(r\ge10,j\ge5\), the contradiction is \((15r^2-142r-6)/6>0\).
The remaining range of parameters is therefore \(9\le r\le28\), with \(j\le5\) at \(r=9\) and \(j\le4\) otherwise.
The 170 rows in Appendix~\ref{sf:large-rows} are excluded by Proposition~\ref{sc:bundle-tests}, the width given by~\eqref{sf:eisenstein-width}, and the absence of points fixed by all of \(G(4,2,r)\) in the nontrivial \(G(4,2,r)\) torsor marked in Appendix~\ref{sf:records}.
The lifts of the entire groups at the cusps in Appendix~\ref{sf:records} then exclude the corresponding cusps.
\end{proof}

\subsection{Eleven points}
\begin{proposition}\label{sh:eleven}
The nonexistence theorem holds for eleven points.
\end{proposition}
\begin{proof}
Here \(d=8,r=7\).
First use only the bound \(a\le7wt\).
Equations~\eqref{sh:monomial-log}--\eqref{sh:monomial-h} give \(g=0,h<0,1\le j\le7\); Appendix~\ref{sf:eleven-rows} records all 413 rows of the finite range.
By the tests in Appendix~\ref{sf:eleven-rows} that use only \(a\le7wt\), Lemma~\ref{sc:compact-monomial} except for \(D_7\) with possible points whose stabilizer is primitive, and the induction hypothesis for irreducible factors satisfying the condition on full interior stabilizers in Lemma~\ref{sg:fixed}, every irreducible finite factor is a parabolic subgroup of one of the groups
\[
 A_8,\ B_8,\ D_8,\ G(4,2,8),\ G(4,4,8),\ E_8.
\]
Among these parabolic subgroups, all reflections have order two, and products of total rank at most 3 have orders dividing \(384\).
Thus \(\mathcal T=18L=2K_W+B\) is integral and \(384c_3^{\rm orb}=-32256L^3\) extends across the cusps and the exceptional locus \(\operatorname{Exc}(\pi)\) of the small projective morphism \(\pi\).
On a closed locus \(Z_J=\mathbb P^k\), \(k\ge3\), with \(\mathcal T|_{Z_J}=bH\), this says
\[
 \frac{448b^3}{81}\in\mathbb Z,\qquad 9\mid b,\qquad L\cdot\text{line}\ge\frac12.
\]
For a maximal short pair \(J\), the corresponding closed locus is the divisor \(E=Z_J\simeq\mathbb P^7\).
Relative nefness of \(K_W\) makes \(\pi|_E\) finite, and normality of its image, the GIT quotient of the original merged weights, makes this restriction an isomorphism.
The divisor \(E\) cannot be a branch component, by Proposition~\ref{sh:projective}; its normal bundle is \(\mathcal O(-1)\), so Lemma~\ref{sc:negative-normal} gives \(L|_{\mathbb P^7}\ge H\).
The lines in the remaining closed loci \(Z_J\) have degree against \(-K_W\) at most five, and the coordinate curves \(\gamma_i\) have \(t\ge1/3\).
The generators of \(\overline{\mathrm{NE}}(W)\) in Proposition~\ref{sg:mori} therefore give
\begin{equation}\label{sh:eleven-nef}
 K_W+10L\ \text{nef}.
\end{equation}
This argument precedes the exclusion of \(E_8\).

At an \(E_8\) point take the \(E_7\) curve through it given by a parabolic subgroup \(E_7\), with the same notation for this curve, and let \(R>0\) count the \(E_8\) points on it.
Its cusps of width one and the characters of their groups \(F_P\) give
\[
 2g-2+s+R/2+\sigma=2t,\quad 2t\in\mathbb Z_{>0},\quad
 a=18s+\frac{57}{2}R+\sigma-72t.
\]
The target associated with the invariant of greatest degree has \(h=-18t+6R+4s\), but \eqref{sh:eleven-nef} gives
\begin{equation}\label{sh:eleven-budget}
 \frac{57}{2}R+18s+\sigma\le82t.
\end{equation}
If \(h\ge0\), these inequalities give \(5t/2-2\le R\le2t/3\), hence \(t\le12/11\).
If \(R+s\ge4t\), they give \(4t-4\le R\le20t/21\), hence \(t\le21/16\).
In either case \(2t\in\mathbb Z\) leaves \(t=1/2,1\), forcing \(R<1\).
Therefore \(h<0\) and \(R+s<4t\).
From the flags of ideals in the local algebras and the derivatives of the theta functions associated with the root lattice, we obtain the following degrees of the coefficients after division by the zeros forced by the differential:
\[
\begin{array}{c|r@{\qquad}c|r}
d_i&\deg&d_i&\deg\\ \hline
2&-16t+3R+3s&10&-8t+2R+2s\\
6&-12t+3R+2s&12&-6t+2R+s\\
8&-10t+3R+2s&14&-4t+R+s
\end{array}
\]
All are negative; use \(3R+2s\le82t/9\) and \(2R+s\le328t/57\) for the middle entries.
The coefficient of the invariant of degree eighteen is constant, so the Hom space has dimension one.
Hence \(\ell=2\) by Proposition~\ref{sg:negative}, and by Proposition~\ref{sg:kernel} the kernel \(F=f^*T_{Z_I^{\mathrm s}}\) of \(\psi:f^*T_W\twoheadrightarrow M\) has no quotient of degree less than \(e\ge-h\).
But the projection of \(F\) to the invariant of degree fourteen has degree \(h_{14}=-14t+5R+3s\), and
\[
 h_{14}+h=11R+7s-32t
 \le\frac7{18}\left(\frac{57}{2}R+18s\right)-32t
 \le-\frac t9<0.
\]
This contradiction excludes \(E_8\) in every genus and CM field.

Now the tests in Appendix~\ref{sf:eleven-rows} that use \eqref{sh:eleven-nef} exclude all monomial factors of ranks 7 and 8 and the cusps at which \(F_P\) is such a factor.
For \(D_7\), no stabilizer \(E_7\) containing \(D_7\) is possible, since \(\det D_7/\det E_7=2\) is not the square of an integer; its possible points with stabilizer \(E_8\) have just been excluded.
Its last nonuniform row has \(a=6,t=1/2\), contrary to \(a\le10t\), and Lemma~\ref{sc:compact-monomial} now applies.
An \(E_7\) curve has no point with a larger irreducible stabilizer left.
Without ends, \(a\ge126t\) follows from the row of the inverse differential associated with the invariant of greatest degree.
With ends,
\[
 a=36(1-g)-36t-17\sigma,\quad h=-18t+4s,\quad 2t\in\mathbb Z_{>0}.
\]
After two successive divisions by powers of \(q_*\) at the cusps, at most two directions remain for \(g\ge1\), or four for \(g=0,t\ge1\), forcing \(a>0\) when \(a=36(1-g)-36t-17\sigma\) gives \(a\le0\).
At \(t=1/2\), \(s=1,2,3\) give \((a,h)=(-16,-5),(1,-1),(18,3)\).
Only \(s=2\) survives positivity and \eqref{sh:eleven-nef}.
Then \eqref{sh:git} forces \(\ell=5,e=1,\delta=0\).
Four independent maps to \(\mathcal O(-1)\), after two divisions, have coefficients of degrees \((-4,-2,-1,0,1,2,4)\).
After the constant coefficient of the invariant of degree ten vanishes, three independent covectors remain; the third projection defined by the theta functions associated with the root lattice forces a third division by powers of \(q_*\) and leaves only degrees fourteen and eighteen, a contradiction.
The induction hypothesis excludes all irreducible factors of smaller rank satisfying the condition on full interior stabilizers in Lemma~\ref{sg:fixed}, including type \(A\).
Thus an \(A_7\) curve exists.
Proposition~\ref{sc:type-a} gives \(a\ge23t+\sigma>10t\), contradicting \eqref{sh:eleven-nef}.
\end{proof}

\subsection{Twelve points}
\begin{proposition}\label{sh:twelve}
The nonexistence theorem holds for twelve points.
\end{proposition}
\begin{proof}
Here \(d=9,r=8\), and descent of \(\mathcal O(1,\ldots,1)\) gives \(a\in2\mathbb Z\).
An \(E_8\) curve has polynomial degrees \((2,8,12,14,18,20,24,30)\), degrees \((1,2,2,3,3,4,4,5,6)\) of the generators of the ring of invariant sections, and
\[
 2g-2+s+\sigma=2t,\quad
 a=29s+2-2g-128t,\quad
 h=-30t+6s,\quad 2t\in\mathbb Z_{>0}.
\]
These equations and parity give \(h+2t<0\).
Thus \(df\) and the quadratic invariant are both killed by the row of the inverse differential associated with the invariant of greatest degree.
We use the normalized Eisenstein series \(E_2,E_4,E_6\), with \(E_2\) quasimodular, and the modular discriminant \(\Delta_{\mathrm{mod}}(\tau)=(E_4(\tau)^3-E_6(\tau)^2)/1728\).
The unique theta function of degree one of \(E_8\) satisfies
\[
 \Theta=E_4+c(E_2E_4-E_6)p_2+O(z^4),\qquad c\ne0.
\]
If \(E_4=0\), then \(E_4^3-E_6^2=1728\Delta_{\mathrm{mod}}\ne0\) makes its quadratic coefficient nonzero.
Consequently, we obtain a second zero at the cusps in every coefficient for the directions of \(A_P\) from the equation in the radial direction or, when \(E_4=0\), from the equation given by the quadratic coefficient of \(\Theta\).
After division by the zeros forced by the differential, the degrees of the coefficients for the directions of \(A_P\) are \((d_i-30)t+4s\).
For \(g\ge1\) at most two directions remain; for \(g=0,t\ge1\) at most five remain.
Thus \(\ell\le3\) or \(6\), contradicting \eqref{sh:git} and \(a\le60(1-g)-70t\).
At \(t=1/2\), the possible \(s=0,2\) give \(a=-62,-4\) and again too few directions.
This excludes \(E_8\) points, \(E_8\) curves, and cusps with \(F_P=E_8\) before the monomial groups are treated.

Lemma~\ref{sc:compact-monomial} excludes curves without ends with monomial full generic stabilizer.
For curves with ends and with monomial full generic stabilizer, \eqref{sh:monomial-log}--\eqref{sh:monomial-h}, \(a\le8wt\), and counting coefficients as in Appendix~\ref{sf:finite-domain} imply \(g=0,h<0,1\le j\le7\).
Appendix~\ref{sf:twelve-rows} excludes all 117 rows.
The induction hypothesis excludes irreducible factors of smaller rank satisfying the condition on full interior stabilizers in Lemma~\ref{sg:fixed}, so every remaining finite stabilizer is a product of type \(A\) groups.
For \(\mathcal T=20L\) and a closed locus \(Z_J=\mathbb P^k\), \(k\ge3\), with \(\mathcal T|_{Z_J}=bH\), Proposition~\ref{sc:chern-integrality} gives
\[
 \frac{27b^2}{20},\ \frac{9b^3}{25}\in\mathbb Z,\qquad 10\mid b.
\]
Thus its lines have \(t\ge1/2\); their degree against \(-K_W\) is at most eight.
The coordinate curves \(\gamma_i\) have \(a=2,t\ge7/20\) by Lemma~\ref{sc:moving-coordinate}.
The generators of \(\overline{\mathrm{NE}}(W)\) in Proposition~\ref{sg:mori} therefore give \(K_W+16L\) nef, including the rays spanned by the lines \(\beta_J\) with \(2|J|=n\), on which \(K_W\) has degree zero.
For the resulting \(A_8\) curve, Proposition~\ref{sc:type-a} gives \(a\ge61t/2+\sigma>16t\), the required contradiction.
\end{proof}

\subsection{Thirteen or more points}
\begin{proposition}\label{sh:large}
The nonexistence theorem holds for every \(n\ge13\).
\end{proposition}
\begin{proof}
By the induction hypothesis, no irreducible factor of rank at most \(d-2\) satisfies the condition on full interior stabilizers in Lemma~\ref{sg:fixed}.
A maximal factor outside type \(A\) must therefore lie in a monomial group of rank \(d-1\) or \(d\), and Proposition~\ref{sh:monomial} excludes it.
The classification of parabolic subgroups now gives an \(A_{d-1}\) curve, with larger groups \(A_d\) and products, and ends of width one at which the generators of the ring of invariant sections for \(F_P=A_{d-1}\) have degrees \(1^d\).
All finite stabilizers are products of type \(A\).
Put
\[
 b_N=\prod_{q\mid N,\ q\ {\rm prime}}q^{\lceil2v_q(N)/3\rceil},
 \qquad C_N=\frac{2N^2}{b_N}.
\]
For the integral \(T=2NL\), Proposition~\ref{sc:chern-integrality} gives \(N^2\mid b^3\) on every closed locus \(Z\) of dimension at least three where sections coincide, with \(T|_Z=bH\).
Thus \(b_N\mid b\), and every \(K_W\)-negative line in such a locus has
\[
 \frac{-K_W\cdot\text{line}}{L\cdot\text{line}}\le\frac{2N^2}{b_N}=C_N.
\]
This includes the case \(Z=W=\mathbb P^d\) of a singleton \(J\), if it exists.
Every such locus on whose lines \(K_W\) is negative has dimension at least five, so these Chern restrictions are legitimate.
For coordinate curves, Lemma~\ref{sc:moving-coordinate} gives
\[
 L\cdot\gamma\ge\frac{N-3}{2N},\qquad
 \frac{-K_W\cdot\gamma}{L\cdot\gamma}\le\frac{4N}{N-3}<C_N,
\]
where the last inequality follows from \(b_N\le N\).
The generators of \(\overline{\mathrm{NE}}(W)\) in Proposition~\ref{sg:mori} therefore give \(K_W+C_NL\) nef.
There is no preliminary exclusion of \(W=\mathbb P^d\) in this argument.

Proposition~\ref{sc:type-a}, applied to the \(A_{d-1}\) curve, gives
\[
 a\ge\frac{N^2-4N+1}{2}t+\sigma.
\]
For \(N=9,10,11,12,13,14\), one has \(b_N=N\), and \(C_N=2N<(N^2-4N+1)/2\).
For \(N\ge15\), the inequalities \(b_N\ge N^{2/3}\) and \(N^{1/3}<N/6\) give
\[
 C_N\le2N^{4/3}<\frac{N^2}{3}<\frac{N^2-4N+1}{2}.
\]
The last strict inequality is \(N^2-12N+3>0\), valid for \(N\ge15\).
This contradicts \(a\le C_Nt\) for every \(n\ge13\).
\end{proof}

\begin{proof}[Proof of Theorem~\ref{si:main}]
The case of nine points is proved by Theorems~\ref{sn:compact-terminal} and~\ref{sn:nonuniform-terminal}, and the case of ten points by Theorem~\ref{st:ten-main}.
Assume the result for every smaller number of labels, starting at nine.
Proposition~\ref{sg:persistent} excludes each irreducible factor of rank at most \(d-2\) satisfying the condition on full interior stabilizers in Lemma~\ref{sg:fixed}, because the Baily--Borel compactification of the quotient of its fixed ball by the full normalizer is the GIT quotient on \(n-s\) labels obtained by merging \(s+1\) of the original labels into one, where \(s\in\{1,2\}\) is the rank of the factor.
Propositions~\ref{sh:eleven}, \ref{sh:twelve}, and~\ref{sh:large} exclude the remaining cases \(n=11\), \(12\), and \(n\ge13\), respectively.
All stability walls are taken into account by the smooth GIT quotient \(W=X_\nu\) of Lemma~\ref{sg:resolution}, obtained by perturbing \(\mu\) to weights \(\nu\) in a nearby chamber, and by the arguments avoiding the finitely many cusps.
\end{proof}

\appendix
\section{Derivatives of theta functions and differentials of quotient maps}\label{sa:flags}
This appendix proves Propositions~\ref{sc:root-theta} and~\ref{sc:parabolic-flags}.
The degrees of polynomial invariants, the degrees of generators of the ring of invariant sections, and the weights of Jacobi forms remain distinct throughout.

\subsection{Projections defined by theta functions associated with the root lattice}
\begin{proof}[Proof of Proposition~\ref{sc:root-theta}]
Write the basic Jacobi forms of~\cite[Theorem~3.6]{Wirthmuller} as \(\phi_0,\ldots,\phi_r\), of weights \(0,-d_1,\ldots,-d_r\) and indices \(1,m_1,\ldots,m_r\).
The relative and fiber statements of \cite[Theorem~3.6(ii)--(iv)]{Wirthmuller}, at level three, identify the ring of invariant sections on each elliptic fiber with the specialization of the polynomial algebra of Weyl invariant Jacobi forms.
For linear coordinates \(z_a\), set \(\mathcal J_{0i}=m_i\phi_i\), \(m_0=1\), and \(\mathcal J_{ai}=\partial_{z_a}\phi_i\).
By \cite[Propositions~2.2--2.3, Theorem~3.1, Section~4]{WangJacobi}, the Jacobian satisfies
\[
 \det\mathcal J=c\prod_{\alpha\in R^+}
 \frac{\vartheta(\tau,(\alpha,z))}{\eta(\tau)^3},\qquad c\ne0.
\]
Induction on Taylor degree in the Jacobi transformation law kills terms of \(\phi_i\) below degree \(d_i\).
The inhomogeneous term only uses lower terms already zero, leaving a holomorphic modular form of negative weight.
If \(\phi_0(\tau,0)=0\), the Jacobian would have order at least \(|R^+|+2\), contradicting \(\det\mathcal J=c\prod_{\alpha\in R^+}\frac{\vartheta(\tau,(\alpha,z))}{\eta(\tau)^3}\).
Put \(r_i=\phi_i/\phi_0^{m_i}\), and write \(r_i=\sum_k a_{ik}(\tau)p_k\) modulo the square of the invariant maximal ideal.
The coefficient \(a_{ik}\) has modular weight \(d_k-d_i\).
Coefficients of negative weight vanish, coefficients of weight zero are constant, and the term of lowest degree in \(z\) of \(\det\mathcal J\) makes each diagonal block of equal degree invertible.
For type \(D_r\), the index one form \(\prod_i\vartheta(\tau,z_i)/\eta(\tau)^{3r}\) is odd under the change of sign of one coordinate.
Averaging separates \(P_0=x_1\cdots x_r\) from a symmetric invariant of the same degree.
Thus the indicated projections are triangular with invertible diagonal blocks.
Polynomial generation shows that derivatives obtained from sections of degree at most \(j\) are spanned by the \(dr_i\) with \(m_i\le j\).
This proves the projection statement while retaining the components along the differentials of basic invariants of higher degree.

For the translation assertion, put \(\lambda_\Theta(x)=t_x^*\Theta\otimes\Theta^{-1}\).
Primitive simple roots give \(A_\tau^{W(R)}=\ker\lambda_\Theta\).
At a common fixed point where the Weyl group acts trivially on the fiber of \(\Theta\), a lift of \(t_x\) to the theta group commutes with the equivariant structure for the Weyl group: its scalar commutator is the ratio of the characters at \(x\) and zero.
It therefore transports the section spaces and their projections.
The same comparison preserves eigenspaces of any further automorphism fixing both points and acting trivially on their fibers.
The modular action of \(-I\) gives parity \((-1)^{d_i}\).
\end{proof}

For the cusps of type \(D_v\), \(v=5,6\), the actual polarization and trivial character on the fiber of \(\mathcal V_*\) can be recovered without choosing an isogenous substitute.
Adjacent simple roots give \(Q(D_v)_{W(D_v)}=0\) and \(\operatorname{Pic}^0(A)^{W(D_v)}=0\).
Invariant positive integral alternating forms are integral multiples of the root form tensored with the elliptic form.
Connected reflection divisors through the origin, together with generation of the entire group of lifts by reflections, force trivial character on the fiber of \(\mathcal V_*\).
The minimal degrees of the generators of the ring of invariant sections are \((1^4,2^{v-3})\) \cite{Rains}.
Every proper Veronese subring of the ring of invariant sections gives a cone with a singular vertex: no nonidentity scalar is a pseudoreflection on the four coordinates of degree one.
Consequently the smooth cusp itself uses the minimal bundle.
At a relevant fixed point, let \(\mathfrak n\) be the maximal ideal of its image, let \(F_1\subset\mathfrak n/\mathfrak n^2\) be the span of the differentials of \(s/\phi_0\) for the invariant sections \(s\) of degree one, and let \(F_1^{\rm sym}\subset F_1\) be the \(+1\) eigenspace of the involution \(\iota(x_1,\ldots,x_v)=(-x_1,x_2,\ldots,x_v)\), which normalizes \(W(D_v)\) and sends \(P_0=x_1\cdots x_v\) to \(-P_0\).
With \(f_j=e_j(x_1^2,\ldots,x_v^2)\), Proposition~\ref{sc:root-theta} gives
\[
 F_1=\mathbb C\,dP_0\oplus F_1^{\rm sym},\qquad
 F_1^{\rm sym}\longrightarrow\langle df_1,df_2\rangle
\]
isomorphically, uniformly in the elliptic parameter and at all relevant translated fixed points.

\subsection{Why the algebra comparison preserves the flags}\label{sa:ade-flags}
Let \(G\) be a finite complex reflection group with well-generated irreducible factors and reflections of order two, and let \(\pi:V\to V/G\).
The product of the Saito structure makes \(\operatorname{coker}d\pi\) a finite birational unital algebra on the reduced reflection arrangement.
Here is the localization argument needed to use that product.
Multiplication by the Euler vector field has matrix \(U=x_rI+U_0(x_1,\ldots,x_{r-1})\) and determinant the discriminant \cite[Theorem~5.6 and its proof]{KatoManoSekiguchi2020}.
Its characteristic polynomial is the discriminant with \(x_r\) translated by the spectral variable.
Generic roots on every discriminant component are distinct \cite[Theorem~5.3]{Bessis2015}.
There the product is diagonal in canonical coordinates and the quotient is a fold.
The square differential matrix has nonzero determinant, so its image is isomorphic to its locally free source.
Multiplication preserves this image at every point of codimension one by the fold description.
The coefficients in a basis of the image therefore extend by Hartogs, as in \cite[Proposition~2.3]{GrangerMondSchulze2012}.
The square presentation of \(\operatorname{coker}d\pi\) has reduced determinant, so \(\operatorname{coker}d\pi\) is Cohen--Macaulay and torsion-free of rank 1 on each arrangement component.

For a finite birational unital \(S\)-algebra \(M\), one has \(\operatorname{End}_S(M)=M\).
An endomorphism is multiplication in the total quotient ring, and its value at one belongs to \(M\).
An isomorphism of the underlying modules thus conjugates their endomorphism algebras; the original map differs from the resulting algebra isomorphism by a unit.
At a point with stabilizer \(P\), the quotient germ is the product of its fixed space and the normal \(P\)-quotient.
The fixed factor contributes zero to the cokernel, and parabolic subgroups remain well-generated \cite[Lemma~2.7]{Bessis2015}.
This proves localization, up to a unit, for the normal quotient.

For the inclusions \(D_5\subset E_6\), \(E_6\subset E_7\), and \(E_7\subset E_8\), let \(F(x,t)\) be the corresponding semiuniversal deformation with base \(B\), discriminant \(D\), and projection \(\rho:\{F=d_xF=0\}\to B\).
Kodaira--Spencer identifies \(\mathcal T_B/\operatorname{Der}_B(-\log D)\) with \(\rho_*\mathcal O_{\{F=d_xF=0\}}\) \cite[Introduction, (0.1)]{GrangerMondSchulze2012}.
Multiplication by the Euler vector field has the same image \cite[Lemma~4.6]{KatoManoSekiguchi2020}.
Apply the argument using \(\operatorname{End}_S(M)=M\) before taking fibers to compare these two products.
Apply it also on the reflection arrangement to compare the product of~\cite{GrangerMondSchulze2012} with the product of the Saito structure.
Only then use the algebra fiber comparison of \cite[Proposition~4.21, Theorem~4.23]{GrangerMondSchulze2012}.
Thus the cokernel of the differential of the quotient map on the fiber is the local Tjurina algebra up to a unit.
The sheaves on the deformation base and on the reflection space have not been identified as sheaves on the same space.

The complete flags consist of the following ideals, each defined by the algebra structure alone, with \(\mathfrak m\) the maximal ideal:
\[
\begin{array}{c|l}
D_5&\mathfrak m^3\subset\mathfrak m^2\subset
 \operatorname{Ann}(\mathfrak m^2)\subset\mathfrak m\subset A\\
E_6&\mathfrak m^3\subset\mathfrak mJ\subset\mathfrak m^2
 \subset J+\mathfrak m^2\subset\mathfrak m\subset A\\
E_7&\mathfrak m^4\subset\mathfrak m^3\subset\mathfrak mI
 \subset\mathfrak m^2\subset I\subset\mathfrak m\subset A\\
G_{33}&\mathfrak mS\subset\mathfrak m^2\subset S\subset\mathfrak m\subset A .
\end{array}
\]
The first three algebras and their ordered bases are
\[
\begin{array}{c|c|c}
D_5&\mathbb C[x,y]/(xy,x^2+y^3)&(y^3,y^2,x,y,1)\\
E_6&\mathbb C[x,y]/(x^2,y^3)&(xy^2,xy,y^2,x,y,1)\\
E_7&\mathbb C[x,y]/(x^2+y^3,xy^2)&(y^4,y^3,xy,y^2,x,y,1).
\end{array}
\]
For \(E_6\), \(J\) is the closure of square-zero elements in \(\mathfrak m\setminus\mathfrak m^2\), here the ideal \((x)\); for \(E_7\), \(I=\operatorname{Ann}(\mathfrak m^3)\).
Direct multiplication shows each flag consists of the spans of the initial segments of its ordered basis.
It is therefore preserved both by algebra isomorphisms and by unit multiplication.

Brieskorn identifies the weighted deformation base with the Weyl quotient \cite[p.~282]{Brieskorn1971}.
For \(E_7\subset E_8\) and \(E_6\subset E_7\), after splitting off a square use respectively
\[
\begin{array}{c|c}
F_a&(h_i)\\ \hline
y^3+ayz^3+z^5&(z^4,z^3,yz,z^2,y,z,1)\\
y^3+az^4+yz^3&(yz^2,yz,z^2,y,z,1),
\end{array}\qquad a\ne0 .
\]
The same ideals are the spans of the initial segments of these bases; in the second row \(J=(y-z^2/(8a))\).
The monomials multiplying the normal deformation parameters differ only in their first entries, \(yz^2\) and \(z^3\).
Their classes have diagonal coefficients \(-5/(3a)\) and \(-3/(4a)\), respectively, and all other diagonal coefficients one.

For \(D_5\subset E_6\), use
\begin{align}
 F_T&=y^3+z^4+T_2yz^2+T_5yz+T_6z^2+T_8y+T_9z+T_{12},\notag\\
 T(u)&=(-3u^2,3u^5,-3u^6/2,-3u^8/4,u^9,-3u^{12}/16).
 \label{sa:d5-axis}
\end{align}
At its \(D_5\) point put \(Z=z-u^3/2\), \(w=y-uZ\).
Then \(F_{T(u)}=w^3+3uw^2Z+Z^4\), with basis \((Z^3,Z^2,w,Z,1)\) of its Tjurina algebra and relations
\[
 Z^4=0,\qquad wZ=2Z^3/(3u^2),\qquad w^2=-2uwZ .
\]
With \(T_2\) held constant, the classes of \((yz,z^2,y,z,1)\) are lower triangular in this basis, with diagonal \((2/(3u^2),1,1,1,1)\).
At \(u=1\), let \(A\) be this Tjurina algebra and \(\mathfrak m\) its maximal ideal; comparison with the flag \(\mathfrak m^3\subset\mathfrak m^2\subset\operatorname{Ann}(\mathfrak m^2)\subset\mathfrak m\subset A\) of ideals for \(D_5\) is upper triangular with nonzero diagonal.
The resulting matrix of the normal differential is \(K_0^{-1}U\), with \(K_0\) lower triangular invertible and \(U\) upper triangular invertible, so all its leading principal minors are nonzero.

\subsection{\texorpdfstring{The differential for \(G_{33}\subset G_{34}\)}{The differential for G33 in G34}}\label{sa:g33-flag}
Use flat basic invariants \(x_i\) of \(G_{34}\), of degrees \((6,12,18,24,30,42)\), and put \(e_i=\partial_{x_i}\).
The invariant of degree six is nonzero on the axis \cite[Theorem~3, Case ST34 (3)]{SekiguchiE7ST34}, so put \(\lambda=x_1=u^6\).
We pull back the tangent sheaf, with the product of the Saito structure, over \(R=\mathbb C[\lambda]\); the resulting algebra has basis \((e_1,\ldots,e_6)\) with weights \((6,5,4,3,2,0)\) after a common shift.
The quadratic potentials in \cite[Section~9]{KatoManoSekiguchi2015}, identified as \(G_{34}\) in \cite[Remark~5.11]{KatoManoSekiguchi2020}, give the fiber of this algebra over \(\lambda=0\)
\[
 A_{34}=\mathbb C[a,b]/(a^4,a^2b,b^2-a^3),\quad
 a=e_5,\ b=e_4,\quad (a^3,ab,a^2)=(2e_1,e_2,2e_3).
\]
For \(G_{33}\) the algebra is
\[
 A_{33}=\mathbb C[A,B]/(A^3,A^2B,B^2),\quad
 (\partial_{p_4},\partial_{p_6},\partial_{p_{10}},\partial_{p_{12}},\partial_{p_{18}})
 =(AB,A^2,B,A,1).
\]
Saturate the axis ideal in the pulled back algebra and let \(Q\) be the free quotient of rank 5.
Subtract the coordinates of the support from \(a,b\); they belong to \(R\) because each is one fifth of the trace of multiplication by the corresponding element \(a\) or \(b\) on \(Q\).
The augmentation ideal satisfies \(\mathfrak m^3=0\) generically and hence everywhere by torsion-freeness.
The kernel of \(A_{34}\to Q/\lambda Q\) contains nonzero \(a^3\) and has dimension one by flatness, so \(Q/\lambda Q=A_{33}\).
Graded Nakayama gives the basis \((1,a,b,a^2,ab)\).
Weights and \(\mathfrak m^3=0\) imply
\[
 b^2=c\lambda^2a^2+d\lambda ab,\qquad d^2+4c=0 .
\]
The second equation follows from the generic quadratic relation of rank 1.
Replace \(b\) by \(b-d\lambda a/2\); then \(b^2=0\).
The square-zero ideal is \(S=\langle b,a^2,ab\rangle\).
The inverse images of \(\mathfrak mS,\mathfrak m^2,S,\mathfrak m,Q\), together with the axis ideal, are free direct summands, with special fibers the spans of the initial segments of \((e_1,\ldots,e_6)\).
Their homogeneous lifts can mix an \(e_i\) only with basis elements of lower weight multiplied by nonnegative powers of \(\lambda\).
Consequently each initial block remains invertible on the generic fiber.
This argument with weights proves that all leading principal minors including the axis are nonzero; removing the axis then proves that all five leading principal minors in the normal directions are nonzero.
The center contains \(\mu_6\), and the nonzero invariant of degree six bounds the order of the scalar action on the axis by six, so that order is exactly six.
In the normal directions, the coordinates on the coarse quotient are \(u^2p_4,p_6,u^2p_{10},p_{12},p_{18}\), of degrees \(6,6,12,12,18\); dividing each degree by six gives the weights \((1,1,2,2,3)\), against the weights \((2,3,4,5,7)\) obtained in the same way from the degrees \(12,18,24,30,42\) of the basic invariants of \(G_{34}\) other than the axis coordinate.

\begin{proof}[Completion of Proposition~\ref{sc:parabolic-flags}]
For each inclusion, let \(D_i\) be the degrees of the basic invariants of the larger group, in increasing order, after deleting the lowest.
The nonzero leading minors just proved and homogeneity give
\[
 M(v)=\operatorname{diag}(u^{D_i})B
 \operatorname{diag}(u^{\epsilon_i-k_i}),
\]
where \(B\) is constant with every leading principal minor nonzero.
Elimination of the tangent column uses ratios of derivatives of restricted invariants, holomorphic in \(v\).
Elimination below the diagonal uses \(u^{D_i-D_j}\), \(i>j\), also holomorphic in \(v\), and changes only target coordinates.
The diagonal orders are \((D_i-k_i+\epsilon_i)/e\), giving the four tuples of orders in Proposition~\ref{sc:parabolic-flags}.
The last row of the inverse differential uses only the last transformed target coordinate.
On the kernel of the last row that coordinate vanishes, so the preceding row of the inverse differential has pole order at most the preceding entry of the tuple of orders.
\end{proof}

\subsection{\texorpdfstring{The image of the ambient local ring along the axis for \(D_5\subset E_6\) and an obstruction to smooth threefolds}{The image of the ambient local ring along the axis for D5 in E6 and an obstruction to smooth threefolds}}\label{sa:d5-image}
The image of the ambient local ring along \eqref{sa:d5-axis} is exactly \(\mathcal A=\mathbb C\{u^2,u^5\}\); its only missing nonnegative exponents are one and three.
In the local Tjurina algebra, the two lowest normal directions span \(\mathfrak m^2=\langle Z^2,Z^3\rangle\).
We reduce the monomials multiplying the ambient deformation parameters modulo this ideal in the basis \((w,Z,1)\) and obtain the subspace of dimension three containing the tangent axis and those two directions:
\begin{equation}\label{sa:d5-tangent-equations}
 dT_8=-u^6dT_2/4-u^3dT_5/2,\qquad
 dT_9=-u^3dT_6,\qquad dT_{12}=u^6dT_6/4 .
\end{equation}
There is no smooth threefold through zero containing the axis and with these tangent spaces away from zero.
For such a threefold, the limiting projection onto \((T_2,T_5,T_6)\) would be invertible, making the threefold a holomorphic graph.
Its derivative coefficients along the curve would then lie in \(\mathcal A\), but the coefficient \(-u^3/2\) in \eqref{sa:d5-tangent-equations} does not.
This obstruction applies to derivatives of actual ambient holomorphic functions; arbitrary covectors on the normalization need not have coefficients in \(\mathcal A\).

\subsection{\texorpdfstring{The \(G_{33}\) theta characters and derivatives}{The G33 theta characters and derivatives}}
\begin{lemma}\label{sa:g33-theta}
For a cusp with \(F_P=G_{33}\) and the minimal polarization, the generators of the ring of invariant sections have degrees \((1,1,3,3,4,6)\).
For a scalar \(\omega\) of order three, their characters are \((1,\omega,1,\omega,1,1)\).
At such a cusp, the four common fixed points in the abelian variety equal the kernel of the polarization.
At each relevant point with trivial character of the full stabilizer on the fiber of \(\mathcal V_*\), a lift of the corresponding translation to the theta group transports these characters and the projections of the first derivatives of invariant sections.
For the basic invariants of polynomial degrees \((4,6,10,12,18)\) on the normal representation, the successive divisions of coefficients use powers of \(q_*\) with exponents \((1,3,3,4,6)\).
\end{lemma}
\begin{proof}
Use the polarized abelian variety of \cite[Sections~4.2 and 7.6]{Rains}:
\[
 E:y^2=x^3+1,\quad b=(0,1),\quad
 A_0=\ker(E^6\xrightarrow{\sum}E),\quad
 A=A_0/\langle(b,\ldots,b)\rangle .
\]
On \(A_0/S_6=|6O|\), use as coordinates the coefficients \((c_0,c_x,c_y,c_{x^2},c_{xy},c_{x^3})\) of the monomials \(1,x,y,x^2,xy,x^3\).
The sections of degree one are \(\theta_0=c_0+2c_{x^3}\), \(\theta_1=c_{x^2}\), with \(\theta_0(0)\ne0\).
Translation by \(b\) is \((x,y)\mapsto(-2x/(y+1),(y-3)/(y+1))\) and multiplies both sections by four, so the action on the fibers, normalized by dividing by four, fixes them.
The scalar \(\omega I\) of order three acts on each factor \(E\) by \((x,y)\mapsto(\omega x,y)\), hence on these coordinates with weights \((0,2,0,1,2,0)\) modulo three.

Here are explicit nonzero minors for the two new characters of degree three.
Work over \(\mathbb F_{1009}\), with \(\omega=374\), and use the following five points for each tuple; the sixth point is the negative of their sum in \(E\).
\[
\begin{array}{c|ccccc}
 a&P_1&P_2&P_3&P_4&P_5\\\hline
 1&(43,657)&(773,520)&(318,858)&(690,511)&(417,539)\\
 2&(227,130)&(228,395)&(142,879)&(864,669)&(981,805)\\
 3&(67,456)&(655,43)&(62,397)&(350,949)&(224,424)\\
 4&(394,779)&(165,206)&(671,577)&(653,869)&(871,634)\\
 5&(695,1007)&(243,824)&(734,374)&(792,43)&(370,183)\\
 6&(461,528)&(350,60)&(773,520)&(543,63)&(492,183)\\
 7&(572,539)&(231,649)&(797,632)&(313,399)&(829,358)\\
 8&(297,479)&(228,395)&(166,422)&(67,456)&(133,267)\\
 9&(133,742)&(864,340)&(717,489)&(224,585)&(734,374).
\end{array}
\]
Every point lies in \(E[316]\), where division by three means multiplication by \(211\).
Thus the reflection \(\rho\) in \(v=(1,1,\omega,\omega,\omega^2,\omega^2)\) acts on each tuple by
\[
 \rho(P)_i=P_i-[\omega^{e_i}][211]\sum_k[\omega^{-e_k}]P_k,
 \qquad (e_1,\ldots,e_6)=(0,0,1,1,2,2).
\]
Here \([\omega](x,y)=(\omega x,y)\), and put \(\tau(P)_i=P_i+b\).
For any of these tuples or their transforms, let \(c(P)\) be the kernel vector of the evaluation matrix on \((1,x,y,x^2,xy,x^3)\), normalized by \(c_0+2c_{x^3}=1\).
The kernels have dimension one and the normalizing denominators are nonzero.
For \(j=0,1,2\), order lexicographically the exponent vectors
\[
 \mathcal M_j=\{u\in\mathbb Z_{\ge0}^6:\textstyle\sum_i u_i=3,
 \ (0,2,0,1,2,0)\cdot u\equiv j\pmod3\}.
\]
Define the matrix \(M_j\) by the rows
\[
 (M_j)_{2a-1,u}=c(P^{(a)})^u-c(\rho P^{(a)})^u,
 \qquad
 (M_j)_{2a,u}=c(P^{(a)})^u-c(\tau P^{(a)})^u.
\]
The following minors use the first indicated rows and all columns except the listed ones; the last column gives their determinants modulo \(1009\).
\[
\begin{array}{c|r|r|c|r}
 j&|\mathcal M_j|&\text{number of rows}&\text{omitted columns}&\det\\\hline
 0&21&18&4,19,21&775\\
 1&17&15&13,17&457\\
 2&18&17&13&523.
\end{array}
\]
The chord and tangent law on \(y^2=x^3+1\) and ordinary determinant expansion give these residues.
The points and their transforms have torsion of order dividing \(948\), prime to \(1009\).
Thus the minors lift to characteristic zero, since every denominator is a unit.
The degrees of Rains's generators of the ring of invariant sections give total dimension six, so the upper character bounds \((3,2,1)\) are equalities.
Cubics in \(\theta_0,\theta_1\) account for \((2,1,1)\), leaving new characters \(1,\omega\).
The polynomial degrees \((4,6,10,12,18)\) of the basic invariants then give the asserted characters in degrees four and six as well.

We compute the projections of the differentials, keeping the polynomial degree distinct from the degree of a section.
Write \((c_0,\ldots,c_5)=(c_0,c_x,c_y,c_{x^2},c_{xy},c_{x^3})\).
Let \(\mathfrak n\) be the maximal ideal at the image of the origin in \(A/G_{33}\).
For \(k\geq1\), let \(F_k\subset\mathfrak n/\mathfrak n^2\) be the span of the differentials at the origin of \(s/\theta_0^m\), where \(s\) is an invariant section of degree \(m\leq k\).
Rains's polynomial ring of invariant sections gives
\[
 (\dim F_1,\dim F_3,\dim F_4,\dim F_6)=(1,3,4,5).
\]
We verify that projection onto the classes of the differentials of the first \(1,3,4,5\) basic invariants, in increasing polynomial degree, has the corresponding rank.

We first specify finite inputs over \(\mathbb F_{1009}\).
List the \(315\) nonzero points of \([3]E(\mathbb F_{1009})\) as \(P_0,\ldots,P_{314}\), in lexicographic order of their coordinate representatives in \(\{0,\ldots,1008\}^2\).
Put
\[
 (m_1,\ldots,m_5)=(1,7,31,73,137),\qquad
 (b_1,\ldots,b_5)=(0,17,51,83,121).
\]
For \(0\leq a<315\), take the first five points of a tuple to be \(P_{am_i+b_i\bmod315}\), and take the sixth point to be the negative of their sum.
Discard a tuple if the tuple or either of its transforms under \(\rho,\tau\) has a repeated point, contains the origin, or has no normalized coefficient vector \(c\).
There are \(285\) remaining tuples, ordered by \(a\).
For \(m=1,3,4\) and \(j=0,1,2\), let \(M_{m,j}\) be the matrix whose columns are indexed by the lexicographically ordered set
\[
 \mathcal M_{m,j}=\{u\in\mathbb Z_{\geq0}^6:\textstyle\sum_i u_i=m,
 (0,2,0,1,2,0)\cdot u\equiv j\pmod3\}.
\]
The \(k\)th remaining tuple \(P\), \(1\le k\le285\), gives the rows \(2k-1\) and \(2k\), with entries \(c(P)^u-c(\rho P)^u\) and \(c(P)^u-c(\tau P)^u\) in the column \(u\), where \(c^u=\prod_ic_i^{u_i}\).
In each matrix, scan the columns in increasing order, choose the first available nonzero row after elimination at each pivot, and retain the original row indices in pivot order.
The selected rows and pivot columns give the following nonzero determinants:
\[
\begin{array}{c|c|c|c|c}
 m&|\mathcal M_{m,j}|&\operatorname{rank}M_{m,j}&\dim\ker M_{m,j}&\det\\\hline
 1&(3,1,2)&(2,0,2)&(1,1,0)&(544,1,160)\\
 3&(21,17,18)&(18,15,17)&(3,2,1)&(701,552,919)\\
 4&(45,39,42)&(41,35,40)&(4,4,2)&(580,438,964).
\end{array}
\]
The three entries in each tuple correspond to \(j=0,1,2\), and the empty determinant equals one.
The total kernel dimensions are \(2,6,10\), equal to the dimensions of the invariant sections given by Rains.
Every invariant section lies in the corresponding kernel, so equality of these total dimensions identifies all the kernels with the spaces of invariant sections.
Use the nullspace basis obtained by setting each free variable in turn to one and the others to zero in the reduced row echelon form.
Order these basis vectors by their free columns and concatenate the three characters in the order \(0,1,2\).
With indices starting at zero, select the sections
\[
 (m,i)=(1,1),(3,1),(3,3),(4,1).
\]

For the Taylor expansion at the origin of \(E\), use the formal coordinate \(z\) satisfying \(dz=dx/(2y)\).
Write
\[
 x(z)=z^{-2}\sum_{r\geq0}a_rz^{6r},\qquad y(z)=x'(z)/2,
 \qquad a_0=1,\quad a_1=-1/7.
\]
For \(r\geq2\), the coefficients satisfy
\[
 ((6r-2)(6r-3)-12)a_r=6\sum_{v=1}^{r-1}a_va_{r-v}.
\]
Put \(B(z)=(1,x(z),y(z),x(z)^2,x(z)y(z),x(z)^3)\).
The three choices of \((\lambda_1,\ldots,\lambda_6)\in\mathbb F_{1009}^6\) are
\[
\begin{gathered}
 (94,181,986,534,540,692),\quad (980,438,34,975,252,348),\\
 (902,849,41,952,851,441).
\end{gathered}
\]
Their coordinate sums vanish modulo \(1009\).
For each of them, let \(u\) be a formal parameter.
Solve \(\sum_{j=0}^5c_jB_j(\lambda_i u)=0\), \(1\le i\le6\), with \(c_0=1\), and divide the resulting vector by \(c_0+2c_5\).
The full monomial supports of the selected numerators have minimum weights \(4,6,10,12\) for the weights \((0,2,3,4,5,6)\) on \((c_0,\ldots,c_5)\).
Let \(h_4,h_6,h_{10},h_{12}\) be the corresponding first homogeneous terms.
Evaluation at the first \(1,1,2,3\) directions, respectively, gives
\begin{equation}\label{sn:g33-leading-minors}
\begin{array}{c|c|r}
 d&\text{rows of the evaluation matrix}&\det\\\hline
 4&h_4&895\\
 6&h_6&722\\
 10&h_4h_6,\ h_{10}&844\\
 12&h_4^3,\ h_6^2,\ h_{12}&987.
\end{array}
\end{equation}
Hence the classes of \(h_d\) are nonzero modulo products of invariants of smaller positive degrees.
Complete \(h_4,h_6,h_{10},h_{12}\) to homogeneous basic invariants by a generator of degree eighteen.
Modulo products of invariants of positive degree, the four selected differentials project to
\begin{equation}\label{sn:g33-projection-matrix}
 \begin{pmatrix}
 1&0&144&0\\
 0&1&0&217\\
 0&0&1&0\\
 0&0&0&1
 \end{pmatrix}
 \qquad\text{in the directions }(4,6,10,12).
\end{equation}
Every leading principal minor is one.

We pass to characteristic zero through the completion of \(\mathbb Z[\omega]\) at the prime defined by \(1009\) and \(\omega=374\).
The curve has good reduction, the quotient isogeny has degree three prime to \(1009\), and the quotient remains an abelian scheme with its ample descended line bundle.
Generic automorphisms and their inverses extend by the N\'eron mapping property of the abelian scheme.
Normalize the isomorphisms of the descended line bundle on the identity fiber; their cocycle equalities then extend, so the full linearization is retained.
Since \(1009\nmid|G_{33}|\), the Reynolds operator commutes with reduction.
It follows from the vanishing of higher cohomology for an ample line bundle on an abelian variety that the modules of invariant sections commute with reduction as well.
The formal logarithm through polynomial degree twelve has integral coefficients at this prime.
The polynomial invariant rings on both fibers have basic invariants of degrees \((4,6,10,12,18)\).
Lift homogeneous generators by the Reynolds operator and apply Nakayama's lemma in each degree; the indecomposables in these degrees are therefore free modules of rank 1.
Thus the nonzero minors in \eqref{sn:g33-leading-minors} and \eqref{sn:g33-projection-matrix} are reductions of minors of the actual maps to \(\mathfrak n/\mathfrak n^2\), and the projections over \(\mathbb C\) have ranks \(1,3,4\).
The chart \(\theta_0\ne0\) is an affine space of dimension five, with coordinates given by the other five generators divided by the corresponding powers of \(\theta_0\), so \(F_6=\mathfrak n/\mathfrak n^2\).
A change of homogeneous basic invariants multiplies each class of degree \(d\) by a nonzero scalar, since the degrees of the five basic invariants are distinct.
The projection ranks therefore hold for the basic invariants \(p_4,p_6,p_{10},p_{12},p_{18}\) of \(G_{33}\) used in Appendix~\ref{sa:g33-flag}.
The projections in question are isomorphisms; higher polynomial components of the differentials need not vanish.
These projections give the required successive elimination of coefficients in the directions of degrees \(6,12,18\), by division by powers of \(q_*\) with exponents \(3,4,6\), and the analogous projections in the scalar character \(\omega\).

For completeness, the group of common fixed points of \(G_{33}\) on \(A\) is the diagonal image of \(E[2]\).
Uniformize \(E\cong\mathbb C/\Lambda\) by the differential \(dx/(2y)\), with \(\Lambda=\Omega\mathbb Z[\omega]\), where \(\Omega\) is chosen so that \(b\equiv\Omega/(1-\omega)\bmod\Lambda\); then \([\omega]\) is \(z\mapsto\omega z\).
An element \(\lambda\in\Lambda\) represents the point \({\lambda/6\bmod\Lambda}\) of \(E[6]\), and \(1-\omega\mid\lambda\) means \(\lambda\in(1-\omega)\Lambda\).
The lift \((\lambda/6,\ldots,\lambda/6,-5\lambda/6)\) changes under the additional reflection by \((\omega\lambda/3)v\).
This change vanishes modulo \(\langle(b,\ldots,b)\rangle\) exactly when \(1-\omega\mid\lambda\); the Weil pairing gives the same kernel of the polarization, of order four.
At a point with trivial character on the fiber of \(\mathcal V_*\), a lift of its translation to the theta group commutes with \(G_{33}\): as in the proof of Proposition~\ref{sc:root-theta}, its scalar commutator with each element is the ratio of the characters at that point and at zero.
Conjugate the scalar by that lift to transport its eigenspaces.
Therefore the spaces of derivatives have projections of the same ranks at all relevant points, where the components of these spaces along the differentials of basic invariants of higher degree are retained.
The derivatives in these directions first occur for generators of degrees \(1,3,3,4,6\), which gives the stated powers of \(q_*\).
\end{proof}

\section{Curves in projective space obtained from disks fixed by finite reflection groups}
\label{ap:projective}
The comparisons in this appendix concern the class \(L\) of the original lattice, including the finite characters of \(L\).
They do not assume that \(K_W\) plus any multiple of \(L\) is nef on the quotient of ten points.
Write $H=c_1(\mathcal O(1))$.

\begin{lemma}\label{st:euler}
If $f:C\to\mathbb P^k$ is nonconstant, $u=\deg f^*\mathcal O(1)>0$, $Q_h$ is a line bundle on $C$ of degree $h$, and $f^*T_{\mathbb P^k}\to Q_h$ is nonzero, then $h\ge u$.
The corresponding map $\kappa:C\to(\mathbb P^k)^*$ to the dual projective space, after common zeros are removed, has degree at most $h-u$.
If $\kappa$ is nonconstant and its derivative vanishes at $R$ distinct points, then
\begin{equation}\label{st:dual}
 R\le2(h-u)+2g(C)-2.
\end{equation}
A constant $\kappa$ is a hyperplane containing $f(C)$.
\end{lemma}
\begin{proof}
The Euler sequence represents the map by the sections of $Q_h\otimes f^*\mathcal O(-1)$ given by the \(k+1\) summands \(\mathcal O(1)\) of its middle term.
After their common zeros are removed, these sections define $\kappa$; thus $\kappa^*\mathcal O(1)$ is $Q_h\otimes f^*\mathcal O(-1)$ twisted by minus the effective divisor of these zeros, which proves the degree statements.
A general projection of its image to $\mathbb P^1$ preserves its degree and every prescribed zero of its derivative.
Riemann--Hurwitz gives~\eqref{st:dual}.
\end{proof}

\begin{proposition}\label{st:p6}
If the Baily--Borel compactification of an arithmetic ball quotient of dimension six is $\mathbb P^6$ and $L=\lambda H$, then $\lambda\ge1$.
The arithmetic lattice defining this quotient may be cocompact or nonuniform.
\end{proposition}
\begin{proof}
Assume $0<\lambda<1$.
On every curve obtained from a disk fixed by a finite reflection group, $a=7u$, $t=\lambda u$, and each nonzero row of the inverse differential has $h\ge u>t$.
The classification of parabolic subgroups~\cite[Tables 10--12]{KrishnasamyTaylor}, together with Proposition~\ref{sg:persistent}, leaves parabolics of $A_6,B_6,D_6,G(3,3,6),E_6,G_{34}$ and the $G(m,p,5)$ curves with $m\ge4$ or $(m,p)=(3,1)$.
We exclude every irreducible factor of rank at most 4 satisfying the condition on full interior stabilizers in Lemma~\ref{sg:fixed}, since such a factor contradicts the ample tangent bundle on a complete curve outside the boundary.
For a $G(m,p,5)$ curve without ends, we obtain from the line bundle corresponding to \(P_0\) and the Riemann--Hurwitz formula \(2g-2+(1-1/m)R+\sigma=2t\) for \(C\)
\[
 (3m-5)j+(m-1)pu+2mg+m\sigma\le2m,\qquad j=mt\in\mathbb Z_{>0}.
\]
For $m\ge4$ this is impossible already with $j,u\ge1$.
The exceptional $(3,1)$ equality has $h=-1$.
For a $G(m,p,5)$ curve with ends, the line bundle $w_KL$ gives $t\ge u/w_K$.
The computation with \(P_0\) for \(m=6\), with all polarizations and widths retained, gives
\[
 a\ge(23p/6+1)h_P+53t+R+\sigma
 \ge(23p+59)u/6>7u.
\]
For the groups $G(3,1,5)$ and $G(4,1,5)$ the same computation gives
\[
 a\ge23t+14h/3+R+\sigma\ge17u/2,\qquad
 a\ge33t+19h/4+R+\sigma\ge13u.
\]
Here respectively $6L$ and $4L$ are line bundles, and the relevant rows of the inverse differential, for \(P_0\) or for the invariant of greatest degree, are nonzero.
For the remaining groups over the Gaussian field, \(P_0\) gives a shorter argument than a finite enumeration.
After the group $G(4,1,5)$ is excluded, the $G(4,2,5)$ ends have width one and types $A=(1^3,2^3)$, $B=(1^2,2^4)$, $C=(1,2^2,4^3)$, $D=(1,2,4^4)$.
Their sums $B_c$ and the pole orders $b_c$ in the direction of \(P_0\) are $(9,1),(10,1),(17,2),(19,2)$.
The torsor has no point fixed by all of $G(4,2,5)$, so it is not an end of this curve.
We justify the projection onto the direction of \(P_0\) in the Gaussian model $Y^2=X^4-Z^4$.
Write $\mathsf X=\prod_iX_i$, $\mathsf Y=\prod_iY_i$, $\mathsf Z=\prod_iZ_i$, $z_i=Z_i/X_i$, $f_j=e_j(z_1^4,\ldots,z_5^4)$, and
$P=(\prod_i z_i)^{4/p}$, where $p=2$ or $4$.
The symmetric theta ratios and the product ratio satisfy
\[
 \frac{\mathsf Y}{\mathsf X^2}
 =(1-f_1+f_2-f_3+f_4-P^p)^{1/2},
 \qquad
 \frac{\mathsf Z^{4/p}}{\mathsf X^{4/p}}=P .
\]
Since $p>1$, their linear parts separate $dP$ from the direction of every elementary symmetric invariant.
Thus the local lattice of the conormal bundle on the coarse curve, after removal of the radial direction, is $q^{b_c}\mathcal O\,dP\oplus\Lambda_{\rm sym}$.
Changing to additive elliptic coordinates is triangular on the symmetric invariants and a nonzero scalar on $P$, so this is the original root representation.
At an interior point with stabilizer $G(4,p,6)$, $P=u^{4-4/p}y$, $v=u^4$, and \(P_0\) for $G(4,p,6)$ is $Z=vy$; the block of the elementary symmetric invariants in the normal directions is $vI+J_4$.
The equation $g_1=v+f_1$ and the generators of degree one that are nonzero at the ends make the curve immersed at the points with stabilizer $G(4,p,6)$ and at the ends.
There is no first-order product mixing, and product inertia is trivial on the root space.
The same rational projection onto the direction of $dP$ therefore preserves every local lattice on the coarse curve.
At a point with stabilizer $G(4,p,6)$, the line bundle corresponding to \(P_0\) on the coarse quotient contributes $-(1-1/p)$, and the pullback of \(P_0\) has one zero, so that together they contribute $1/p$.
Consequently the nonzero quotient associated with \(P_0\) has degree $h_P=-20t/p+R/p+\sum b_c$.
For $p=2$, put $\beta=\sum b_c$, so $h_P=-10t+R/2+\beta$.
The bounds $B_c\ge8b_c+1$ give
\[
 a\ge8h_P+28t+s+5R/4+\sigma>7u,
\]
because Lemma~\ref{st:euler} gives $h_P\ge u$.
At the point over $Z=0$ fixed by all of $G(4,2,5)$, the invariant nonzero section of degree one is $\mathsf X^2$ for $A,B$, and $\mathsf X$ for $C,D$.
The character on the fiber of \(\mathcal V_*\) is trivial there, so Lemma~\ref{sg:cusp-lift} supplies the excluded $G(4,2,5)$ curve.
This removes the entire $A$--$D$ cusps, including the opposite side of $C,D$ where the kernel is $G(4,4,5)$ and the width is two.
Only $E=(1^2,2,4^3)$ and the torsor $(1^2,4^4)$ remain for $G(4,4,5)$, both with width one.
At every relevant point their invariant section $\mathsf X$ is nonzero of degree one, giving trivial character on the fiber of \(\mathcal V_*\).
Here $b_c=1$, $B_c\ge16$, and the same quotient associated with \(P_0\) has $h_P=-5t+R/4+s$ and
\[
 a\ge16h_P+33t+R+\sigma>7u.
\]
These local computations with the quotient associated with \(P_0\) use neither descent of $2L$ nor nefness of \(K_W\) plus a multiple of \(L\).
They retain the original actions of the groups \(F_P\), the scalar characters, and the central translations at the cusps.
In particular the preliminary finite tests for the larger values of \(m\) are unnecessary.

All remaining reflections have order two, and the orders of stabilizers of rank 2 divide $24$.
Thus $14L=bH$ is integral, and Proposition~\ref{sc:chern-integrality} gives $24c_2^{\rm orb}=504L^2=(18b^2/7)H^2$ integral.
It follows that $7\mid b$ and $\lambda=1/2$.
Consequently $u=2t$, $a=14t$, and $t\in\frac12\mathbb Z_{>0}$.
For $G_{33}$, lifting the entire group at a cusp of width one gives
\[
 51s+38R=192t-6,\quad s+5R/6\le2+2t,\quad t=1/2,1.
\]
Neither value has an integral solution.
For $B_5$, the four possible pairs at an end are
\[
 (B_c,\nu_c)=(6,1),(9,2),(10,2),(11,2),
 \qquad B_c\ge9\nu_c/2.
\]
Equation~\eqref{ew-comparison-canonical} and $h=-10t+R+\sum\nu_c$ therefore give
\[
 a\ge13t+9h/2+R+\sigma\ge22t>14t .
\]
If an $E_6$ point exists, take a $D_5$ curve through it.
Let $R,D$ count its $E_6,D_6$ points.
We divide over \(\mathbb C\{\xi\}\) the row of the inverse differential associated with the invariant of greatest degree by the common order of the zeros and poles of the entries of this row; the resulting row has direction
\[
 dT_{12}+\xi^3dT_9/2+\xi^6dT_6/4.
\]
Since the ratio $\xi^3/2$ does not lie in the image $\mathcal A=\mathbb C\{\xi^2,\xi^5\}$ of the ambient local ring, there is no constant map $\kappa$ as in Lemma~\ref{st:euler}; this remains true under $\mathrm{GL}(6,\mathcal A)$ coordinate changes.
Both the image and the row have zero derivative at $\xi=0$, so~\eqref{st:dual} applies at all $R$ points.
The degree equations give
\[
 16R+8s+5D+\sigma=41t,\quad h-u=(t-D-\sigma)/4,
\]
and their excess over~\eqref{st:dual} is $(t+8s+11D+23\sigma)/16>0$.
After these exclusions, the ends with $F_P=G(3,3,5)$ have width one and degrees $(1^3,3^3)$.
The nonzero target line bundle corresponding to \(P_0\) has degree $h_P=-5t+R/3+s$.
Equation~\eqref{ew-comparison-canonical} and the Riemann--Hurwitz formula for \(C\) give
\[
 a=12h_P+23t+R+\sigma\ge47t>14t ,
\]
using $h_P\ge u=2t$ from Lemma~\ref{st:euler}.
For $D_5$ there are now no points with stabilizer $E_6$; its cusps have width one and, for $F_P=D_5$, generators of degrees $(1^4,2^2)$.
We obtain the quotient associated with \(P_0\) from the direction of the odd theta function, taken separately, and from the frame $P_0/u$ on the coarse quotient at $D_6$; this quotient has degree $h_P=-5t+D/2+s$.
Here $D$ counts $D_6$ points.
Equation~\eqref{ew-comparison-canonical} and the Riemann--Hurwitz formula for \(C\) give
\[
 a=8h_P+13t+D+\sigma\ge29t>14t ,
\]
again with $h_P\ge u=2t$ from Lemma~\ref{st:euler}.
Finally $A_5$ gives $6h=a-14t-\sigma\le0$, contrary to $h\ge u$.
\end{proof}

\begin{lemma}\label{st:curvature}
Let $Z=\mathbb P^k$, $k\ge2$, be a closed locus of coincident sections in the smooth GIT quotient \(W=X_\nu\) obtained by perturbing \(\mu\) to weights \(\nu\) in a nearby chamber, with $N_{Z/W}=\mathcal O(-1)^{\oplus c}$.
Assume $\pi|_Z$ is an isomorphism and its image meets the boundary in a finite set.
If $L|_Z=\lambda H$ and $\lambda<1$, then $Z$ is the Baily--Borel compactification of the quotient of the ball fixed by the corresponding finite reflection group by its full normalizer, and the automorphic line bundle of this quotient has class \(L|_Z\).
\end{lemma}
\begin{proof}
Since \(\pi|_Z\) is an isomorphism, the image of \(Z\) has the same normal bundle away from its finitely many cusps.
Lemma~\ref{sc:negative-normal} applies with the restriction of the class \(L\) of the original lattice.
The inequality \(\lambda<1\) forces \(s=c\), so the Baily--Borel compactification of the quotient of the fixed ball over \(Z\) by the full normalizer is \(Z\).
\end{proof}

\begin{proposition}\label{st:compact-projective}
For a compact arithmetic ball quotient with coarse space $\mathbb P^k$ and $L=\lambda H$, one has $\lambda\ge1/3$ if $k=4$ and $\lambda\ge2/3$ if $k=5$.
\end{proposition}
\begin{proof}
For every complete curve obtained from a disk fixed by a finite reflection group, use $a=(k+1)u$ and $h\ge u$.
In dimension five the direction of \(P_0\) for $G(m,p,4)$ yields
\[
 (4m-6)j+(m-1)pu+2mg+m\sigma\le2m.
\]
It excludes $m\ge4$ and $(m,p)=(3,1)$.
The remaining groups are $A_5,B_5,D_5,G(3,3,5),G_{33}$.
In these groups, the orders of stabilizers of rank 2 divide $24$.
The integral classes $12L$ and $24c_2^{\rm orb}=360L^2$ imply $\lambda\in\frac16\mathbb Z_{>0}$.
For $\lambda=1/6,1/3,1/2$, a $D_4$ curve, with $R,M$ counting $G_{33},D_5$ points, gives
\[
 6u+18t=11R/2+4M+\sigma,\qquad
 5R/6+M/2+\sigma=2t+2-2g.
\]
Thus $g=0,u\le2$ and $28R+21M=(36+96\lambda)u-12$.
Reduction modulo seven excludes $\lambda=1/6$ and $\lambda=1/2$.
The only remaining case is
\begin{equation}\label{st:d4-residual}
 \lambda=1/3,\quad u=1,\quad R=2,\quad M=0,\quad\sigma=1.
\end{equation}
Its exclusion by triality retains the characters of the triality automorphisms of \(D_4\).
Then a $B_4$ curve obeys $(1+4\lambda)u+4g\le4$; integrality of $2t$ excludes $\lambda=1/6,1/3$, and $\lambda=1/2$ requires $4R=16u-2$.
With $G_{33}$ absent, the inequality in dimension five for $G(m,p,4)$, applied to $G(3,3,4)$, gives $u+j\le1$.
Finally $A_4$ gives $5h=6u-9t-\sigma<5u$.

For $k=4$ the analogous inequality $(3m-5)j+(m-1)pu+2mg+m\sigma\le2m$ excludes $m\ge5$ and $G(4,1,4)$.
The remaining full stabilizers are
\[
\begin{gathered}
 A_4,B_4,D_4,G(3,1,4),G(3,3,4),G(4,2,4),G(4,4,4),\\
 G_{28},G_{29},G_{30},G_{31},G_{32}.
\end{gathered}
\]
Now $30L=bH$ and $25920c_3^{\rm orb}=-259200L^3$ give $5\mid b$, hence $\lambda\in\frac16\mathbb Z_{>0}$.
At $\lambda=1/6$, the $G_{25}$ curve through a $G_{32}$ point satisfies
\[
 5u+29t=28R/3+\sigma,\quad5R/6+\sigma=2t+2-2g,
\]
leaving $u=R=2,\sigma=1$.
Its image is a smooth conic.
The effective action on the line fixed by \(G_{25}\subset G_{32}\) has order six, so for a coordinate \(x\) on this line, \(v=x^6\) is a coordinate on the coarse curve.
The restrictions of the basic invariants of degrees $12,18,24,30$, when nonzero, have orders $2,3,4,5$ in \(v\); hence the differential of the curve vanishes at \(v=0\), a contradiction.
For $G(3,1,3)$ the equations
\[
 5u+20t=11R/3+\sigma,\quad2R/3+\sigma=2t+2-2g,\quad-9t+R\ge u
\]
leave $u=1,R=2<5u/2$.
No reflections of order three remain, and the integral bundle $10L=2K_{\mathbb P^4}+B$ excludes $10\lambda=5/3$.
\end{proof}

\begin{proof}[Exclusion of~\eqref{st:d4-residual} by triality]
The curve is a line, with normal bundle $\mathcal O(1)^{\oplus4}$.
Its two $G_{33}$ points are $p,q$; at its other two points with nontrivial inertia, the inertia has order two and acts trivially on the root space.
For $D_4$ put $e_j=e_j(x_1^2,\ldots,x_4^2)$ and
\[
 h_4=e_2-e_1^2/4,\quad P=x_1x_2x_3x_4,\quad
 h_6=e_3-e_1e_2/6+e_1^3/48.
\]
The outer $S_3$ preserves $h_6$ and $h_4^2+12P^2$.
Its eigenlines for the action of order three are $h_4\pm2\sqrt{-3}P$.
The product relation in the fundamental group says that the two nontrivial triality automorphisms given by the local monodromy at $p,q$ are inverse to each other.
Thus these two eigenlines, though not preserved by the outer $S_3$, are globally defined here.
In terms of the line bundles for the invariants indexed by $2,4a,4b,6$, one has $E_{4a}E_{4b}=E_2E_6$, including the scalar and automorphic characters.
Their degrees on the coarse curve are all $-2$.
In the differential of the quotient map, the columns associated with $E_2,E_{4a},E_{4b},E_6$ vanish at $p$ and $q$ to orders $(1,0,2,1)$ and $(1,2,0,1)$, respectively.
Consequently
\[
 S_2=E_2(p+q),\quad S_6=E_6(p+q),\quad
 S_{4a}=E_{4a}(2q),\quad S_{4b}=E_{4b}(2p)
\]
are line bundles of degree zero with $S_{4a}S_{4b}=S_2S_6$.
After the tangent direction is removed, the local blocks of size two by two have leading forms
\begin{equation}\label{st:triality-leading-blocks}
 \begin{pmatrix}a_2v&a_4\\a'_2v^2&a'_4v\end{pmatrix},
 \qquad
 \begin{pmatrix}b_4v^2&b_6v\\b'_4v^3&b'_6v^2\end{pmatrix},
 \qquad \frac{a_4b_4}{a_2b_6}=\frac1{12}.
\end{equation}
The ratio follows from the following polynomial identities.
Use \(\sum z_i=0\), the permutation action of \(S_6\), and the reflection in \((1,1,\omega,\omega,\omega^2,\omega^2)\), where \(\omega^2+\omega+1=0\).
Put \(p_j=\sum_i z_i^j\) and
\[
\begin{gathered}
 F_4=p_4-p_2^2/2,\qquad
 F_6=5p_2^3-30p_2p_4-20p_3^2+48p_6,\\
 F_{12}=300p_2^3p_3^2+315p_2^3p_6-990p_2^2p_3p_5
 -450p_2^2p_4^2-300p_2p_3^2p_4-90p_2p_4p_6\\
 \hfill{}+864p_2p_5^2-250p_3^4+840p_3^2p_6
 -540p_3p_4p_5+300p_4^3+72p_6^2.
\end{gathered}
\]
These three polynomials are invariant under the permutation action of \(S_6\) and the reflection in \((1,1,\omega,\omega,\omega^2,\omega^2)\).
Set \(\vartheta=\sqrt{-3}\), \(\omega=(\vartheta-1)/2\), \(s=\sum_{i=1}^4x_i\), and substitute
\[
 z_1=u+\vartheta s/6,\quad z_2=-u+\vartheta s/6,\quad
 z_{i+2}=x_i-(3+\vartheta)s/12\quad(1\le i\le4).
\]
Write \(H_A=h_4+2\vartheta P\), \(H_B=h_4-2\vartheta P\), and \(T_{12}=F_{12}+3F_6^2/4\).
Up to terms of degree at least seven in \(x_1,\ldots,x_4\), direct expansion gives
\[
\begin{aligned}
 F_4&=-2u^2e_1+\omega H_A,\\
 F_6&=16u^6+30(1+\vartheta)u^2H_B-36h_6+5e_1^3/4,\\
 T_{12}&=2700(1+\vartheta)u^8H_B+32400u^6h_6-1725u^6e_1^3.
\end{aligned}
\]
The terms in \(e_1^3\) disappear in the indecomposable normal directions.
Along the tangent direction the invariant \(F_6\) restricts to \(16u^6\), so use \(v=u^6\) and the frames \((e_1/u^4,H_A,H_B/u^4,h_6)\) on the coarse quotient.
Removing the tangent direction therefore gives the two blocks in \eqref{st:triality-leading-blocks} and the required product
\[
 \frac{a_4b_4}{a_2b_6}
 =\left(-\frac\omega2\right)\frac{1+\vartheta}{12}
 =\frac1{12}.
\]
Choose global frames compatible with $S_{4a}S_{4b}=S_2S_6$.
The four columns of the differential of the quotient map, with entries in $\mathcal O(1)^{\oplus4}$, then satisfy
\[
\begin{gathered}
 v_{4a}(p)=A_pv_2(p),\qquad v_{4b}(p)=B_pv_6(p),\\
 v_{4b}(q)=A_qv_2(q),\qquad v_{4a}(q)=B_qv_6(q),\qquad A_pB_p=A_qB_q.
\end{gathered}
\]
For $p=(1:0),q=(0:1)$ their determinant is a constant multiple of $(A_pB_p-A_qB_q)x^2y^2$, hence zero.
The map \(\mathbb B^5/D_4\to\mathbb B^5/\Gamma\) induced from the quotient by the full generic stabilizer \(D_4\) to the quotient by \(\Gamma\) is a local isomorphism at a general point of the curve, so its differential has rank 4 on the normal directions there, a contradiction.
\end{proof}

\begin{proposition}\label{st:p5-sixth}
A nonuniform arithmetic ball quotient with compactification $\mathbb P^5$ cannot have $L=H/6$.
\end{proposition}
\begin{proof}
The line bundle $w_KL$ forces \(K=\mathbb Q(\sqrt{-3})\).
For $m=6$ and rank 4, the comparison using \(P_0\) gives $a\ge(17p/6+1)h_P+30t+R+\sigma$.
Using $a=6u$, $h_P\ge u$, and $t\ge u/6$ gives a contradiction.
For $G_{32}$, equation~\eqref{ew-curve-determinant-equation}, the Riemann--Hurwitz formula for \(C\), and Lemma~\ref{st:euler} leave only $(g,u,t,s,\sigma)=(0,2,1/3,2,2/3)$.
The image is a smooth conic, but the degrees $(2,3,4,5,6)$ of the generators at cusps of width one make the differential zero.
For $F_4$, the types with width one are $A=(1,1,2,2,3)$ and $B=(1,2,2,3,4)$; the only remaining case is $(g,t,s_A,s_B,\sigma,u)=(0,1/2,1,2,0,3)$.
The flags defined by the theta functions associated with the root lattice project onto $\langle dp_2\rangle$, $\langle dp_2,dp_6\rangle$, and $\langle dp_2,dp_6,dp_8\rangle$.
These projections are nonzero since the matrices formed by the values of the theta functions at the fixed points and the Taylor coefficients of the lattice sums in degree $2$, in degrees $2,6$, and in degrees $2,6,8$ have determinants $3\vartheta_2^4\vartheta_3^4\vartheta_4^4/8$, $-5760\Phi\eta(\tau)^{12}\eta(2\tau)^{12}$, and
$-696729600\Phi^4\eta(\tau)^{16}\eta(2\tau)^{16}$, where $\Phi=(\vartheta_3^4+\vartheta_4^4)/2\ne0$ at a modulus \(\tau\) for which \(\mathbb Z+\tau\mathbb Z\) is homothetic to the ring of Eisenstein integers.
The final identity is an identity of weight 24 on $\Gamma_0(2)$; its Sturm bound is six, so the expansion of the lattice sums through $q^6$ is an exact finite verification.
Successive divisions by powers of \(q_*\) subtract $5,6,8,10$ from the degrees $3+d/2$ of the coefficients for the basic invariants of degree \(d\).
Thus $\dim\operatorname{Hom}(f^*T_{\mathbb P^5},\mathcal O(3))\le1$, but the Euler sequence gives dimension at least two.
For $G(3,1,4)$ only $u=2,t=1/3,R=0,s=2,\sigma=2/3$ remains, with both ends $(1^2,3^3)$.
Let $B_3$ be the total reduced branch divisor whose generic inertia has order three, and define $d_3\in\mathbb Z_{\ge0}$ by $\mathcal O_{\mathbb P^5}(B_3)=\mathcal O_{\mathbb P^5}(d_3)$.
The image of \(f\) lies in \(B_3\), so the differential of the section \(s_{B_3}\) of \(\mathcal O_{\mathbb P^5}(B_3)\) that cuts out \(B_3\) is defined along \(f\).
Restricted to the normal direction corresponding to \(P_0\), this differential maps the line bundle corresponding to \(P_0\), of degree $-12t=-4$, to $f^*\mathcal O_{\mathbb P^5}(B_3)$, with one interior zero and three at each end.
Since $\deg f^*\mathcal O_{\mathbb P^5}(1)=u=2$, counting these zeros gives $2d_3+4=7$, which is impossible.
The $B_4$ equations give $10t\le4(1-g)-2\sigma$ with $2t\in\mathbb Z_{>0}$, impossible.
For $D_4$ retain its points with stabilizers $G_{33}$ and $D_5$: $g=0,u,s\le2$ contradict $2s\equiv3u+2\pmod7$.
Only after excluding them use $G(3,3,4)$: $u=2,t=1/3$ requires $12s+5R=26$, $s\le2$.
Lemma~\ref{st:euler} excludes irreducible factors of lower rank satisfying the condition on full interior stabilizers in Lemma~\ref{sg:fixed}.
Finally $A_4$ gives $5h=27t-\sigma<5u$.
\end{proof}

\begin{proposition}\label{st:p5-quarter}
A nonuniform arithmetic ball quotient with compactification $\mathbb P^5$ cannot have $L=H/4$.
\end{proposition}
\begin{proof}
The field is Gaussian.
For $G_{31}$ with a common fixed point the only possible case is $(g,u,t,s,\sigma)=(0,2,1/2,3,0)$.
Let $e_j$ denote the elementary symmetric polynomial of degree $j$ in six variables, so $e_6(x)=x_1\cdots x_6$, and set $i_{2j}=e_j(x_1^2,\ldots,x_6^2)$ for $j=1,2,3$.
Near $(1:\cdots:1)$ in the sextic model $i_2^3-4i_2i_4+8i_6-16e_6(x)=0$, normalize $e_1(x)=6$, write $x_i=1+\xi_i$, and set $q_j=e_j(\xi_1,\ldots,\xi_6)$ for $2\le j\le6$.
The cotangent coordinates are ordered as $(q_2,q_3,q_5,q_6)$.
Up to a constant factor in each row, the following matrix gives the coefficients of the differentials of the ratios \(e_j(x)/e_1(x)^j\), for the generators \(e_2,e_3,e_4\) of degrees \(2,3,4\), in the directions of the basic invariants of degrees \(8,12,20,24\) on the normal space:
\[
 \begin{pmatrix}1&0&0&0\\4&1&0&0\\6&3&1/2&-1/4\end{pmatrix}.
\]
Linearizing the sextic equation gives $dq_4=dq_5/2-dq_6/4$.
After \(k\) successive divisions by powers of \(q_*\) at the three cusps, the coefficient for the basic invariant of degree \(d\) has degree $2+d/2-3k$; this degree is zero for \(d=8\), \(k=2\) and for \(d=20\), \(k=4\), so these two coefficients are constant.
Setting them to zero in turn bounds $\dim\operatorname{Hom}(f^*T_{\mathbb P^5},\mathcal O(2))$ by two; the Euler sequence gives at least three.
The $G_{31}$ torsor has sixteen points fixed by its normal subgroup $J_0=2^{2+4}$, and $G_{31}/J_0$ acts on them with orbits of sizes six and ten.
Its stabilizer orders $120,72$ are not divisible by $|G(4,2,4)/J_0|=48$.
It therefore contains no point fixed by all of $G(4,2,4)$.
For $G(4,2,4)$ with a fixed point, the four types with width one are $A=(1^3,2^2)$, $B=(1^2,2^3)$, $C=(1,2^2,4^2)$, $D=(1,2,4^3)$.
In each of these types, the degree of the generator associated with \(P_0\) is half the highest degree.
The seven remaining combinations of types at the ends are $C$, $AA$, $CD$, $AAD$, $BBC$, $CDD$, $CDDD$, with $u=1,1,2,2,2,3,4$ and $\nu=\sum_c\nu_c=4u$.
Every row has $R=0$ and a nonzero target line bundle corresponding to \(P_0\), of degree $-8t+\nu/2=0<u$.
The remaining factors are parabolics of $A_5,B_5,D_5,G(4,4,5),F_4,G_{29}$, but Lemma~\ref{st:euler} excludes irreducible factors of smaller rank satisfying the condition on full interior stabilizers in Lemma~\ref{sg:fixed}.
Here the orders of stabilizers of rank 2 divide $24$, so $24c_2^{\rm orb}=360L^2$ extends across the finite cusp set.
Its value on a plane is $45/2$, contradicting integrality.
\end{proof}

\subsection{Projective spaces of dimensions seven through nine}\label{sf:projective}

\begin{proposition}\label{sh:projective}
For \(7\le k\le9\), projective space \(\mathbb P^k\) is not the Baily--Borel compactification of an arithmetic quotient of \(\mathbb B^k\).
\end{proposition}
\begin{proof}
The argument for generation by reflections in the proof of Proposition~\ref{sg:arithmetic} applies to this smooth simply connected rational compactification. Retain the group over \(K\) generated by the normalized reflections, and put \(w=|\mu(K)|\).
On \(\mathbb P^k\), put \(a=(k+1)u\), \(u>0\).
Lemma~\ref{st:euler} gives \(h\ge u\) for every nonzero quotient \(f^*T_{\mathbb P^k}\to Q_h\).
We exclude every irreducible factor of rank at most \(k-2\) satisfying the condition on full interior stabilizers in Lemma~\ref{sg:fixed}, since such a factor contradicts this inequality on a complete curve avoiding the finitely many cusps.
The classification of parabolic subgroups thus leaves the curves fixed by subgroups of corank 1; we treat \(G_{34}\), then the monomial groups, then the ADE groups.
At a smooth cusp the link is a sphere, which makes the linear action of \(F_P\) on \(\mathbb C^r\) irreducible.

For \(k=7\), first remove the primitive group \(G_{34}\), before applying Proposition~\ref{sc:monomial-comparison} at its subgroups that can occur as \(F_P\).
Its polynomial degrees are \((6,12,18,24,30,42)\), the generators of its ring of invariant sections have degrees \((1,3,4,6,7,9,12)\), and
\[
 a=41s+2-2g-132t,\quad h=-42t+12s,\quad6t\in\mathbb Z_{>0}.
\]
There is no torsor, and the reflection divisors give a lift of the entire group with trivial character at the origin.
For \(G_{34}\) curves without ends, \(h<0\), which excludes them immediately.
For \(G_{34}\) curves with ends, the Riemann--Hurwitz formula for \(C\) reads \(2g-2+s+\sigma=2t\).
Its combination with \(h\ge a/8>0\) gives \(94t+55\sigma\le108(1-g)\), hence \(g=0\) and in particular \(6t\le9\).
In that complete finite range, \(a\in8\mathbb Z_{>0}\), \(a\le48t\), and \(h\ge a/8\) leave no possibility.
The lift of the entire group then excludes \(G_{34}\) cusps as well.

Now consider the monomial groups.
For both cocompact and nonuniform lattices, defer \(D_6\) and \(D_7\) until their possible points with stabilizers \(E_7\) and \(E_8\) have been removed by means of the \(E_6\) and \(E_7\) curves through them.
For curves without ends and every \(m\ge2\), the Riemann--Hurwitz formula for \(C\) gives
\[
 h\le-(r-1)j+\frac{2j+2m}{m-1}\le7-r,\qquad r=k-1\ge7.
\]
For \(r=6\), use \(h\le4-4j\) for \(B_6\) and \(h\le-5j+(2j+2m)/(m-1)\) for \(m\ge3\); these again contradict \(h>0\).
For curves with ends, use all normalized entries in Appendix~\ref{sf:records}, in descending order of \(m\).
For \(r=6,7\), omit \(D_r\) here and treat it with the other ADE groups after the non-ADE monomial groups have been removed.
The bound \(A\le rm+1\) of~\eqref{sf:cusp-bounds} and the Riemann--Hurwitz formula for \(C\) give
\[
 a\le(2rm+2)(1-g)-(D-2rm)t,
\]
since the omitted coefficient of \(R\) is \(\eta-r(m-1)=2r-1+1/p-rm\le0\).
As \(a=(r+2)u\ge r+2\) and \(D>2rm\) for \(r=6,7\), this proves \(g=0\) and
\[
 1\le j\le
 \left\lfloor\frac{mr(2m-1)}{D-2rm}\right\rfloor.
\]
These are the following complete \(j\)-ranges:
\[
\begin{array}{c|rrrrrrrrrr}
(m,p)&(2,1)&(3,1)&(3,3)&(4,1)&(4,2)&(4,4)&(6,1)&(6,2)&(6,3)&(6,6)\\ \hline
r=6&2&3&6&4&7&9&7&11&13&16\\
r=7&1&2&3&3&4&5&5&7&8&9 .
\end{array}
\]
In \eqref{sf:domain}, impose \(a\in(k+1)\mathbb Z_{>0}\), \(a\le(k+1)wt\), and \(h\ge a/(k+1)\).
For every listed pair there are zero survivors.
For these non-ADE monomial groups, comparison of group orders excludes ends whose groups are exceptional Weyl groups.
In rank 7, for example, \(m^6\,7!>576\,7!\) for \(m\ge3\), but \(|W(E_7)|/|G(2,1,7)|=9/2\) is not integral.

For \(r=8\), there is no point with a primitive stabilizer of rank 9 and
\[
 h\le2m(1-g)-(r-3)j+(2-m)R-m\sigma.
\]
The estimate is nonpositive unless \(g=0\) and \((m,j)=(3,1),(4,1),(6,1),(6,2)\).
For \((m,j)=(3,1)\), \(s=2\) gives \(R\le1,h\le0\), and \(s=1\) gives \(R\le2,h\le-2\).
For \((m,j)=(4,1),s=1\), one has \(h\le-1\).
For \((6,1),s=1\), one has \(h\le0\), but \(s=2\) is impossible because it requires \(5R/6+\sigma=1/3\).
For \((6,2)\), the two cases have \(h\le-6,-2\).
The sole remaining case \(m=4,j=1,s=2\) has \(h>0\) only when \(p>1,R=0,\nu_1=\nu_2=4\), both widths one, and \(B_c\ge28\).
Thus \(a\ge24\), contrary to \(a\le40t=10\).
This excludes all these curves with monomial full generic stabilizer and then the cusps at which \(F_P\) is such a group.

For \(k=7\), after these non-ADE groups have been removed, the orders of stabilizers of rank 4 in the remaining ADE factors divide \(2880\).
For \(16L=bH\), Proposition~\ref{sc:chern-integrality} gives \(1575b^4/512\in\mathbb Z\), hence \(8\mid b\) and \(a\le16t\).
An \(E_7\) point or \(E_6\) cusp gives an \(E_6\) curve with
\[
 2g-2+s+R/2+\sigma=2t,\quad
 a=11s+13R+2-2g-42t,\quad
 h=-12t+3R+3s,\quad2t\in\mathbb Z_{>0}.
\]
The coefficient \(13R\) includes the sign character of the odd invariant degrees.
Eliminating \(s\) in the inequality \(8h\ge a\) from Lemma~\ref{st:euler} gives \(10t+22\sigma\le42(1-g)\), hence \(g=0,2t\le8\).
The conditions \(a\in8\mathbb Z_{>0}\), \(a\le16t\) leave only \((t,s,R,\sigma,a,h)=(3/2,3,4,0,24,3)\), for which the line bundle associated with an invariant of odd degree would have degree \(-19/2\).
For \(D_6\), projection of the inverse differential to the directions of the elementary symmetric invariants gives \(5h=a-12t-R-\sigma\).
Together with \(h\ge a/8\), this requires \(a\ge32t\), a contradiction.
The last \(A_6\) curve has \(7h=a-20t-\sigma<0\).

For \(k=8\), the orders of stabilizers of rank 3 in the ADE groups divide \(24\).
The integral class \(18L=bH\) and \(24c_3^{\rm orb}=-2016L^3\) give \(28b^3/81\in\mathbb Z\), hence \(9\mid b\) and \(a\le18t\).
An \(E_8\) point or \(E_7\) cusp gives an \(E_7\) curve.
At \(E_8\), the local coordinate of the coarse curve is \(v=u^2\).
Homogeneity gives pole order six for the row of the inverse differential associated with the invariant of degree eighteen, and order 28 for the determinant; on the coarse quotient, the stabilizers act trivially on the fibers of the line bundles associated with the invariant directions.
Consequently
\[
 2g-2+s+R/2+\sigma=2t,\quad
 a=17s+28R-70t+2-2g,\quad
 h=-18t+6R+4s,\quad2t\in\mathbb Z_{>0}.
\]
Elimination gives \(9h=2a-18t-3R-2\sigma\).
The inequality \(9h\ge a\) from Lemma~\ref{st:euler} and \(a\le18t\) force \(R=\sigma=0,a=18t,s=5t,3t=2-2g\).
Thus \(g=0,t=2/3\), contradicting \(2t\in\mathbb Z\).
For the remaining \(D_7\) curve, \(6h=a-21t-R-\sigma\) gives \(a\ge63t+3R+3\sigma\), again impossible.
The last \(A_7\) curve has \(8h=a-27t-\sigma<0\).

For \(k=9\), the only remaining primitive group is \(E_8\).
Its tuple in \eqref{sf:root-records}, with width one and \(F_P=W(E_8)\), gives
\[
 2g-2+s+\sigma=2t,\quad a=29s+2-2g-128t,\quad
 2t\in\mathbb Z_{>0},\quad a\in10\mathbb Z_{>0}.
\]
The upper bound \(a\le60(1-g)-70t\) forces \(g=0,t=1/2,s=3\), and then \(a=25\), a contradiction.
Only type \(A\) groups remain in \(\mathbb P^9\).
Write \(20L=bH\).
Proposition~\ref{sc:chern-integrality} gives \(27b^2/20,9b^3/25\in\mathbb Z\), hence \(10\mid b\) and \(a/t\le20\).
The last \(A_8\) curve satisfies \(9h=a-35t-\sigma<0\), contrary to Lemma~\ref{st:euler}.
This completes the independent exclusions of \(\mathbb P^7,\mathbb P^8,\mathbb P^9\).
\end{proof}

\section{Finite ranges and the remaining cases of monomial groups}\label{sf:finite}
This appendix specifies the finite ranges of parameters and every rejection test used in Sections~\ref{st:ten} and~\ref{sh:section}.
All operations are on integers and rational numbers.
The bounds~\eqref{sf:domain}--\eqref{sf:sharp-labels} define finite sets of multisets of normalized entries and widths at the ends, and~\eqref{sf:e-domain} treats every value of \(e\), including unbounded intervals.
Repeated affine actions remain distinct geometric possibilities; actions with the same group, degrees of the generators of the ring of invariant sections, width, and condition on the existence of a common fixed point are counted as one row of the enumeration.
Every action giving the same row is excluded by the test specified for that row in Appendices~\ref{sf:large-rows}--\ref{sf:eleven-rows}.
The three finite ranges of parameters used in the induction are organized as follows; the counts refer to these rows, with all permitted values of \(e\) retained.
\begin{center}
\begin{tabular}{c|r|p{.4\linewidth}|c|c}
Points&Rows&Tests&Proof&Survivors\\\hline
\(n\ge13\)&170&Proposition~\ref{sc:bundle-tests}, widths and torsors&\ref{sf:large-rows}&0\\
\(n=12\)&117&The same tests, with the injection~\eqref{sc:normal-test}&\ref{sf:twelve-rows}&0\\
\(n=11\)&413&Tests using only \(a\le7wt\), then the nefness of \(K_W+10L\) and the tests that use it&\ref{sf:eleven-rows}&0
\end{tabular}
\end{center}
The last row is used in two stages, before and after \eqref{sh:eleven-nef} and the exclusion of \(E_8\); the stronger bound is not assumed in the first stage.

\subsection{Cusps and their widths}\label{sf:records}
For \(F_P\) of rank \(r\ge6\), append \(m\) to each prefix until its length is \(r+1\).
Rains's classification in the strongly crystallographic case gives the following degrees of the generators before normalization \cite[Section 4.4]{Rains}.
\[
\begin{array}{c|l@{\qquad}c|l}
(m,q)&\text{prefixes}&(m,q)&\text{prefixes}\\ \hline
(2,1)&(1,1,1),(1,1),(1),()&(6,1)&(1),(2,3),(2),(3),()\\
(2,2)&(1,1,1,1)&(6,2)&(1,3),(2,3,3),(3,3)\\
(3,1)&(1,1),(1),()&(6,3)&(1,2),(2,2,3),(2,2)\\
(3,3)&(1,1,1)&(6,6)&(1,2,3)\\
(4,1)&(1,2),(2,2),(1),(2),()&(4,4)&(1,1,2)\\
(4,2)&(1,2,2),(2,2,2),(1,1)^*,(1,2),(2,2)&&
\end{array}
\]
The \((6,1)\) prefixes \((2,3),(2),(3)\) occur three, two, and two times, respectively; there are 36 entries before normalization.
The asterisk denotes the nontrivial torsor, which has no point fixed by all of \(G(4,2,r)\).
The absence of such a point does not exclude smaller stabilizers or nontrivial characters on the fibers of \(\mathcal V_*\).

Here is the complete normalization rule.
First divide the tuple of degrees given by the entry by its greatest common divisor, obtaining \(a=(a_0,\ldots,a_r)\).
Put
\[
 h_i=\gcd(a_j:j\ne i),\qquad H_0=\operatorname{lcm}_ih_i.
\]
For every \(v\mid H_0\), the Veronese subring of degree \(v\) is a polynomial ring whose generators have degrees
\begin{equation}\label{sf:veronese}
 \alpha_i=\gcd(v,h_i)a_i/v .
\end{equation}
These are precisely the Veronese subrings that are polynomial rings.
The coordinate reflection subgroups have pairwise coprime orders \(\gcd(v,h_i)\), and Chevalley--Shephard--Todd gives the criterion.
Removing first a group of scalars of order \(g\) acting only on the fibers of \(\mathcal V_*\) replaces \(q_{\rm old}\) by \(q_{\rm old}^g\) and divides \(c_*\) by \(g\).
This gives 64 entries after normalization and 32 distinct tuples with their group parameters retained.
Every normalized entry satisfies
\begin{equation}\label{sf:cusp-bounds}
 (r-1)b\le A=\sum_i\alpha_i\le rm+1,\qquad b=\max_i\alpha_i\le m.
\end{equation}
Before normalization, the greatest degree occurs at least twice and \(\sum_i(b-\alpha_i)\le2b\); applying \eqref{sf:veronese} proves the same bounds after normalization.

The lifts of the entire groups at the cusps follow directly from the reflection roots.
For \(G(m,q,r)\), consider the reflections exchanging adjacent coordinates up to a factor that is one or a primitive \(m\)th root of unity; these reflections lift with their positive Gram matrix on \(\mathbb C^r\).
Their rational span fixes a Hermitian plane through the isotropic line of the cusp, and the full generic stabilizer of the corresponding disk is \(G(m,p,r)\), \(q\mid p\mid m\).
For \(q=1,m=2,3,4\), a diagonal image of order \(m\) supplies a lift of the entire group.
For \(m=6\), if the kernel of the character on the fiber of \(\mathcal V_*\) is \(G(6,p,r)\lhd G(6,q,r)\), this character has order exactly
\begin{equation}\label{sf:eisenstein-width}
 \kappa=p/q .
\end{equation}
Put \(\zeta_6=\exp(\pi i/3)\).
Let \(T_{ij}\) be the diagonal operator on \(\mathbb C^r\) that multiplies coordinates \(i,j\) by \(\zeta_6,\zeta_6^{-1}\), respectively, and fixes the other coordinates.
Let \(E_i\) be the projection onto coordinate \(i\).
For distinct \(i,j,k\), the identities \(2I-T_{ij}-T_{ij}^{-1}=E_i+E_j\) and \((E_i+E_j)(E_i+E_k)=E_i\) split the lattice \(\Lambda_P\) by coordinates.
Since \(1-\zeta_6\) is a unit, the kernel has a unique fixed point \(x\); therefore \(F_P=G(6,q,r)\) fixes it, \((F_P)_x=F_P\), and \(\operatorname{im}\chi_x\simeq F_P/G(6,p,r)\), of order \(p/q\).

The Weyl groups \(E_6,D_7,E_7,E_8\) lift in their entirety, have no torsor, and have width one, and the generators of their rings of invariant sections have degrees
\begin{equation}\label{sf:root-records}
 (1^3,2^3,3),\quad(1^4,2^4),\quad
 (1,1,2,2,2,3,3,4),\quad(1,2,2,3,3,4,4,5,6).
\end{equation}
For these groups \((s_\alpha-1)\beta=\alpha\) gives zero coinvariants and no invariant \(\operatorname{Pic}^0\).
The polarization is an integral power of the minimal one even at CM moduli.
Connected reflection divisors through the origin force trivial linearization, and the criterion using greatest common divisors after deleting one entry in \eqref{sf:veronese} forces that power to be one.
The positive rational root span supplies the Hermitian plane, and the trivial character of the full stabilizer on the fiber of \(\mathcal V_*\) gives width one.

\subsection{The finite range and the unbounded values of \texorpdfstring{\(e\)}{e}}\label{sf:finite-domain}
The proof of Lemma~\ref{sa:g33-theta} specifies the computation over \(\mathbb F_{1009}\) and the projections of the Taylor expansions for \(G_{33}\). For the rows in the table, \eqref{sf:domain}--\eqref{sf:e-domain} and Appendices~\ref{sf:large-rows}--\ref{sf:eleven-rows} supply the search ranges and rejection criteria.

For fixed \(r,m,p,j\), after \(g=0\) has been proved, enumerate
\begin{equation}\label{sf:domain}
 1\le s\le\lfloor2+2j/m\rfloor,\quad
 0\le R\le\left\lfloor\frac{2+2j/m-s}{1-1/m}\right\rfloor,\quad
 \sigma=2+2j/m-s-(1-1/m)R .
\end{equation}
The allowed values of \(\sigma\) are sums of \(1-1/e\), \(1<e\mid w\), with at most \(\lfloor2\sigma\rfloor\) summands.
Use \(w=4\) for \(m=4\), \(w=6\) for \(m=3,6\), and the union of the three possibilities \(w=2,4,6\) for \(m=2\).
At each end independently use every normalized entry in Appendix~\ref{sf:records} with \(q_c\mid p\) and every \(\kappa_c\mid p/q_c\), before applying the stricter exclusions using the widths \(\kappa_c\).
The quantities \(B_c,\nu_c,a,h\) are those in \eqref{sh:monomial-log}--\eqref{sh:monomial-h}.
When all widths are one, the line bundle corresponding to \(P_0\) imposes
\[
 rj/p+(1-1/p)R\in\mathbb Z.
\]
Take only integral \(a\), impose \(a\le rwt\), and impose even \(a\) for twelve points.

For completeness, the coefficient bound used in the enumeration is as follows.
Let \(v_c\) be the zero forced at the \(c\)-th end after the map kills \(df\), and set \(V=\sum_cv_c\), \(\nu=\sum_c\nu_c\).
If \(\kappa_c\ne1\), use \(v_c=1\); if \(\kappa_c=1\), use \(v_c=\min_i\alpha_{c,i}\).
In the latter case, if the tuple has a unique generator of degree one and the other degrees have a common divisor greater than one, the equation in the radial direction instead forces the least of those other degrees.
Nonvanishing follows from the description of the width as a greatest common divisor in Lemma~\ref{sc:complete-curves}.
The coefficients for the elementary symmetric invariants have degrees
\[
 R+\nu-V-(r-1-i)j\quad(p>1),\qquad
 R+\nu-V-(r-i)j\quad(p=1).
\]
The triangular differential at the points with a larger irreducible stabilizer imposes \(R\) zeros.
There is at most one extra coefficient for \(P_0\).
Thus the bounds
\begin{equation}\label{sf:labels}
 \ell\le\left\lfloor\frac{\nu-s}{j}\right\rfloor+3,\qquad
 \ell\le\left\lfloor\frac{\nu-V}{j}\right\rfloor+2\quad(p=1)
\end{equation}
are valid.
The sharper count used in Appendices~\ref{sf:large-rows}--\ref{sf:eleven-rows} is
\begin{equation}\label{sf:sharp-labels}
 \ell\le1+\max\left(0,\min\left(q_{\rm sym},
 \left\lfloor\frac{\nu-V}{j}\right\rfloor+1\right)\right)+\epsilon_P ,
\end{equation}
where \(q_{\rm sym}=r\) for \(p=1\), and \(r-1\) otherwise.
Here \(\epsilon_P=0\) for \(p=1\); for \(p>1\), take \(\epsilon_P=1\) unless the following test forces \(\epsilon_P=0\).
When all widths are one and \(p>1\), the line bundle corresponding to \(P_0\) has degree \(E_P=-rj/p-(1-1/p)R\), and \(h-E_P-V<R\) forces \(\epsilon_P=0\).

There is no search cutoff on \(e\).
For each permitted \(\ell\), solve exactly
\begin{equation}\label{sf:e-domain}
 e\ge\max(1,-h),\qquad
 a-(n-2\ell)e\in2\mathbb Z_{\ge0}.
\end{equation}
If \(n-2\ell>0\), this is a finite interval with its parity condition.
If \(n-2\ell=0\), it is \(\{e\in\mathbb Z:e\ge\max(1,-h)\}\) when \(a\in2\mathbb Z_{\ge0}\), and empty otherwise.
If \(n-2\ell<0\), it is an unbounded half-line with the indicated congruence.
When the coefficient is even and \(a\) is even, both parities of \(e\) must be retained.
Condition~\eqref{sc:sorted-test} is monotone in \(e\) for both injections of Proposition~\ref{sc:bundle-tests}, so failure at the smallest permitted integer excludes its whole interval or half-line.
In particular, no positivity condition on \(a\) is added to the enumeration.

\subsection{Ranks at least 9: the 170 rows}\label{sf:large-rows}
The analytic range proved in Proposition~\ref{sh:monomial} is
\[
 9\le r\le28,\qquad
 1\le j\le
 \begin{cases}5&r=9,\\4&r\ge10.\end{cases}.
\]
With \eqref{sf:sharp-labels}, the domain \eqref{sf:domain}--\eqref{sf:e-domain} gives
\[
\begin{array}{c|rrrr}
r&9&10&11&14\\ \hline
\text{rows}&127&41&1&1
\end{array}
\]
and zero rows in every other rank.
At rank 9 the nonzero counts by \((m,p)\) are \(3,4,4,30,3,4,79\) at \((3,1),(3,3),(4,2),(4,4),(6,1),(6,2),(6,6)\), respectively.
The remaining comparisons using the tangent bundle have the following inputs.
Here \(\rho\) is the cover degree and \(\beta\) its total allowed pole order.
\[
\begin{array}{c|c|c|c|c|c|c|c|c}
r&(m,p)&j&s&\rho&\beta&a&\ell&e_{\min}\\ \hline
9&(4,2)&2&3&1&12&17&5&5\\
9&(4,4)&2&3&2&20&6,7,8,9&6&6\\
9&(6,1)&1,2,3&j&1&6j&7j+2&\le6&3j\\
9&(6,6)&3&3&2&30&8,7&7&10,9\\
9&(3,1)&1&2&1&6&7&6&3\\
9&(3,3)&1&2&1&6&9&5,6&3\\
10&(6,2)&2&2&1&12&12&6,7&6\\
10&(4,4)&1&2&2&12&8,10&6,7&3\\
10&(4,4)&2&3&1&12&12,14,16,18&6&6
\end{array}
\]
For the comparison using the tangent bundle, sort
\[
 2+\beta,\quad \beta-\rho ij\ (1\le i<r),\quad
 \beta-\rho rj/p,
\]
with the last entry interpreted as the \(i=r\) symmetric entry when \(p=1\).
The required positive rank and degree are \(n-\ell-2\) and \(\rho a+(\ell-1)\rho e\).
Rows 1,2,3,5,7 have available/required ranks \((4,6),(3,5),(4,5),(4,5),(4,5)\).
Row 6 gives \(5<6\), or \(9+5e\ge24>23\).
Row 4 requires degree \(122>95\) in rank 4.
The last two rows give \(5<6\), or \(2a+6\rho e\ge52>45\).
Every comparison improves as \(e\) increases.
The other rows of rank 10 with \(m=6\) violate \eqref{sf:eisenstein-width}.
The unique rows of ranks 11 and 14 are
\[
 (m,p,j,s,R,\sigma,a,h)
 =(4,2,1,2,0,1/2,4,-4),\quad
 (4,2,1,2,0,1/2,10,-5).
\]
They require respectively one and two points fixed by all of \(G(4,2,r)\) in the nontrivial \(G(4,2,r)\) torsor, which do not exist.
Only the two rows of rank 9
\[
 (m,p,R,a,t,h)=(3,1,1,16\ {\rm or}\ 18,1/3,-2)
\]
remain.
They have two cusps of width one, whose rings of invariant sections have generators of degrees one and three.
The comparison using the normal bundle has target line bundles whose degrees, divided by \(\rho\), are \(7-i\), \(1\le i\le9\).
For \(\ell=5,\ldots,10\), required degrees \(22,24,26,28,30,32\) exceed available sums \(20,18,15,11,6,0\).
For \(\ell<5\), too few target summands have degree at least \(e\ge2\).
This completes all 170 exclusions.

\subsection{Rank 8: the 117 rows}\label{sf:twelve-rows}
For twelve points, \(a\) is even and \(a\le8wt\).
The bounds
\[
 h\le-5j+2m(1-g)+(2-m)R-m\sigma,\qquad
 a\le(16m+2)(1-g)-(D-16m)t
\]
show that \(g\ge1\) is impossible: counting coefficients as in Appendix~\ref{sf:finite-domain} gives \(\ell\le4\) and \(a>0\), but the second bound gives \(a<0\).
If \(h\ge0\), then \(j\le2\); all normalized entries and widths in that finite range contradict integral \(a\le8wt\).
If \(j\ge8\), then \(\ell\le6\) and \(a\ge0\), contrary to \(a\le16m+2-j(12+8/p)<0\).
Thus \(g=0,h<0,1\le j\le7\).
For this rank, replace the condition on the allowed values of \(\sigma\) by the weaker necessary condition \(\sigma=0\) or \(\sigma\ge1/2\), which every allowed value satisfies.
Using \eqref{sf:sharp-labels} gives
\[
\begin{array}{c|rrrrrrrrr}
(m,p)&(2,1)&(2,2)&(3,1)&(3,3)&(4,2)&(4,4)&(6,2)&(6,3)&(6,6)\\ \hline
\text{rows}&9&20&2&2&12&40&1&8&23 .
\end{array}
\]
The pairs \((4,1),(6,1)\) leave no row.
Two rows have \(a=0\), and 27 domains for \(e\) are unbounded; these are retained.

Apply Proposition~\ref{sc:bundle-tests} with the widths \(\kappa_c\) and the sharper comparison using the normal bundle whenever the end is immersed.
Lifts of the entire groups exclude \(G(4,1,8)\) and \(G(3,1,8)\) cusps after the comparisons on the curves obtained from the disks ending at these cusps.
The tuple \((1,1,4^7)\) is excluded only when its width requires a point fixed by all of \(G(4,2,8)\), which does not exist.
Equation~\eqref{sf:eisenstein-width} removes 29 rows with \(m=6\).
For the three remaining rows with \(m=6\),
\[
 T=G(6,3,8),\quad t=1/2,\quad s=3,\quad
 R=\sigma=0,\quad h=-6,\quad a=20,18,16,
\]
the end types are \(A=(1,1,3^7)\), \(B=(1,2,6^7)\), and \(C=(2,2,3,6^6)\), with combinations \(ABB,ABC,ACC\).
Width one forces a nonzero generator of degree one at \(A,B\), and nonzero generators of degrees two and three at \(C\).
Thus the critical divisor of \(f\) has degree \(c_f=0,1,2\), respectively.
The degrees of the normal targets in Proposition~\ref{sc:bundle-tests}, divided by \(\rho\), are
\[
 (12,9,7,6,3,0,-3,-6).
\]
Each row has \(\ell=5,e\ge6\), so the bundle \(\mathcal K\) of Proposition~\ref{sg:kernel} has rank 4 and degree at least \(42,39,36\), respectively.
All exceed the sum \(12+9+7+6=34\).
No assertion of immersion at those ends is required.

The two remaining comparisons over the Gaussian field using the normal bundles have
\[
\begin{array}{c|c|c|c|c}
T&a&j&s&\text{widths}\\ \hline
G(4,4,8)&6&1&2&(2,1)\\
G(4,4,8)&10&2&3&(1,1,1).
\end{array}
\]
On the double and identity covers the upper bounds on the degrees of the normal targets are \((8,6,6,4,2,0,-2,-4)\).
Only \(\ell=5\) and \(e'=4\) remain, where \(e'=\rho e=-\deg\varrho^*M\) for the cyclic cover \(\varrho\) of degree \(\rho\) and the line bundle \(M\) of degree \(-e\) along the coincident sections in Proposition~\ref{sg:negative}, with \(\rho=2\) in the first row and \(\rho=1\) in the second; moreover the bundle \(\mathcal K\) of Proposition~\ref{sg:kernel} has rank 4 and degree \(24=8+6+6+4\).
It would have to equal the four largest target line bundles.
The direction of \(P_0\) instead has pole at most one at every point fixed by all of \(G(4,4,8)\) with trivial character on the fiber of \(\mathcal V_*\).
Write
\[
 \Omega=\{z\in\mathbb Z[i]^8:\textstyle\sum_k z_k\in(1-i)\mathbb Z[i]\},
 \quad A=\mathbb C^8/\Omega,\quad
 M=p^*\mathcal O([0])^{\boxtimes8},
\]
where \(E_i=\mathbb C/\mathbb Z[i]\) and \(p:A\to E_i^8\) has degree two.
The character summands of \(p_*\mathcal O_A\) give odd and even theta products.
On each factor \(E_i\), let \(\theta_0\) be the odd theta section with a simple zero at the origin, and let \(\theta_h\) be the even theta section vanishing at \(h=(1+i)/2\) and not at the origin.
After pullback to \(A\), we regard \(\Theta_0(z)=\prod_{k=1}^8\theta_0(z_k)\) and \(\Theta_h(z)=\prod_{k=1}^8\theta_h(z_k)\), given by the two character summands, as sections of \(M\).
The ratio \(\Theta_0/\Theta_h\) has initial indecomposable term \(c\prod_k z_k\), \(c\ne0\).
Changing one coordinate sign separates this term from symmetric invariants.
The four fixed points are \(\Omega^\#/\Omega=K(M)\), with \(\Omega^\#=\mathbb Z[i]^8+\mathbb Z(1+i)(1,\ldots,1)/2\).
Consider a fixed point at which the group acts trivially on the fiber of \(M\); a lift of the translation by this point to the theta group preserves \(M\) and transports the calculation.
Thus, after localizing at a nonzero section of degree one and removing the radial direction, the differential of the contraction \(\mathcal V_*/F_P\to\operatorname{Spec}R_P\) of the cusp contains \(q\,d(\prod_k z_k)\) at each required fixed point.
This lowers the degree of the line bundle corresponding to \(P_0\) from six to zero on the double cover, or to at most three on the identity cover, contradicting its required minimum four.

The remaining \(B_8\) row is
\[
 a=8,\quad t=1/2,\quad s=3,\quad R=\sigma=0,\quad h=-2,
 \qquad\alpha_c=(1^3,2^6).
\]
For the minimal \(D_8\) polarization, three even theta functions and one odd theta function of degree one give the first two independent projections of first derivatives of invariant sections.
For a map to \(\mathcal O(-2)\), the first division by powers of \(q_*\) at the cusps leaves coefficients of degrees \(i-5\), so those first two projections vanish.
The equation in the radial direction forces a second division by powers of \(q_*\) in all coefficients of degree one.
The resulting degrees \(i-8\) leave only the row of the inverse differential associated with the invariant of greatest degree.
A Hom space of dimension one now gives \(a\ge8(-h)=16\), a contradiction.
All twenty \(D_8\) rows require a \(B_8\) end.
The complete rejection counts are
\[
\begin{array}{l|r@{\qquad}l|r}
\text{Proposition~\ref{sc:bundle-tests}}&16&B_8\text{ theta}&1\\
B_8\text{ ends}&20&G(3,1,8)\text{ tests/ends}&2+2\\
G(4,1,8)\text{ ends}&33&\text{nontrivial torsor/product tests}&9+2\\
\text{equation~\eqref{sf:eisenstein-width}}&29&\substack{\text{normal bundles over}\\\text{the field }\mathbb Q(\sqrt{-3})}&3 .
\end{array}
\]
These sum to 117.

\subsection{Rank 7: the 413 rows}\label{sf:eleven-rows}
Use only the bound \(a\le7wt\), deferring \(D_7\) when a point with stabilizer \(E_8\) is possible.
The inequalities
\[
 h\le-4j+2m(1-g)+(2-m)R-m\sigma,\qquad
 a\le(14m+2)(1-g)-(D-14m)t
\]
exclude \(g\ge1\) by \(\ell\le4\), and leave no row with \(h\ge0\) and \(j\le3\).
For \(j\ge8\), the inequalities \(a\ge-h\ge4j-12\), \(a\le86-49j/6\), and \(\ell\le5\) leave only \(m=p=6,j=8\), where \(-h\ge24>a\).
Hence \(g=0,h<0,1\le j\le7\).
Use the weaker necessary condition \(\sigma=0\) or \(\sigma\ge1/2\), which every allowed value of \(\sigma\) satisfies.
With \eqref{sf:sharp-labels}, the finite range gives
\[
\begin{array}{c|rrrrrrrrrr}
(m,p)&(2,1)&(2,2)&(3,1)&(3,3)&(4,1)&(4,2)&(4,4)&(6,2)&(6,3)&(6,6)\\ \hline
\text{rows}&21&43&2&5&1&51&132&18&16&124 .
\end{array}
\]
All 25 unbounded domains for \(e\) are retained.
The exclusions have two stages, with different available upper bounds on \(a\).
Before \eqref{sh:eleven-nef}, use only \(a\le7wt\); the following tests exclude 166 rows.
\[
\begin{array}{l|r|l}
 \text{groups}&\text{number}&\text{test}\\\hline
 G(6,p,7)&148&\kappa_c=p/q_c\\
 G(6,p,7)&10&\text{cyclic covers}\\
 G(3,1,7)&2&\text{normal bundle}\\
 G(3,3,7)&5&\text{excluded }G(3,1,7)\text{ ends}\\
 G(4,1,7)&1&\text{normal bundle}.
\end{array}
\]
These tests suffice for the list of possible factors used to prove \eqref{sh:eleven-nef}; no row is excluded here using \(a\le10t\).
Keep the other 247 rows, including \(D_7\), until \eqref{sh:eleven-nef} and the exclusion of \(E_8\) have been proved.
The new bound excludes 205 of these rows by \(a>10t\).
The remaining 42 rows are excluded by a \(B_7\) end in one case, a \(G(4,1,7)\) end in 25 cases, absence of a point fixed by all of \(G(4,2,7)\) in the nontrivial \(G(4,2,7)\) torsor in eight cases, and the comparison using the normal bundle in eight cases.
The total is \(166+205+42=413\).
The ten rows tested using cyclic covers have \(R=0\) and at most two nontrivial widths.
At the least permitted \(e\), one of the two monotone inequalities in Proposition~\ref{sc:bundle-tests} fails.
The same comparison using the normal bundle excludes the two \(G(3,1,7)\) rows and the \(G(4,1,7)\) row, and lifts of the entire groups exclude the corresponding cusps.
The absence of a common fixed point again excludes only \((1,1,4^6)\) when a point of its torsor fixed by all of \(G(4,2,7)\) is required.
In the eight cases of this comparison over the Gaussian field, \(R=\sigma=0\) and the generators of the rings of invariant sections have degrees one, two, and four.
Hence the width \(\kappa\) is the least degree of a generator that does not vanish at the point of the end, and the end is immersed in \(Q=q^\kappa\).
Seven cases have \(j=2,s=3\) and one cusp of width two.
After the double cover, the upper bounds on the degrees of the normal targets \((12,9,8,4,0,-4,-8)\) have sum at most 29 in rank 3, but the bundle \(\mathcal K\) of Proposition~\ref{sg:kernel} requires 36.
The last case \(j=4,s=4,a=9,\ell=5,e\ge9\) has width one and requires 43 against the same bound.
After the bound \(a\le10t\) has been established, Proposition~\ref{sh:eleven} gives explicitly the arguments for \(D_7\) and for \(E_7\) without points with a larger irreducible stabilizer.

\section*{Acknowledgments}
The first author was supported by JST SPRING, Grant Number \mbox{JPMJSP2119}.
The second author was supported by JSPS KAKENHI Grant Number \mbox{JP26KJ0476}.

\printbibliography

@article{ALA1,
  author = {Auffarth, Robert and Lucchini Arteche, Giancarlo},
  title = {Smooth quotients of abelian varieties by finite groups},
  journal = {Ann. Sc. Norm. Super. Pisa Cl. Sci. (5)},
  series = {5},
  volume = {21},
  number = {1--2},
  year = {2020},
  pages = {673--694},
  doi = {10.2422/2036-2145.201812_008},
  eprint = {1801.00028},
  eprinttype = {arxiv},
  eprintclass = {math.AG},
  url = {https://arxiv.org/abs/1801.00028}
}

@article{ALA2,
  author = {Auffarth, Robert and Lucchini Arteche, Giancarlo},
  title = {Smooth quotients of complex tori by finite groups},
  journal = {Math. Z.},
  volume = {300},
  number = {2},
  year = {2022},
  pages = {1071--1091},
  doi = {10.1007/s00209-021-02826-3},
  eprint = {1912.05327},
  eprinttype = {arxiv},
  eprintclass = {math.AG},
  url = {https://arxiv.org/abs/1912.05327v6}
}

@article{ALAQ,
  author = {Auffarth, Robert and Lucchini Arteche, Giancarlo and Quezada, Pablo},
  title = {Smooth quotients of abelian surfaces by finite groups that fix the origin},
  journal = {Cubo},
  volume = {24},
  number = {1},
  year = {2022},
  pages = {37--51},
  doi = {10.4067/S0719-06462022000100037},
  eprint = {1809.05405},
  eprinttype = {arxiv},
  eprintclass = {math.AG},
  url = {https://arxiv.org/abs/1809.05405v2}
}

@article{BailyBorel,
  author = {Baily, Jr., Walter L. and Borel, Armand},
  title = {Compactification of arithmetic quotients of bounded symmetric domains},
  journal = {Ann. of Math. (2)},
  series = {2},
  volume = {84},
  number = {3},
  year = {1966},
  pages = {442--528},
  doi = {10.2307/1970457},
  url = {https://annals.math.princeton.edu/1966/84-3/p11}
}

@article{BorelDensity,
  author = {Borel, Armand},
  title = {Density properties for certain subgroups of semi-simple groups without compact components},
  journal = {Ann. of Math. (2)},
  series = {2},
  volume = {72},
  number = {1},
  year = {1960},
  pages = {179--188},
  doi = {10.2307/1970150},
  url = {https://www.jstor.org/stable/1970150}
}

@article{BorelHarishChandra,
  author = {Borel, Armand and Harish-Chandra},
  title = {Arithmetic subgroups of algebraic groups},
  journal = {Ann. of Math. (2)},
  series = {2},
  volume = {75},
  number = {3},
  year = {1962},
  pages = {485--535},
  url = {https://www.mathi.uni-heidelberg.de/~wienhard/retreat10/references/borel\_harishchandra\_2.pdf}
}

@misc{BrownHiggins,
  author = {Brown, Ronald and Higgins, Philip J.},
  title = {The fundamental groupoid of the quotient of a {Hausdorff} space by a discontinuous action of a discrete group is the orbit groupoid of the induced action},
  year = {2002},
  howpublished = {arXiv:math/0212271v1},
  eprint = {math/0212271},
  eprinttype = {arxiv},
  eprintclass = {math.AT},
  url = {https://arxiv.org/abs/math/0212271v1}
}

@article{DM,
  author = {Deligne, Pierre and Mostow, George D.},
  title = {Monodromy of hypergeometric functions and non-lattice integral monodromy},
  journal = {Publ. Math. Inst. Hautes {\'E}tudes Sci.},
  volume = {63},
  year = {1986},
  pages = {5--89},
  url = {https://www.numdam.org/item/PMIHES\_1986\_\_63\_\_5\_0/}
}

@incollection{GeckMalle,
  author = {Geck, Meinolf and Malle, Gunter},
  title = {Reflection groups},
  booktitle = {Handbook of Algebra},
  editor = {Hazewinkel, Michiel},
  volume = {4},
  publisher = {Elsevier},
  year = {2006},
  pages = {337--383},
  eprint = {math/0311012},
  eprinttype = {arxiv},
  url = {https://arxiv.org/abs/math/0311012v2},
  address = {Amsterdam}
}

@article{HKM,
  author = {Hulek, Klaus and Kond{\=o}, Shigeyuki and Maeda, Yota},
  title = {Compactifications of the ancestral {Eisenstein} {Deligne--Mostow} variety},
  journal = {J. Algebraic Geom.},
  year = {2026},
  doi = {10.1090/jag/867},
  eprint = {2403.18345},
  eprinttype = {arxiv},
  eprintclass = {math.AG},
  url = {https://arxiv.org/abs/2403.18345v3},
  note = {Published online July 2, 2026}
}

@article{HMSV,
  author = {Howard, Benjamin and Millson, John and Snowden, Andrew and Vakil, Ravi},
  title = {The equations for the moduli space of {$n$} points on the line},
  journal = {Duke Math. J.},
  volume = {146},
  number = {2},
  year = {2009},
  pages = {175--226},
  doi = {10.1215/00127094-2008-063},
  url = {https://math.stanford.edu/~vakil/files/hmsv030708.pdf}
}

@article{Hu,
  author = {Hu, Hongsheng},
  title = {On exterior powers of reflection representations, {II}},
  journal = {Canad. J. Math.},
  year = {2025},
  doi = {10.4153/S0008414X25000215},
  eprint = {2401.08215},
  eprinttype = {arxiv},
  eprintclass = {math.RT},
  url = {https://arxiv.org/abs/2401.08215v2},
  note = {Published online March 24, 2025}
}

@book{Kobayashi,
  author = {Kobayashi, Shoshichi},
  title = {Differential Geometry of Complex Vector Bundles},
  series = {Publications of the Mathematical Society of Japan},
  number = {15},
  publisher = {Princeton University Press},
  year = {1987}
}

@article{KrishnasamyTaylor,
  author = {Krishnasamy, Muraleedaran and Taylor, D. E.},
  title = {Normalisers of parabolic subgroups in finite unitary reflection groups},
  journal = {J. Algebra},
  volume = {504},
  year = {2018},
  pages = {479--505},
  doi = {10.1016/j.jalgebra.2018.02.026},
  eprint = {1712.09563},
  eprinttype = {arxiv},
  url = {https://arxiv.org/abs/1712.09563}
}

@article{Looijenga,
  author = {Looijenga, Eduard},
  title = {Compactifications defined by arrangements, {I}: The ball quotient case},
  journal = {Duke Math. J.},
  volume = {118},
  number = {1},
  year = {2003},
  pages = {151--187},
  doi = {10.1215/S0012-7094-03-11816-5},
  eprint = {math/0106228},
  eprinttype = {arxiv},
  url = {https://arxiv.org/abs/math/0106228v2}
}

@article{Lozovanu,
  author = {Lozovanu, Victor},
  title = {Multiplicities of irreducible theta divisors},
  journal = {J. Differential Geom.},
  volume = {128},
  number = {3},
  year = {2024},
  pages = {1149--1176},
  doi = {10.4310/jdg/1729092456},
  eprint = {2002.04360},
  eprinttype = {arxiv},
  url = {https://arxiv.org/abs/2002.04360v3},
  note = {Numbered references are to arXiv:2002.04360v3}
}

@book{Meyer,
  author = {Serre, Jean-Pierre},
  title = {A Course in Arithmetic},
  series = {Graduate Texts in Mathematics},
  number = {7},
  publisher = {Springer-Verlag},
  address = {New York},
  year = {1973},
  doi = {10.1007/978-1-4684-9884-4},
  url = {https://www.math.purdue.edu/~jlipman/MA598/Serre-Course\%20in\%20Arithmetic.pdf}
}

@incollection{MokCompactification,
  author = {Mok, Ngaiming},
  title = {Projective algebraicity of minimal compactifications of complex-hyperbolic space forms of finite volume},
  booktitle = {Perspectives in Analysis, Geometry, and Topology},
  editor = {Itenberg, Ilia and J{\"o}ricke, Burglind and Passare, Mikael},
  series = {Progress in Mathematics},
  volume = {296},
  publisher = {Birkh{\"a}user},
  year = {2012},
  pages = {331--354},
  doi = {10.1007/978-0-8176-8277-4_14},
  url = {https://www-fourier.univ-grenoble-alpes.fr/~demailly/source\_files/envelope/viro\_volume/chap14.pdf},
  address = {Boston, MA},
  note = {Numbered references are to the 24-page uncorrected proof}
}

@misc{NakanoSevenConfigurations2026,
  author = {Nakano, Ryunosuke and Matsukawa, Hisato},
  title = {Seven ordered points and the nonexistence of arithmetic ball quotient compactifications},
  year = {2026},
  date = {2026-10-16},
  eprint = {2610.23227},
  eprinttype = {arxiv},
  eprintclass = {math.AG},
  note = {arXiv:2610.23227v1},
  url = {https://arxiv.org/abs/2610.23227v1}
}

@article{Ohmoto,
  author = {Ohmoto, Toru},
  title = {Equivariant {Chern} classes of singular algebraic varieties with group actions},
  journal = {Math. Proc. Cambridge Philos. Soc.},
  volume = {140},
  year = {2006},
  pages = {115--134},
  doi = {10.1017/S0305004105008820},
  url = {https://eprints.lib.hokudai.ac.jp/dspace/bitstream/2115/5513/1/MPCPS140.pdf}
}

@misc{Rains,
  author = {Rains, Eric M.},
  title = {Quotients of abelian varieties by reflection groups},
  year = {2024},
  howpublished = {arXiv:2303.04786v2},
  eprint = {2303.04786},
  eprinttype = {arxiv},
  eprintclass = {math.AG},
  url = {https://arxiv.org/abs/2303.04786v2},
  note = {Version 2}
}

@article{Stover,
  author = {Stover, Matthew},
  title = {Arithmeticity of complex hyperbolic triangle groups},
  journal = {Pacific J. Math.},
  volume = {257},
  number = {1},
  year = {2012},
  pages = {243--256},
  doi = {10.2140/pjm.2012.257.243},
  url = {https://msp.org/pjm/2012/257-1/pjm-v257-n1-p14-p.pdf}
}

@article{WangJacobi,
  author = {Wang, Haowu},
  title = {Weyl invariant {Jacobi} forms: a new approach},
  journal = {Adv. Math.},
  volume = {384},
  year = {2021},
  pages = {107752},
  eid = {107752},
  doi = {10.1016/j.aim.2021.107752},
  eprint = {2007.16033},
  eprinttype = {arxiv},
  url = {https://arxiv.org/abs/2007.16033v2}
}

@article{WeakFactorization,
  author = {Abramovich, Dan and Karu, Kalle and Matsuki, Kenji and W{\l}odarczyk, Jaros{\l}aw},
  title = {Torification and factorization of birational maps},
  journal = {J. Amer. Math. Soc.},
  volume = {15},
  number = {3},
  year = {2002},
  pages = {531--572},
  eprint = {math/9904135},
  eprinttype = {arxiv},
  url = {https://arxiv.org/abs/math/9904135v4}
}

@article{Wirthmuller,
  author = {Wirthm{\"u}ller, Klaus},
  title = {Root systems and {Jacobi} forms},
  journal = {Compos. Math.},
  volume = {82},
  number = {3},
  year = {1992},
  pages = {293--354},
  url = {https://www.numdam.org/item/CM\_1992\_\_82\_3\_293\_0/}
}

@inproceedings{Brieskorn1971,
  author = {Brieskorn, Egbert},
  title = {Singular elements of semi-simple algebraic groups},
  booktitle = {Actes du Congr\`es International des Math\'ematiciens (Nice, 1970), Tome 2},
  publisher = {Gauthier-Villars},
  address = {Paris},
  year = {1971},
  pages = {279--284},
  url = {https://www.mathunion.org/fileadmin/ICM/Proceedings/ICM1970.2/ICM1970.2.ocr.pdf}
}

@article{GrangerMondSchulze2012,
  author = {Granger, Michel and Mond, David and Schulze, Mathias},
  title = {Partial normalizations of {Coxeter} arrangements and discriminants},
  journal = {Mosc. Math. J.},
  volume = {12},
  number = {2},
  year = {2012},
  pages = {335--367},
  doi = {10.17323/1609-4514-2012-12-2-335-367},
  eprint = {1108.0718},
  eprinttype = {arxiv},
  note = {Numbered references are to arXiv:1108.0718v3, amended postprint}
}

@article{KatoManoSekiguchi2015,
  author = {Kato, Mitsuo and Mano, Toshiyuki and Sekiguchi, Jiro},
  title = {Flat structures without potentials},
  journal = {Rev. Roumaine Math. Pures Appl.},
  volume = {60},
  number = {4},
  year = {2015},
  pages = {481--505},
  url = {https://imar.ro/journals/Revue_Mathematique/pdfs/2015/4/6.pdf}
}

@article{KatoManoSekiguchi2020,
  author = {Kato, Mitsuo and Mano, Toshiyuki and Sekiguchi, Jiro},
  title = {Flat structure on the space of isomonodromic deformations},
  journal = {SIGMA Symmetry Integrability Geom. Methods Appl.},
  volume = {16},
  year = {2020},
  eid = {110},
  pagetotal = {36},
  doi = {10.3842/SIGMA.2020.110},
  url = {https://sigma-journal.com/2020/110/sigma20-110.pdf}
}

@article{SekiguchiE7ST34,
  author = {Sekiguchi, Jiro},
  title = {Simple singularity of type {$E_7$} and the complex reflection group {ST34}},
  journal = {Indag. Math. (N.S.)},
  volume = {36},
  number = {2},
  year = {2025},
  pages = {567--592},
  doi = {10.1016/j.indag.2024.07.008},
  eprint = {2311.16629},
  eprinttype = {arxiv},
  note = {Numbered references are to arXiv:2311.16629v1}
}

@article{Bessis2015,
  author = {Bessis, David},
  title = {Finite complex reflection arrangements are {$K(\pi,1)$}},
  journal = {Ann. of Math. (2)},
  volume = {181},
  number = {3},
  year = {2015},
  pages = {809--904},
  doi = {10.4007/annals.2015.181.3.1},
  eprint = {math/0610777},
  eprinttype = {arxiv},
  note = {Numbered references are to arXiv:math/0610777v5}
}

@article{HulekMaedaEight,
  author = {Hulek, Klaus and Maeda, Yota},
  title = {Revisiting the moduli space of 8 points on {$\mathbb P^1$}},
  journal = {Adv. Math.},
  volume = {463},
  year = {2025},
  pages = {110126},
  eid = {110126},
  doi = {10.1016/j.aim.2025.110126},
  eprint = {2211.00052},
  eprinttype = {arxiv},
  eprintclass = {math.AG},
  url = {https://arxiv.org/abs/2211.00052v3}
}

\bigskip
\noindent
\begin{minipage}{\linewidth}
Department of Mathematics, Graduate School of Science, Hokkaido University, Sapporo 060-0810, Japan.

\smallskip
Ryunosuke Nakano: \href{mailto:nakano.ryunosuke.i3@elms.hokudai.ac.jp}{\texttt{nakano.ryunosuke.i3@elms.hokudai.ac.jp}}.

\smallskip
Hisato Matsukawa: \href{mailto:matsukawa.hisato.f4@elms.hokudai.ac.jp}{\texttt{matsukawa.hisato.f4@elms.hokudai.ac.jp}}.
\end{minipage}

\end{document}